\documentclass[11pt,reqno]{amsart}
\usepackage{a4wide}
\numberwithin{equation}{section}
\usepackage{mathrsfs}
\usepackage{amsfonts}
\usepackage{amsmath}
\usepackage{stmaryrd}
\usepackage{amssymb}
\usepackage{amsthm}
\usepackage{mathrsfs}
\usepackage{url}
\usepackage{amsfonts}
\usepackage{amscd}
\usepackage{indentfirst}
\usepackage{enumerate}
\usepackage{amsmath,amsfonts,amssymb,amsthm}
\usepackage{amsmath,amssymb,amsthm,amscd}
\usepackage{graphicx,mathrsfs}
\usepackage{appendix}

\usepackage[numbers,sort&compress]{natbib}
\usepackage{color}
 \usepackage[colorlinks, linkcolor=blue, citecolor=blue]{hyperref}

\newcommand\N{\mathbb N}

\newcommand\R{\mathbb R}

\newcommand\mbb\mathbb
\newcommand\mbf\mathbf
\newcommand\mcal\mathcal
\newcommand\mfrak\mathfrak
\newcommand\mrm\mathrm
\newcommand\msf\mathsf
\renewcommand\a\alpha
\renewcommand\b\beta
\newcommand\g\gamma
\newcommand\G\Gamma
\renewcommand\d\delta
\newcommand\D\Delta
\newcommand\e\varepsilon
\newcommand\z\zeta
\renewcommand\t\theta
\newcommand\Th\Theta
\newcommand\la\lambda
\newcommand\La\Lambda
\newcommand\s\sigma
\newcommand\si\varsigma
\newcommand\Si\Sigma
\newcommand\ups\upsilon
\newcommand\U\Upsilon
\newcommand\ph\varphi
\renewcommand\o\omega
\renewcommand\O\Omega
\newcommand\wt\widetilde
\newcommand\wh\widehat
\newcommand\ol\overline
\newcommand\ul\underline
\newcommand\mr\mathring
\newcommand\ub\underbrace
\newcommand\pa\partial
\newcommand\n\nabla
\newcommand\fa\forall
\newcommand\ex\exists
\newcommand\es\emptyset
\newcommand\wk\rightharpoonup
\newcommand\inc\hookrightarrow
\newcommand\linf\varliminf
\newcommand\lsup\varlimsup
\newcommand\os\overset
\newcommand\us\underset
\newcommand\sr\stackrel
\newcommand\Ot\Leftarrow
\newcommand\To\Rightarrow
\newcommand\map\mapsto
\newcommand\ot\leftarrow
\newcommand\lot\longleftarrow
\newcommand\lto\longrightarrow
\newcommand\tot\leftrightarrow
\newcommand\ltot\longleftrightarrow
\newcommand\sm\backslash
\renewcommand\Cup\bigcup
\renewcommand\Cap\bigcap
\newcommand\sub\subset
\newcommand\Sub\Subset
\newcommand\sne\subsetneq
\newcommand\bus\supset
\newcommand\Bus\Supset

\newcommand\eq\equiv
\newcommand\ox\otimes
\newcommand\Ox\bigotimes
\newcommand\pl\oplus
\newcommand\Pl\bigoplus
\newcommand\x\times
\renewcommand\c\circ
\newcommand\q\quad
\renewcommand\l\left
\renewcommand\r\right
\newcommand\fr\frac

\usepackage[numbers,sort&compress]{natbib}
\usepackage[dvipsnames]{xcolor}
\usepackage{color}

\definecolor{bondiblue}{rgb}{0.0, 0.58, 0.71}
\def\sideremark#1{\ifvmode\leavevmode\fi\vadjust{\vbox to0pt{\vss
			\hbox to 0pt{\hskip\hsize\hskip1em
				\vbox{\hsize2.1cm\tiny\raggedright\pretolerance10000
					\noindent #1\hfill}\hss}\vbox to15pt{\vfil}\vss}}}%

\newtheorem{Thm}{Theorem}[section]
\newtheorem{Lem}[Thm]{Lemma}

\newtheorem{Prop}[Thm]{Proposition}

\newtheorem{Rem}[Thm]{Remark}

\begin{document}

\title[Lane-Emden problem]
{Morse index, topological degree and
 local uniqueness of multi-spikes solutions to the Lane-Emden problem in dimension two}
\author[I. Ianni, P. Luo and S. Yan]{Isabella Ianni, Peng Luo, Shusen Yan}


 \address[Isabella Ianni]{Dipartimento di Scienze di Base e Applicate per l'Ingegneria, Sapienza
 Universit$\grave{a}$di Roma, Via Scarpa 16, 00161 Roma, Italy}
 \email{isabella.ianni@uniroma1.it}

  \address[Peng Luo]{School of Mathematics and Statistics, Key Laboratory of Nonlinear Analysis and Applications
(Ministry of Education), and Hubei Key Laboratory of Mathematical
Sciences,
Central China Normal University, Wuhan 430079, China}
  \email{pluo@ccnu.edu.cn}

 \address[Shusen Yan]{School of Mathematics and Statistics, Key Laboratory of Nonlinear Analysis and Applications
(Ministry of Education), Central China Normal University, Wuhan 430079, China}
\email{syan@ccnu.edu.cn}

\begin{abstract}
We consider {\sl multi-spike} positive solutions to  the  Lane-Emden problem in any  bounded smooth planar domain and compute  their Morse index, extending to the dimension $N=2$ classical theorems due to Bahri-Li-Rey (1995) \cite{BLR95} and Rey (1999) \cite{Rey} when $N\geq 4$ and $N=3$, respectively. Furthermore, by deeply investigating their concentration behavior, we also derive the total topological degree.
The Morse index and the degree counting formula yield a new local uniqueness result.
\end{abstract}
\maketitle
{\small

\section{Introduction and main results}

\setcounter{equation}{0}

We consider  the Lane-Emden problem
\begin{equation}\label{1.1}
\begin{cases}
-\Delta u=u^{p}  &\text{in}~\Omega,\\[0.5mm]
u>0  &\text{in}~\Omega,\\[0.5mm]
u=0 &\text{on}~\partial \Omega,
\end{cases}
\end{equation}
where $\Omega\subset \R^2$ is a smooth bounded domain and $p>1$.
\vskip 0.1cm

It is well known that problem \eqref{1.1} admits at least one solution, while the question of  uniqueness or multiplicity is significantly
more complex,  depending on both the topology and geometry of the domain $\Omega$, as well as on the value of the exponent $p$ (see for instance \cite{Dancer88}). A complete understanding of this issue is still far from being reached both in dimension $N=2$ than in the higher dimensions.
For example, the solution is unique in any domain $\Omega$ provided that $p$ is sufficiently close  to $1$ (\cite{DamascelliGrossiPacella, DancerMA2003,  L94}). Moreover, for any value of $p$, uniqueness is known to hold when $\Omega$ is a ball (\cite{GNN}) or, more generally, when it is a symmetric and convex domain  with respect to $2$ orthogonal directions (\cite{Dancer88, DamascelliGrossiPacella}). More recently, uniqueness for problem \eqref{1.1} has also been established in general convex domains, without any symmetry assumption, for sufficiently large $p$ (see \cite{DGIP2019,GILY2021}), thus partially solving a longstanding conjecture proposed in \cite{Dancer88}).
Conversely, multiple solutions exist in suitable non-convex domains $\Omega$ (see \cite{Dancer88, EMP2006} for examples involving dumbbell shaped domains,  or \cite{EMP2006,BCGP} for the case of non-simply connected domains).
In particular all the solutions found in \cite{EMP2006} concentrate at a finite number of distinct points in $\Omega$, and vanish elsewhere as $p\rightarrow +\infty$. Note that all solutions of problem \eqref{1.1} are uniformly bounded in $p$ (see \cite{KS2018}). Notably, for the concentrating solutions built in  \cite{EMP2006}, the local maximum values around each concentration point converge all to $\sqrt e$.
The location of these concentration points is determined by  the Green function $G$ and  the Robin function $R$ associated with  $-\Delta$ in $\Omega$ under Dirichlet boundary conditions. Especially, in \cite{EMP2006},  the existence of a solution concentrating at $k\in\mathbb N$ distinct points $x_{\infty,1},\cdots, x_{\infty,k}\in\Omega$ is obtained whenever
$\boldsymbol{x_{\infty}} :=(x_{\infty,1},\cdots, x_{\infty,k})$
is a non-degenerate critical point of the  Kirchhoff-Routh function  $\Psi_{k}: \Omega^{k} \rightarrow \R$, which is defined as
\begin{equation}\label{stts}
\Psi_{k}(\boldsymbol{a}):= \sum^k_{j=1} \Psi_{k,j}(\boldsymbol{a}),\quad~\mbox{ with }~\Psi_{k,j}(\boldsymbol{a}):=  R\big(a_j\big)- \sum^{k}_{m=1,m\neq j} G\big(a_j,a_m\big),
\end{equation}
for $\boldsymbol{a}=(a_1,\cdots,a_k)$, with $a_j\in \Omega,$ $j=1,\cdots,k$.

\vskip 0.1cm
Recently, a rather complete characterization of the asymptotic behavior as $p\rightarrow +\infty$ of \emph{any} family $u_p$ of solutions to problem \eqref{1.1} has been obtained in \cite{DIP2017-1}, with sharp constants later provided in \cite{DGIP2018,T2019}, under the following  energy uniform bound assumption:
	\begin{equation*}
	\label{bound}
	\sup_{p}\ p \|\nabla u_p\|^2_{2}\leq C.
	\end{equation*}
It was shown that solutions $u_{p}$ necessarily concentrate, up to a subsequence, at a critical point $(x_{\infty,1},\cdots, x_{\infty,k})$ of $\Psi_{k}$, for some $k\in\mathbb N$, and vanish elsewhere, in the same way as the solutions found in \cite{EMP2006}. Moreover, their local maximum values converge to $\sqrt{e}$, as $p\rightarrow +\infty$, and the total energy is asymptotically quantized as an integer multiple of $8\pi e$:
		\[\lim_{p\rightarrow +\infty}p\|\nabla u_{p}\|^2_2=k\cdot 8\pi e.\]
Furthermore, the asymptotic local profile of the solutions (suitably rescaled around each concentration point) was also described, showing that it  is given by a radial solution  $U$ of the Liouville equation $-\Delta U = e^U$ in the whole $\mathbb R^2$.
\vskip 0.1cm
We refer to solutions having these properties as \emph{spike solutions}.
It is worth noting  that  least energy solutions of problem \eqref{1.1} are single-spike solutions (i.e. $k=1$). For previous characterizations of this simpler case we refer to  \cite{RenWei1, RenWei2, AdimurthiGrossi}. Observe that when $k=1$, the Kirchhoff-Routh  function $\Psi_{1}$ reduces to the Robin function $R$. Furthermore, by \cite{GrossiTakahashi}, when $\Omega$ is convex it must be $k=1$ for any spike-solution.

 \vskip 0.1cm

The sharp asymptotic  description provided in \cite{DIP2017-1} plays a crucial role in the proof of the  uniqueness result for convex domains developed in \cite{DGIP2019}. Indeed, a key ingredient to derive uniqueness in \cite{DGIP2019} is the computation of the \emph{Morse index} for single-spike solutions ($k=1$). A central observation is that for the \emph{Morse index} $m(u_p)$ (resp. the \emph{augmented Morse index} $m_0(u_p)$) of a solution $u_p$ to problem \eqref{1.1}, the following holds:
\begin{equation}\label{link}
m(u_p)=\sharp \big\{m\in \N: \lambda_{p,m}<1\big\}\quad (\text{resp. } m_0(u_p)=\sharp \big\{m\in \N: \lambda_{p,m}\leq 1\big\}),
\end{equation}
where $\lambda_{p,1}<\lambda_{p,2}\leq \lambda_{p,3}\leq \cdots$ denotes the sequence of eigenvalues for the  problem
\begin{equation}\label{03-16-1}
\begin{cases}
-\Delta v=\lambda p(u_p)^{p-1}v~&\mbox{in}~\Omega,\\
v=0~&\mbox{on}~\partial \Omega.
\end{cases}
\end{equation}
Thus, it becomes evident that the asymptotic information about the eigenvalues $\lambda_{p,j}$ can be derived from the asymptotic information on $u_p$, leading to the computation of $m(u_p)$. For $k$-spike solutions, this computation has been rigorously  carried out  in  \cite{DGIP2019} in the case $k=1$, while when $k\geq 2$, determining the exact  Morse index is still an open question, whose answer requires a deeper understanding of the asymptotic behavior of $u_p$.

\vskip 0.1cm

More recently, the characterization  in \cite{DIP2017-1}  has been refined in \cite{GILY2021}, where higher order asymptotic expansions have been established (see  Section \ref{subsection:asympt} for more details).

\vskip 0.1cm

One of the aims of this paper is  to compute the Morse index of any \emph{multi-spike} solution of problem \eqref{1.1}, making use of this sharpened analysis.

\vskip 0.1cm

Furthermore, the improved characterization in \cite{GILY2021} enabled the proofs in \cite{GILY2021}
of the nondegeneracy of multi-spike solutions, and of a local uniqueness result for single-spike solutions.
Here, \emph{local uniqueness} means that if $u$ and $\widetilde u$ are two solutions of problem \eqref{1.1} concentrating at the same points, then $u\equiv\widetilde u$, for  sufficiently large $p$. This property is particularly significant in non-convex domains, where  uniqueness \emph{tout-court} is not expected for problem \eqref{1.1}, and the focus instead shifts to counting the number of solutions.  However, for $k$-spike solutions, with $k\geq 2$,
local uniqueness remains an open problem. Indeed, dealing with multi-spike solutions is considerably more intricate, as the interactions between the spikes must be carefully taken into account. In this work, we also explore the local uniqueness property for multi-spike solutions, which, as we will see, is closely linked to the computation of the Morse index.

\vskip 0.1cm

Our first result concerns the computation of the Morse index for any \emph{multi-spike} solution of problem \eqref{1.1}.
To state it precisely, let us denote by $m(\boldsymbol{a},f)$
and $m_0(\boldsymbol{a},f)$ the \emph{Morse and augmented Morse index}, respectively,  of $\boldsymbol{a}=(a_1,\cdots,a_k)\in\Omega^k$ as a critical point of a $C^2$ function $f:\Omega^k\rightarrow \mathbb R$, that is:
\begin{equation}\label{def:MorseOfAPointf}
\begin{array}{lr}
m(\boldsymbol{a},f):=\sharp \Big\{l\in \{1,2,\cdots,2k\}: \theta_l<0 \Big\},\,\,\, \,\,\,
m_0(\boldsymbol{a},f):=\sharp \Big\{l\in \{1,2,\cdots,2k\}: \theta_l\leq 0 \Big\},
\end{array}
\end{equation}
where $\theta_1\leq \theta_2\cdots\leq \theta_{2k}$ are the eigenvalues of the Hessian matrix $D^2 f(\boldsymbol{a})$.
Moreover, we say that the critical point $\boldsymbol{a}$ of $f$ is \emph{nondegenerate} if the Hessian matrix $D^2 f(\boldsymbol{a})$ is nondegenerate.
We have:
\begin{Thm}\label{th1.1}
Let $k\in\mathbb N$, $u_p$ be a $k$-spike solution of problem \eqref{1.1}, and let  $x_{\infty,1},\cdots,x_{\infty,k}\in\Omega$ be its $k$ distinct concentration points.
Then there exists $\hat p>1$ such that  the following holds:
\begin{equation*}
k\leq k+m(\boldsymbol{\mathrm x}_{\infty}, \Psi_{k})\leq m(u_p)\leq m_0(u_p)\leq k+m_0(\boldsymbol{\mathrm x}_{\infty}, \Psi_{k})\leq 3k,\,\,\,\,\mbox{ for }~~p\geq\hat p,
\end{equation*}
where $\Psi_{k}$ is the function in \eqref{stts} and $\boldsymbol{\mathrm x}_{\infty}:=\big(x_{\infty,1},\cdots,x_{\infty,k}\big)$ is a critical point of $\Psi_{k}$. Furthermore, if $\boldsymbol{\mathrm x}_{\infty}$ is nondegenerate, then ($u_p$ is non-degenerate and)
\begin{equation*}
m(u_p)=k+m(\boldsymbol{\mathrm x}_{\infty},\Psi_{k})\in [k,3k],\,\,\,\,\mbox{ for }~~p\geq\hat p.
\end{equation*}
\end{Thm}

\vskip 0.1cm

Theorem \ref{th1.1} extends to dimension $N=2$ classical results on Morse index computation
for  \emph{blowing-up} positive solutions of the higher dimensional Lane-Emden equation, originally established  by  Bahri-Li-Rey for $N\geq 4$ (\cite{BLR95}) and by Rey for $N=3$  (\cite{Rey}). A different, unified proof valid for any dimension  $N\geq 3$ was later provided by Grossi-Pacella (\cite{GP2005}) for {\sl single blowing-up} solutions, and by Choi-Kim-Lee (\cite{CKL16}) for {\sl multiple blowing-up} solutions. This new approach relies on analyzing the  linearized eigenvalue problem \eqref{03-16-1}  at a blowing-up solution $u_{p}$, as $p\rightarrow {(N+2)/(N-2)}^{-}$, leveraging the connection in \eqref{link}. Similar to \cite{GP2005,CKL16}, our proof  is based on obtaining asymptotic estimates for the eigenvalues of problem \eqref{03-16-1}, for $p$ sufficiently large. This is achieved by fully exploiting the sharp asymptotic characterization  of  multi-spike solutions $u_{p}$ established in  \cite{DIP2017-1,DGIP2018}, along with its significant refinements in \cite{GILY2021}. However, the similarity between the two cases ends here. Indeed, the asymptotic behavior of multi-spike solutions of problem \eqref{1.1} as $p\rightarrow +\infty$ is fundamentally different from that of blowing-up solutions  as $p\rightarrow {(N+2)/(N-2)}^{-}$ in the higher dimensional case $N\geq 3$. For instance, unlike blowing-up solutions, multi-spike solutions remain uniformly bounded.
In particular, we emphasize that  in the proof of Theorem \ref{th1.1},  higher order expansions are essential, due to the fact that, as $p\rightarrow +\infty$, the pure power nonlinearity $u^p$ in the equation \eqref{1.1} is not close enough to the exponential nonlinearity $e^U$ which appears in the \emph{limit problem}.  Consequently, the computations become quite intricate, making the proof particularly challenging from a technical standpoint. We stress that, when $\boldsymbol{\mathrm x}_{\infty}$ is non-degenerate, we also re-obtain the nondegeneracy of the multi-spike solution $u_p$, which was already known from \cite{GILY2021}. We recall that $u_p$ is nondegenerate if and only if $\lambda=1$ is not an eigenvalue for problem \eqref{03-16-1}.

\vskip 0.1cm

Our next result is a degree counting formula for the $k$-spike solutions of problem \eqref{1.1}.
\begin{Thm} \label{theorem:topological-degree-formula}
Let $k\in\mathbb N$ and $\boldsymbol{\mathrm x}_{\infty}:=(x_{\infty,1},\cdots,x_{\infty,k})\in \Omega^{k}$, $x_{\infty,i}\neq x_{\infty,j}$ if $i\neq j$ be
a critical point of the Kirchhoff-Routh function $\Psi_{k}$ (see \eqref{stts} for its definition).
If $\boldsymbol{\mathrm x}_{\infty}$ is nondegenerate then the total Leray-Schauder degree of all  $k$-spike solutions of problem \eqref{1.1} concentrating at $\boldsymbol{\mathrm x}_{\infty}$ is given by
\[(-1)^{k+m(\boldsymbol{\mathrm x}_{\infty},\Psi_{k})},\]
where
$m(\boldsymbol{\mathrm x}_{\infty},\Psi_{k})$ is the Morse index of $\Psi_{k}$ at the critical point $\boldsymbol{\mathrm x}_{\infty}$, as defined in \eqref{def:MorseOfAPointf}.
\end{Thm}

Last, as a consequence of Theorem \ref{th1.1} and of Theorem \ref{theorem:topological-degree-formula}, we derive the following local uniqueness result for multi-spike solutions of problem \eqref{1.1}:
\begin{Thm}
 \label{theorem:local-uniqueness}
 Let $k\in\mathbb N$, and  $\boldsymbol{\mathrm x}_{\infty}:=(x_{\infty,1},\cdots,x_{\infty,k})\in \Omega^{k}$, $x_{\infty,i}\neq x_{\infty,j}$ if $i\neq j$ be
a critical point of the Kirchhoff-Routh function $\Psi_{k}$ (see \eqref{stts} for the definition).
If $\boldsymbol{\mathrm x}_{\infty}$ is nondegenerate then the $k$-spike solution of \eqref{1.1}  concentrating at  $\boldsymbol{\mathrm x}_{\infty}$ is unique (and nondegenerate) for $p$ large.
\end{Thm}

The proof of Theorem  \ref{theorem:topological-degree-formula} consists of reducing the problem  to a finite-dimensional one and then applying classical properties of degree theory to derive the conclusion, as in \cite{cl1,cl2} for the mean field equation.
To achieve this goal, the crucial step is to find a suitable set $S_p$ (see \eqref{35-30-8}), which contains all the $k$-spike solutions with
concentration point $\boldsymbol{\mathrm x}_{\infty}$, so that the Leray-Schauder degree can be computed in $S_p$.
We observe that the construction of multiple concentrating solutions in \cite{EMP2006} relies on a Lyapunov-Schmidt finite dimensional reduction, which provides detailed information on their structure. In particular, these solutions can be expressed in terms of a finite number of parameters, $\xi_{p,j}\in\Omega$ and $\alpha_{p,j}>0$ for $j=1,\cdots,k$. Indeed, the solutions constructed in \cite{EMP2006}  are given by
\begin{equation}\label{forma} \sum_{j=1}^k  \alpha_{p,j} P\bar W_{p, \xi_{p,j}}+\omega_{p},\end{equation}
where  $\bar W_{p, \xi_{p,j}}$ represents  the sum of suitably rescaled profile functions centered at $\xi_{p,j}$ (see \eqref{12-30-8}), while  $\omega_p$ is a reminder term. Furthermore, as $p\rightarrow +\infty$, we have
$\xi_{p,j}\rightarrow x_{\infty,j}$, $\alpha_{p,j}\rightarrow 1$,  $\omega_{p}\rightarrow 0$ with explicit convergence rates.
Our strategy for the proof of Theorem \ref{theorem:topological-degree-formula} is  first to demonstrate that \emph{any} $k$-spike solution
 $u_{p}$ of problem \eqref{1.1} takes the form \eqref{forma} (see Proposition \ref{prop:upHasTheFormInEMP}). To establish this, we have
  to refine again the asymptotic analysis of $u_p$ as $p\rightarrow +\infty$, improving upon the results in \cite{GILY2021} (see Proposition \ref{prop:summaryImporvedAsympt}).
%
%
%

\vskip 0.1cm

The uniqueness result in Theorem \ref{theorem:local-uniqueness} is \emph{local} and in general does not
coincide with the uniqueness of solutions to problem \eqref{1.1}, since there can be many different critical points of the Kirchhoff-Routh function $\Psi_{k}$, and there can be also $k$-spike solutions for different $k$'s (as in the existence results in \cite{EMP2006}).

\vskip 0.1cm

Local uniqueness results were proved  either by computing  the Morse index
of the reduced finite dimensional function as in \cite{G100,CNY1,NY100},
or by using the Pohozaev identities \cite{GM100,CGPY2019,DLY100,GILY2021}.
Both methods fail in the study of local uniqueness of $k$-spike solutions
for problem \eqref{1.1} if $k\ge 2$. In fact,
when computing the Morse index
of the reduced finite dimensional function, one needs to
estimate the derivatives of the error term $\omega_p$ in \eqref{forma}.
Due to the exponential smallness of the scalar $\varepsilon_{p, j}$ (see
\eqref{nn3-29-03}), the derivatives of the error term $\omega_p$
is uncontrollable. On the other hand, the starting point to use the
Pohozaev identities is to prove that $x_{p, j}$ (see \eqref{def:xpj})
satisfies $|x_{p, j}-x_{\infty, j}|=o(\varepsilon_{p, j})$, which is
very difficult to verify if $k\ge 2$.
This is the main reason that in \cite{GILY2021}, local uniqueness result can
be proved only for the
 case $k=1$. The novelty here to prove the local uniqueness result for
    the cases $k\geq 2$ is that
    we compute the Morse index of each $k$-spike solution concentrating at $\bf x_\infty$, and the degree of all the $k$-spike solutions concentrating at $\bf x_\infty$. The first step can be achieved by using the local
    Pohozaev identities, while in the second step, it
    is essential to deform
    the corresponding functional suitably, so one can avoid
    the estimates of the derivatives of the error term $\omega_p$.

\vskip 0.1cm

 Some condition on the critical point $\boldsymbol{\mathrm x}_{\infty}$ is needed
 to obtain the local uniqueness result, as evidenced by the example given
 in \cite{GM100} for the singularly perturbed Schr\"odinger equations.


\vskip 0.1cm

We describe the organization of the paper and the structure of the proofs.
\begin{itemize}
\item In Section \ref{section:preliminary} we collect a number of results about the asymptotic characterization of the multi-spike solutions of problem \eqref{1.1} when $N=2$, collected from  \cite{DIP2017-1,DGIP2018,GILY2021}. We also introduce a limit eigenvalue problem and gather its main properties. Furthermore we recall some useful bilinear integral identities concerning the Green and Robin function and their derivatives, previously obtained in  \cite{GILY2021}. Also we derive local Pohozaev identities involving general solutions $u_{p}$ of problem \eqref{1.1} and general eigenfunctions $v_{p,l}$ of problem \eqref{03-16-1} in this section.
\vskip 0.1cm
\item
Section \ref{section:preliminaryEigen} contains starting asymptotic expansions for the eigenfunctions $v_{p,l}$ when their eigenvalues $\lambda_{p,l}$ converge either to $0$ or to $1$. Indeed, due to the properties of the limit eigenvalue problem, these turn out to be the two only relevant cases for computing the Morse index. The main result of the section is stated in Proposition \ref{lemma:autofunz2casi}. In order to get it, we scale  properly each eigenfunction around the concentration points of $u_{p}$, pass to the limit into the rescaled eigenvalue problem, thus deriving the asymptotics for the rescaled eigenfunctions from the knowledge of the limit eigenvalue problem (see Proposition \ref{prop3-3}). When we pass to the limit we  strongly exploit the asymptotic behavior of the multi-spike solutions $u_{p}$ (see Section \ref{subsection:asympt}), furthermore we  first have to exclude that all the rescaled eigenfunctions vanish in the limit (see Lemma \ref{prop3-3prima}).
\vskip 0.1cm
\item
Section \ref{section:Morse} contains the proof of Theorem \ref{th1.1}. In order to compute the Morse index of $u_p$ we first study the eigenpairs $(\lambda_{p,l},v_{p,l})$, for $l=1,\cdots,3k+1$, as $p\rightarrow +\infty$.
We shall first deal with the case $1\leq l\leq k$
 and estimate $\lambda_{p,l}$ (Proposition \ref{prop:ExpansionsFirstGroup}), then we treat the most important case $k+1\leq l\leq 3k$ and obtain
 and expansion for $(\lambda_{p,l},v_{p,l})$
 (Proposition \ref{propSecondoPezzo1} and Proposition \ref{propSecondoPezzo2}), and finally we estimate $\lambda_{p,3k+1}$ (Proposition \ref{prop:expansionLastGroup}).  Observe that the information on the previous pairs, due to the orthogonality condition,  is necessary to investigate the next pairs. In each case the proof follows the same scheme:\vskip 0.1cm
\begin{itemize}
\item we estimate the eigenvalue, using the variational characterization and choosing suitable test functions;\vskip 0.1cm
\item we characterize the limit of the rescaled eigenfunction, using the properties of the limit eigenvalue problem;\vskip 0.1cm
\item we use the bilinear local integral identities  (see Lemma \ref{lem2-1} and the Pohozaev identities in Lemma \ref{lemma:pohoz}) to refine the information on the eigenvalue and also deduce the expansion of the eigenfunction.
\end{itemize}\vskip 0.1cm
One of the main difficulties of our proof is due to the delicate interaction between the  different regions where the $k$-spike solution concentrates. In particular we need to control both the rescalings $w_{p,j}$ of the solution $u_{p}$  and the rescalings $\widetilde v_{p,l,j}$ of the eigenfunction $v_{p,l}$, this is essentially done in Lemma \ref{llma} and Lemma \ref{prop3-3prima} in the previous sections.\vskip 0.1cm
\item In Section \ref{Section:ImprovedAsymptoticup} we refine the asymptotic expansion of the $k$-spike solutions $u_{p}$ of problem \eqref{1.1}, proving that they must have the form of the solutions constructed in
\cite{EMP2006} (see Proposition \ref{prop:upHasTheFormInEMP}). This property will be used to compute the topological degree.\vskip 0.1cm
\item  Section \ref{section:approxPWproblems} contains other useful estimates for the profiles approximating $u_{p}$.\vskip 0.1cm
\item Section \ref{section:degreeAndUniqueness} is  devoted to the proof of Theorem \ref{theorem:topological-degree-formula}. In order to compute the topological degree of multi-spike solutions we reduce the problem  to a finite-dimensional one and apply  classical properties of the degree theory. The reduction is possible thanks to the theory developed in Section \ref{Section:ImprovedAsymptoticup} and Section \ref{section:approxPWproblems}.\vskip 0.1cm
\item Finally, as an application of the degree counting formula and of the Morse index computation, in Section \ref{section:localUniqueness} we deduce the local uniqueness result  stated in Theorem \ref{theorem:local-uniqueness}.
\end{itemize}

\
Throughout this paper, we use the same $C$ to denote various generic positive constants independent with $p$, $\|\cdot\|$  to denote the basic norm in the Sobolev space $H^1_0(\Omega)$,  $x\cdot y$ to mean the scalar product of $x$ and $y$ in $\R^2$ and $\langle\cdot,\cdot\rangle$
to mean the corresponding inner product. We will use $\partial$  to denote the partial derivative for any function $h(y,x)$ with respect to $y$, while we will use $D$ to denote the partial derivative for any function $h(y,x)$ with respect to $x$.

\section{Basic results}\label{section:preliminary}

\subsection{Known asymptotic results for the Lane-Emdem problem} \label{subsection:asympt}
\begin{Thm}[c.f. \cite{DGIP2018,DIP2017-1, GILY2021}]\label{thm:asymptoticup}
Let $u_p$ be a family of solutions to problem \eqref{1.1} satisfying  the energy uniform bound
\begin{equation}\label{11-11-01}
\sup_{p}p\|\nabla u_p\|^2_2<\infty.
\end{equation}
Then there exist a finite number of $k$ of distinct points $x_{\infty,j}\in \Omega$, $j=1,\cdots,k$ and a subsequence of $p$ (still denoted by $p$)   such that setting
$\mathcal{S}:=\big\{x_{\infty,1},\cdots, x_{\infty,k}\big\}$,
one has
\begin{equation}\label{11-14-03N}
\lim_{p\rightarrow +\infty} p u_{p}=8\pi \sqrt{e}\sum^k_{j=1}G(x,x_{\infty,j})\,\,~\mbox{in} ~ C^2_{loc}(\Omega\backslash \mathcal{S}),
\end{equation}
the energy satisfies
\begin{equation*}
\lim_{p\rightarrow +\infty} p \int_{\Omega}|\nabla u_{p}(x)|^2dx=8\pi e\cdot k,
\end{equation*}
and the concentrated points $x_{\infty,j}$, $j=1,\cdots,k$, fulfill the system
\begin{equation*}
\nabla_x \Psi_k\big(x_{\infty,1},\cdots,x_{\infty,k}\big)=0.
\end{equation*}
Moreover, for some small fixed $r>0$, let $x_{p,j}\in \overline{B_{2r}(x_{\infty,j})}\subset\Omega$ be the sequence defined as
\begin{equation}\label{def:xpj}
u_{p}(x_{p,j})=\max_{\overline{B_{2r}(x_{\infty,j})}}u_{p}(x),
\end{equation}
 and
 \begin{equation}\label{defepsilonpj}
 \varepsilon_{p,j}:=\Big(p\big(u_{p}(x_{p,j})\big)^{p-1}\Big)^{-1/2},\end{equation}
 then for any $j=1,\cdots,k$ and fixed small  constant $\delta>0$, it holds
\begin{equation*}
\lim_{p\rightarrow +\infty}x_{p,j}=x_{\infty,j}, ~~~\quad\lim_{p\rightarrow +\infty}u_{p}(x_{p,j})=\sqrt{e},
\end{equation*}
\begin{equation}\label{nn3-29-03}
\e_{p,j}= e^{-\frac{p}4}\Bigl(
 e^{-\big( 2\pi \Psi_{k,j}(\boldsymbol{\mathrm x}_{\infty})+\frac{3\log 2}{2}+\frac{3}{4}\big) }+O\big(\frac{1}{p^{1-\delta}}\big)\Bigr),
\end{equation}
and
\begin{equation}\label{3-29-03}
\frac{\e_{p,j}}{\e_{p,s}}=e^{2\pi \big(\Psi_{k,s}(\boldsymbol{\mathrm x}_{\infty})-\Psi_{k,j}(\boldsymbol{\mathrm x}_{\infty})\big)}
 +O\big(\frac{1}{p^{1-\delta}}\big)\,\, ~\mbox{for}~1\leq j,s\leq k,
\end{equation}
where $\Psi_{k,j}$ is the function in \eqref{stts} and $\boldsymbol{\mathrm x}_{\infty}:=\big(x_{\infty,1},\cdots,x_{\infty,k}\big)$.
And setting
\begin{equation}
\label{defwpj}
w_{p,j}(y):=\frac{p}{u_{p}(x_{p,j})}\Big(u_{p}(x_{p,j}+\varepsilon_{p,j}y)-
u_{p}(x_{p,j})\Big),~y\in \Omega_{p,j}:=\frac{\Omega-x_{p,j}}{\varepsilon_{p,j}},
\end{equation}
one has
\begin{equation}\label{5-8-2}
w_{p,j}= U +\frac{w_{0}}{p}+ O\left(\frac{1}{p^{2}}  \right) \mbox{ in}~C^2_{loc}(\R^2),
\end{equation}
where \begin{equation}\label{def:U}U(x):=-2\log \Big(1+\frac{|y|^2}{8}\Big)\end{equation} is the \emph{unique} positive solution of the Liouville equation
\begin{equation*}
\begin{cases}
-\Delta U=e^U \,\,~\mbox{in}~\R^2,\\[1mm]
\displaystyle\int_{\R^2}e^Udx=8\pi,
\end{cases}
\end{equation*}
and $w_0$ solves the non-homogeneous linear equation
\begin{equation}\label{1-2-4}
-\Delta w_0-e^{U}w_0=-\frac{U^2}{2}e^{U}~\mbox{in}~\R^2
.
\end{equation}
Moreover there exists $C>0$ such that
\begin{equation*}
p \min_{j=1,\cdots,k}|x-x_{p,j}|^2u_p^{p-1}(x)\leq C,~~~\mbox{for any $x\in \Omega$ and for any $p>1$}
\end{equation*}
and
\begin{equation}\label{luoluo2}
p \min_{j=1,\cdots,k}|x-x_{p,j}|\cdot |\nabla u_p(x)|\leq C,~~~\mbox{for any $x\in \Omega$ and for any $p>1$}.
\end{equation}
\end{Thm}

Taking \begin{equation}
\label{defepsilon}
\e_p:=\e_{p,1}e^{2\pi \Psi_{k,1}(\boldsymbol{\mathrm x}_{\infty})},
\end{equation} then
we can deduce from \eqref{nn3-29-03} and \eqref{3-29-03} that
\begin{equation}\label{6-26-32}
\lim_{p\to \infty}
\frac{\e_{p,j}e^{2\pi \Psi_{k,j}(\boldsymbol{\mathrm x}_{\infty})}}{\e_p}=1,~~\mbox{for}~~j=1,2\cdots,k.
\end{equation}

\

Let $w_{p,j}$ be the rescaled solution as defined in \eqref{defwpj}, from \cite{DIP2017-1} we know that the following hold:
\begin{Lem}\label{llma}
For any small fixed $\delta,d>0$, there exist $R_\delta>1$ and $p_\delta> 1$ such that
\begin{equation*}
w_{p,j} \leq \big(4-\delta\big)\log \frac{1}{|y|}+C_\delta,~~\mbox{for}~j=1,\cdots,k,
\end{equation*}
for some $C_\delta>0$, provided $R_\delta\leq |y|\leq \frac{d}{\e_{p,j}}$ and $p\geq p_\delta$.
\end{Lem}

We will also need the following result from \cite[Lemma 3.2 and Lemma 3.3]{GILY2021}, concerning the asymptotic behavior of the positive solutions $u_p$ to problem \eqref{1.1} with \eqref{11-11-01}.

\vskip 0.2cm

\begin{Lem}
For any fixed small $d>0$, it holds
\begin{equation}\label{luo-1}
u_p(x)= \sum^k_{j=1}C_{p,j}G(x_{p,j},x)+
o\Big(\sum^k_{j=1}\frac{\varepsilon_{p,j}}{p}\Big)\,\,
~\mbox{in}~C^1\Big(\Omega\backslash \displaystyle\bigcup^k_{j=1} B_{2d}(x_{p,j})\Big),
\end{equation}
as $p\rightarrow +\infty$, where $x_{p,j}$ and $\varepsilon_{p,j}$ are defined in \eqref{def:xpj} and \eqref{defepsilonpj} respectively and
\begin{equation}\label{luoluo1}
C_{p,j}:= \displaystyle\int_{B_d(x_{p,j})}
u_{p}^{p}(y)dy=\frac{1}{p} \Big(8\pi \sqrt{e}+o(1)\Big).
\end{equation}
\end{Lem}

\vskip 0.1cm

\subsection{Limit eigenvalue problem}
\begin{Lem}[c.f. \cite{DGIP2019}]\label{lemma:limitEigenvalueProblem}
Let $U$ be the function defined in \eqref{5-8-2}, $\lambda\in [0,1]$ and $v\in C^2(\R^2)$ be a solution of the following problem
\[
\begin{cases}
-\Delta v=\lambda e^Uv\,\,~\mbox{in}~\R^2,\\[1mm]
\displaystyle\int_{\R^2}|\nabla v|^2dx<\infty.
\end{cases}
\]
Then either $\lambda=0$ or $\lambda=1$.
\end{Lem}
\begin{Lem}[c.f. \cite{DGIP2019}]\label{llm}
Let $U$ be the function defined in \eqref{5-8-2} and $v\in C^2(\R^2)$ be a solution of the following problem
\begin{equation}\label{kernelEq}
\begin{cases}
-\Delta v=e^Uv\,\,~\mbox{in}~\R^2,\\[1mm]
\displaystyle\int_{\R^2}|\nabla v|^2dx<\infty.
\end{cases}
\end{equation}
Then it holds
\begin{equation*}
v(x)\in \mbox{span}\left\{\frac{\partial U(x)}{\partial x_1},
\frac{\partial U(x)}{\partial x_2},\frac{8-|x|^2}{8+|x|^2}\right\}.
\end{equation*}
\end{Lem}

\subsection{Quadratic forms $P_{j}$ and $Q_{j}$}$\;$\\

For each point $x_{p,j},$ $j=1,\cdots, k$, defined in \eqref{def:xpj}, let
us set the following  quadric forms \begin{equation}\label{07-08-20}
\begin{split}
P_{j}(u,v):=&- 2\theta\int_{\partial B_\theta(x_{p,j})}
\big( \nabla u \cdot \nu\big)
\big( \nabla v\cdot\nu\big)\,d\sigma
+  \theta  \int_{\partial B_\theta(x_{p,j})}
\nabla u \cdot \nabla v \,d\sigma,
\end{split}
\end{equation}
and
\begin{equation}\label{abd}
Q_{j}(u,v):=- \int_{\partial B_\theta(x_{p,j})}\frac{\partial v}{\partial \nu}\frac{\partial u}{\partial x_i}\,d\sigma-
 \int_{\partial B_\theta(x_{p,j})}\frac{\partial u}{\partial \nu}\frac{\partial v}{\partial x_i}\,d\sigma
+ \int_{\partial B_\theta(x_{p,j})}\big( \nabla u\cdot\nabla v \big) \nu_i\,d\sigma,\,\,\,\,i=1,2,
\end{equation}
where $u,v\in C^{2}(\overline{\Omega})$ and  $\theta>0$ is such that $B_{2\theta}(x_{p,j})\subset\Omega$.
Then we have following results, which can be found in \cite{CGPY2019,GILY2021}.

\begin{Lem}\label{lem2-1}
It holds
\begin{equation}\label{1-1}
P_{j}\Big(G(x_{p,s},x), G(x_{p,m},x)\Big)=
\begin{cases}
-\frac{1}{2\pi} ~&\mbox{for}~s=m=j,\\[1mm]
0~&\mbox{for}~s\neq j ~{or}~m \neq j,
\end{cases}
\end{equation}
and
 \begin{equation}\label{abb1-1}
P_{j}\Big(G(x_{p,s},x),\partial_hG(x_{p,m},x)\Big)=
\begin{cases}
\frac{1}{2}\frac{\partial R(x_{p,j})}{\partial x_h} ~&\mbox{for}~s=m=j,\\[1mm]
 -D_{h}G \big(x_{p,s},x_{p,j}\big) ~&\mbox{for} ~m=j,~s\neq j,\\[1mm]
0 ~&\mbox{for}~m\neq j.
\end{cases}
\end{equation}
Moreover
\begin{equation}\label{aluo1}
Q_{j}\Big(G(x_{p,m},x),G(x_{p,s},x)\Big)=
\begin{cases}
-\frac{\partial R(x_{p,j})}{\partial{x_i}} ~&\mbox{for}~s=m=j,\\[1mm]
D_{i}G(x_{p,m},x_{p,j})
 ~&\mbox{for}~m\neq j,~s=j,\\[1mm]
D_{i}G(x_{p,s},x_{p,j})
~&\mbox{for}~m=j,~s\neq j,\\[1mm]
0 ~&\mbox{for}~s,m\neq j,
\end{cases}
\end{equation}
and
\begin{equation}\label{aluo41}
Q_{j}\Big(G(x_{p,m},x),\partial_h G(x_{p,s},x)\Big)=
\begin{cases}
-\frac{1}{2}\frac{\partial^2 R(x_{p,j})}{\partial{x_ix_h}} ~&\mbox{for}~s=m=j,\\[1mm]
D_{i}\partial_h G(x_{p,s},x_{p,j})
 ~&\mbox{for}~m= j,~s\neq j,\\[1mm]
D^2_{ih} G(x_{p,m},x_{p,j})
~&\mbox{for}~m\neq j,~s=j,\\[1mm]
0 ~&\mbox{for}~s,m\neq j,
\end{cases}
\end{equation}
where $G(x,y)$ and $R(x,y)$ are the Green and Robin function, $\partial_iG(y,x):=\frac{\partial G(y,x)}{\partial y_i}$ and
$D_{i}G(y,x):=\frac{\partial G(y,x)}{\partial x_i}$.
\end{Lem}

\vskip 0.2cm

\subsection{Local identities} ~~

\vskip 0.2cm

 We will need the following basic formula involving
the solution $u_p$ and the eigenfunction $v_{p,l}$ of the linearized problem at $u_{p}$, which solve
\begin{equation}\label{5-20-1}
-\Delta u_p=u^p_p\quad\mbox{ and }\quad\Delta v_{p,l}=\lambda_{p,l} pu^{p-1}_pv_{p,l},\,\,\,\,\,\mbox{in }\Omega.
\end{equation}
\begin{Lem}[c.f. \cite{GT1983}] It holds
\label{lem2-5}
\begin{equation}\label{6-25-21}
\begin{split}
 \Big(p\lambda_{p,l}-1\Big)\int_{B_{d}(x_{p,m})} u_p^p  v_{p,l}=&
\int_{\partial B_{d}(x_{p,m})} \left(\frac{\partial u_p}{\partial \nu}  v_{p,l} - \frac{\partial v_{p,l}}{\partial \nu}  u_p\right).
\end{split}\end{equation}
\end{Lem}
\begin{proof}
Observe that for any smooth bounded domain $\Omega'\subset \mathbb R^{2}$ and any  $u,v\in C^2(\overline{\Omega'})$ one has (c.f. \cite{GT1983})
\begin{equation}\label{prec6-25-21}
\int_{\Omega'}\Big(u\Delta v-v\Delta u\Big)dx=\int_{\partial \Omega'}\Big(u\frac{\partial v}{\partial \nu}-v\frac{\partial u}{\partial \nu} \Big)d\sigma,
\end{equation}
where $\nu$ is the outward unit normal of $\partial \Omega'$.
Putting $u=u_p$, $v=v_{p,l}$ and  $\Omega'=B_{d}(x_{p,j})$, for $j=1,\cdots,k$ in \eqref{prec6-25-21}, and using \eqref{5-20-1} one gets \eqref{6-25-21}.
\end{proof}
The following local Pohozaev identities involving
$u_p$ and  $v_{p,l}$ will be also crucial.
\begin{Lem} \label{lemma:pohoz} Let $P_j$ and $Q_{j}$ be the quadratic forms in \eqref{07-08-20} and \eqref{abd}, respectively.
It holds

\begin{equation}\label{3-12-04}
\begin{split}
P_j\big(u_p,v_{p,l}\big)=&
\int_{\partial B_{d}(x_{p,j})} u_p^p  v_{p,l} ( x-x_{p,j})\cdot\nu d\sigma- 2\int_{B_{d}(x_{p,j})}u_p^p  v_{p,l}dx\\&
+(\lambda_{p,l}-1)p \int_{B_{d}(x_{p,j})}
( x-x_{p,j})\cdot\nabla u_p  v_{p,l}u_p^{p-1} dx,
\end{split}
\end{equation}
and
\begin{equation}\label{3-12-04a}
\begin{split}
Q_j\big(u_p,v_{p,l}\big)=&\int_{\partial B_{d}(x_{p,j})} u_p^p  v_{p,l} \nu_id\sigma
+(\lambda_{p,l}-1)p \int_{B_{d}(x_{p,j})}u_p^{p-1}  v_{p,l} \frac{\partial u_p}{\partial x_i}d\sigma.
\end{split}
\end{equation}
\end{Lem}

\begin{proof}

Firstly, for $y\in\mathbb R^{2}$, from \eqref{5-20-1}, we get
\begin{equation}\label{5-20-3}
\begin{split}
\int_{ \Omega'}&\Big(- (x-y)\cdot\nabla v_{p,l}\Delta u_p
-\big( x-y\cdot\nabla u_p\big) \Delta v_{p,l} \Big)dx\\=&
\int_{ \Omega'} \Big( (x-y)\cdot\nabla v_{p,l}u^p_p
+\lambda_{p,l}pu^{p-1}_pv_{p,l}(x-y)\cdot\nabla u_p\Big)dx.
\end{split}\end{equation}
Then, we can use the integration by parts on both sides of
\eqref{5-20-3} to prove  \eqref{3-12-04} by taking $\Omega'=B_{d}(x_{p,j})$ and $y=x_{p,j}$.

\vskip 0.1cm

Also from \eqref{5-20-1} we find
\begin{equation}\label{5-20-2}
-\frac{\partial v_{p,l}}{\partial x_i} \Delta u_p
-\frac{\partial u_p}{\partial x_i} \Delta v_{p,l} =u^p_p \frac{\partial v_{p,l}}{\partial x_i}
+pu^{p-1}_pv_{p,l} \frac{\partial u_p}{\partial x_i}+\big(
\lambda_{p,l}-1\big) pu^{p-1}_pv_{p,l}\frac{\partial u_p}{\partial x_i}.
\end{equation}
Then integrating in  $B_{d}(x_{p,j})$ on both sides of \eqref{5-20-2}, we get  \eqref{3-12-04a}.
\vskip 0.1cm

The readers can refer to the proof of Proposition~2.6 in \cite{GILY2021}
for details.

\end{proof}

\section{Preliminary results for the eigenfunctions}\label{section:preliminaryEigen}
In this section we obtain starting asymptotic expansions, as $p\rightarrow +\infty$, for the eigenfunctions $v_{p,l}$ associated to the eigenvalues $\lambda_{p,l}$ of the linearized problem:
\begin{equation}\label{6-21-1}
\begin{cases}
-\Delta v_{p,l}=\lambda_{p,l} \, p\, u^{p-1}_p\, v_{p,l}~&\mbox{in}~\Omega,\\[2mm]
v_{p,l}=0~&\mbox{on}~\partial \Omega.
\end{cases}
\end{equation}
In particular, in order to compute the Morse index of the solution $u_{p}$ (see Section \ref{section:Morse}), we need to understand the behavior of $v_{p,l}$ in the two occurrences when the associated eigenvalue $\lambda_{p,l}$ converges to either $0$ or $1$, these results are collected in  Proposition \ref{prop3-3} and Proposition \ref{lemma:autofunz2casi} below.
\vskip 0.2cm
W.l.g. we may assume that the eigenfunctions $v_{p,l}$ satisfy
\begin{equation}\label{normalization}
\|v_{p,l}\|_{L^{\infty}(\Omega)}=1.
\end{equation}
Letting $l\in \N$, we take for any $j=1,\cdots, k$,
\begin{equation}\label{6-21-2}
\widetilde{v}_{p,l,j}(x):= v_{p,l}\big(x_{p,j}+\e_{p,j}x\big),~~x\in \Omega_{p,j}:=\Big\{x: x_{p,j}+\e_{p,j}x\in \Omega\Big\}.
\end{equation}
From \eqref{6-21-1} and \eqref{normalization}, it follows that
\begin{equation*}
\begin{cases}
-\Delta \widetilde{v}_{p,l,j} =\lambda_{p,l}\widetilde{V}_{p,j}
\widetilde{v}_{p,l,j},&~\mbox{in}~~\Omega_{p,j},\\[1mm]
\widetilde{v}_{p,l,j}=0,&~\mbox{in}~~\partial\Omega_{p,j},\\[1mm]
\|\widetilde{v}_{p,l,j}\|_{L^{\infty}(\Omega_{p,j})}=1,
\end{cases}\end{equation*}
where $\widetilde{V}_{p,j}:=
\Big(1+\frac{w_{p,j}}{p}\Big)^{p-1}$
with $w_{p,j}$ the rescaled function in \eqref{defwpj}.

\begin{Lem}\label{lem2-6}
For any $l\in \N$, if $\lambda_{p,l}$ is bounded independent of $p$,  then for any $j=1,\cdots,k$, there exists
 $V_{l,j}(x)\in C^{2,\alpha}_{loc}\big(\R^2\big)$ with some $\alpha \in (0,1)$ satisfying
\begin{equation}\label{convRisc}
\widetilde{v}_{p,l,j}\to V_{l,j}~~\mbox{in}~~C^{2,\alpha}_{loc}\big(\R^2\big),
~\mbox{as}~p\to \infty,
\end{equation}
 and
 \begin{equation}\label{3-23-1}
\begin{cases}
-\Delta V_{l,j}=\lambda_{\infty,l} e^{U} V_{l,j}\,\,~\mbox{in}~\R^2,\\[2mm]
\|V_{l,j}\|_{L^{\infty}(\R^2)}\leq 1,
\end{cases}
\end{equation}
where $\displaystyle\lambda_{\infty,l}:=\lim_{p\to \infty}\lambda_{p,l}$ by taking a subsequence.
\end{Lem}
\begin{proof}

We observe from \eqref{5-8-2}
that $\widetilde{V}_{p,j}$ is bounded in $B_R(0)$ for any $R>0$.
Thus  the result follows from the $L^p$ estimates
and Schauder estimates for the elliptic equations.

\end{proof}

To examine the Morse index of $u_p$, we will need to study the case $\lambda_{\infty,l}=0$ and $\lambda_{\infty,l}=1$. W.l.g we further assume the orthogonality of the eigenfunctions in Dirichlet norm as follows:
\begin{equation}\label{5-25-31}
\int_{\Omega}\nabla v_{p,l}\cdot \nabla v_{p,l'}dx =0~~~\mbox{if}~~l\neq l'.
\end{equation}
In fact, if $\lambda_{\infty,l}\neq \lambda_{\infty,l'}$, then \eqref{5-25-31} is direct.
If $\lambda_{\infty,l}= \lambda_{\infty,l'}$, by standard Schimidt orthogonalization, we find
 new functions (still denoted by $v_{p,l}$) satisfying \eqref{5-25-31}.

\begin{Lem}\label{prop3-3prima}
Let $l\in \N$, if $\lambda_{\infty,l}\in [0,\infty)$, then there exists $j\in\{1,2,\cdots,k\}$ such that it holds
\begin{equation}\label{5-25-10}
V_{l,j}\not\equiv 0,
\end{equation}
where $V_{l.j}$ is the limit function in \eqref{convRisc}.
\end{Lem}
\begin{proof}
Firstly,  using \eqref{11-14-03N} and the fact $\|v_{p,l}\|_{L^{\infty}(\Omega)}=1$, it holds
\begin{equation*}
\begin{split}
\lambda_{p,l} pu^{p-1}_pv_{p,l}
= O\Big( \frac{C^{p-1}}{p^{p-2}}\Big)
 \to 0~~\mbox{locally uniformly in}~~\Omega\backslash \big\{x_{\infty,1},\cdots,x_{\infty,k}\big\} ~~\mbox{as}~~ p\to \infty.
\end{split}
\end{equation*}
Then from the standard elliptic regularity estimates, we can deduce that
\begin{equation}\label{zero}
\begin{split}
 {v}_{p,l}(x)
 \to 0~~\mbox{local uniformly in}~~\Omega\backslash \big\{x_{\infty,1},\cdots,x_{\infty,k}\big\} ~~\mbox{as}~~ p\to \infty.
\end{split}
\end{equation}

Assume by contradiction that $V_{l,j}\equiv 0$ for some $l\in \N$ and
all $j=1,\cdots,k$, namely that
\begin{equation*}
\begin{split}
 \widetilde{v}_{p,l,j}(x) \to 0~~\mbox{local uniformly in}~~\R^2 ~~\mbox{as}~~ p\to \infty.
\end{split}
\end{equation*}
We will prove that for large $R>0$, it holds
\begin{equation*}
\widetilde{v}_{p,l,j}(x) =o_R(1)~~\mbox{ in}~~B_{\varepsilon_{p,j}^{-1}\tau}(0)\setminus B_R(0),
\end{equation*}
which, together with \eqref{zero}, will give a contradiction to
$\|v_{p,l}\|_{L^{\infty}(\Omega)}=1$.

\vskip 0.1cm

In fact, let $\alpha_{p, j}=\max_{ |y|= R }|v_{p,l}(y)|$ and
$\beta_{p, j}=\max_{ |y|= \varepsilon_{p,j}^{-1}\tau }|v_{p,l}(y)|$.
Then $\alpha_{p, j}=o(1)$ and $\beta_{p, j}=o(1)$. In view of Lemma~\ref{llma},
we have
\[
|\lambda_{p,l}\widetilde{V}_{p,j}
\widetilde{v}_{p,l,j}|\le \frac{C}{|y|^{4-\delta}},\quad
R\leq |y|\leq \frac{\tau}{\e_{p,j}}.
\]
Let $w$ be the solution of the following problem
\[
\begin{cases}
-\Delta w =\frac{C}{|y|^{4-\delta}},\;\;\text{in}\; B_{\varepsilon_{p,j}^{-1}\tau}(0)\setminus B_R(0),\\[1mm]
w(R)= \alpha_{p, j},\;\;
w(\e_{p,j}^{-1}\tau)= \beta_{p, j}.
\end{cases}
\]
Then, we have $w>0$ and $|\widetilde{v}_{p,l,j}|\le w$. On the other hand,

\[
\begin{split}
w(y)=&\alpha_{p, j}+\frac C{(2-\delta)^2}\frac1{R^{2-\delta}}+\Bigl(\beta_{p, j}+
\frac C{(2-\delta)^2}\frac{\e_{p,j}^{2-\delta}}{\tau^{2-\delta}}-
\alpha_{p, j}-\frac C{(2-\delta)^2}\frac1{R^{2-\delta}}
\Bigr)\frac{\ln\frac{|y|}{R}}{ \ln\frac{\tau}{\e_{p,j}R} }\\
&-\frac{C}{(2-\delta)^2}\frac1{|y|^{2-\delta}},
\end{split}
\]
which gives $w=o(1)$. Thus the claim follows.

\end{proof}

\begin{Prop}\label{prop3-3}
Let $l\in \N$, we have the following results:

 \vskip 0.2cm

\noindent \textup{(1)}~ If $\lambda_{\infty,l}=0$, then for any  $j=1,2,\cdots,k$, there exists $c_{l,j}\in\mathbb R$ such that
\begin{equation*}
	\widetilde{v}_{p,l,j}= c_{l,j} +o(1),
\end{equation*}
where $\widetilde{v}_{p,l,j}$ is the rescaled eigenfunction in \eqref{6-21-2}.
Let $\textbf{c}_{l}:=\big(c_{l,1},\cdots, c_{l,k}\big)$, then $$\textbf{c}_{l}\neq \textbf{0}.$$
 Moreover, if $\lambda_{\infty,l}=\lambda_{\infty,l'}=0$ for $l\neq l'$, then we have
\begin{equation}\label{6-26-01}
c_{l,j}c_{l',j}=0\ ~~\,\,~\forall~~ j\in \big\{1,2,\cdots,k\big\}.
\end{equation}
 \vskip 0.2cm

\noindent \textup{(2)}~ If $\lambda_{\infty,l}=1$, then for any $j=1,\cdots, k$ there exists  $(a_{1,l,j},a_{2,l,j},b_{l,j})\in\mathbb R^{3}$ such that
\begin{equation}\label{6-21-4}
\widetilde{v}_{p,l,j}(y)
=\sum^2_{q=1}a_{q,l,j}\frac{y_q}{8+|y|^2}+b_{l,j}\frac{8-|y|^2}{8+|y|^2}+o(1),\quad~\mbox{in}~C^1_{loc}(\R^2),
~~\mbox{as}~p\to +\infty,
\end{equation}
and there exists $j\in \{1,\cdots, k\}$ such that $(a_{1,l,j},a_{2,l,j},b_{l,j})\neq (0,0,0)$.

\vskip 0.1cm

Moreover, if $\lambda_{\infty,l}=\lambda_{\infty,l'}=1$ for $l\neq l'$, then we have
\begin{equation}\label{5-25-40}
\textbf{a}_{l,j}\cdot \textbf{a}_{l',j}+32b_{l,j}b_{l',j}=0,~~\mbox{for any}~j=1,\cdots,k,
\end{equation}
where  $\textbf{a}_{l,j}:=\big(a_{1,l,j},a_{2,l,j}\big)\in \R^{2}$.
\end{Prop}

\begin{proof}
\underline{Case (1):} Let $V_{l,j}$ be the limit function in  \eqref{convRisc}, then it satisfies
\eqref{3-23-1} which  reduces to  \[
\begin{cases}
\Delta V_{l,j}=0\,\,~\mbox{in}~\R^2,\\[2mm]
\|V_{l,j}\|_{L^{\infty}(\R^2)}\leq 1.
\end{cases}
\]
Hence the function $V_{l,j}$ is a constant.  Furthermore Lemma \ref{prop3-3prima} implies that $\textbf{c}_{l}\neq \textbf{0}$.
Next we prove \eqref{6-26-01}. Let us observe that  the orthogonality of the eigenfunctions \eqref{5-25-31} is equivalent to the condition
\[\int_{\Omega_{p,j}}\Big(1+\frac{w_{p,j}}{p} \Big)^{p-1}\widetilde{v}_{p,l,j}\widetilde{v}_{p,l',j}=0, \quad\mbox{ for any }j=1,\cdots, k.\]
By the convergences in \eqref{5-8-2} and in \eqref{convRisc}  and the estimate in Lemma \ref{llma} one can pass to the limit as $p\rightarrow +\infty$ using Lebesgue's Theorem and obtain
\[c_{l,j}c_{l',j}\int_{\mathbb R^{2}}e^{U(y)}dy=0.\]
Hence  $c_{l,j} c_{l',j}=0$, for any $j=1,\cdots, k$, from which we deduce \eqref{6-26-01}.

\vskip 0.2cm

\noindent\underline{Case (2):}
By Lemma \ref{lem2-6} and the assumption $\lambda_{\infty,l}=1$ we have that the limit function $V_{l,j}(x)$ in \eqref{convRisc}
solves problem \eqref{kernelEq}. Then from Lemma \ref{llm} we can deduce \eqref{6-21-4}, for some $(a_{1,l,j},a_{2,l,j},b_{l,j})\in \R^3$ and \eqref{5-25-10} implies that there exists $j\in\{1,\cdots,k\}$ such that
\begin{equation*}
(a_{1,l,j},a_{2,l,j},b_{l,j})\neq (0,0,0).
\end{equation*}
By the orthogonality of different eigenfunctions as stated in
\eqref{5-25-31}, arguing as in Case (1) we deduce that
\[
\int_{\R^2}e^{U(y)}V_{l,j}(y) V_{l',j}(y)dy=0, \quad \forall j=1,\cdots, k.
\]
Using the expression of $V_{l,j}, V_{l',j}$ in \eqref{6-21-4} one deduces that
\begin{equation}\label{luo-1-1}
\sum^2_{q=1}a_{q,l,j}a_{q,l',j}\int_{\R^2}e^{U(y)} \frac{y_q^{2}}{(8+|y|^2)^{2}}dy+b_{l,j}b_{l',j}\int_{\R^2}e^{U(y)} \left(\frac{8-|y|^2}{8+|y|^2}\right)^{2}dy=0, \quad \forall j=1,\cdots, k.
\end{equation}
Also by direct computations, we know
\begin{equation}\label{luo-1-2}
\begin{split}
 \int_{\R^2}e^{U(y)} \frac{y_q^{2}}{(8+|y|^2)^{2}}dy =32 \int_{\R^2} \frac{|y|^{2}}{(8+|y|^2)^{4}}dy
 =64\pi \int^{\infty}_{0} \frac{r^{3}}{(8+r^2)^{4}}dr=\frac{\pi}{12},
\end{split}
\end{equation}
and
\begin{equation}\label{luo-1-3}
\begin{split}
\int_{\R^2}e^{U(y)} \left(\frac{8-|y|^2}{8+|y|^2}\right)^{2}dy =64 \int_{\R^2} \frac{ (8-|y|^2)^2}{(8+|y|^2)^{4}}dy
 =128\pi \int^{\infty}_{0} \frac{(8-r^2)^2r}{(8+r^2)^{4}}dr=\frac{8\pi}{3}.
\end{split}
\end{equation}
Then \eqref{5-25-40} follows by \eqref{luo-1-1}, \eqref{luo-1-2} and \eqref{luo-1-3}.
\end{proof}

\begin{Lem}\label{lem2-9}
Let $l\in \mathbb{N}$, then
\begin{equation}\label{3-12-03d}
\begin{split}
\frac{v_{p,l}(x)}{\lambda_{p,l}}=& p\sum^k_{m=1} G(x_{p,m},x) A_{p,l,m}
+ p\sum^k_{m=1}\sum^2_{q=1} \frac{\partial G(x_{p,m},x) }{\partial x_q} B_{p,l,m,q}
\\&+o\Big(\e_p\Big)\,\,\,\,\,\,\,\,\,\,\, \mbox{in}~
C^1_{loc}\Big(\overline{\Omega}\backslash \big\{x_{\infty,1},\cdots,x_{\infty,k}\big\}\Big),
\end{split}
\end{equation}
where $\varepsilon_{p}$ is defined in \eqref{defepsilon},
\begin{equation}\label{6-21-21}
A_{p,l,m}:=\int_{B_d(x_{p,m})}
u^{p-1}_{p}(y) v_{p,l}(y)dy,
\end{equation}
and
\begin{equation}\label{6-21-22}
B_{p,l,m,q}:= \int_{B_d(x_{p,m})}\big(y-x_{p,m}\big)_q
u^{p-1}_{p}(y) v_{p,l}(y)dy.
\end{equation}
\end{Lem}
\begin{proof}
For $x\in \Omega\backslash  \bigcup^k_{j=1} B_{2d}(x_{\infty,j})$, we have
\begin{equation}\label{5-8-11}
\begin{split}
v_{p,l}(x)=& p \lambda_{p,l} \int_{\Omega}G(y,x)
 u^{p-1}_{p}(y) v_{p,l}(y) dy\\=&  p \lambda_{p,l} \sum^k_{m=1}
 \int_{B_d(x_{p,m})}G(y,x)
 u^{p-1}_{p}(y) v_{p,l}(y) dy
\\& + p \lambda_{p,l} \int_{ \Omega\backslash  \displaystyle\bigcup^k_{m=1} B_{d}(x_{p,m})}G(y,x)
 u^{p-1}_{p}(y) v_{p,l}(y) dy.
\end{split}
\end{equation}
 From \eqref{5-8-11}, we can proceed in exactly the same way
 as in the proof of Proposition~4.2 in \cite{GILY2021}
 to verify \eqref{3-12-03d}.

\end{proof}

\begin{Prop}\label{lemma:autofunz2casi}
Let $l\in \mathbb{N}$.
  If $\lambda_{\infty,l}=1$, then
\begin{equation}\label{3-12-03}
v_{p,l}(x)=-\frac{8\pi}{p} \Big(\sum^k_{j=1}G\big(x,x_{p,j}\big)b_{l,j}+o(1)\Big)~\,~\mbox{in}~
C^1_{ {loc}}\Big(\overline{\Omega}\backslash \big\{x_{\infty,1},\cdots,x_{\infty,k}\big\}\Big),
\end{equation}
where $b_{l,j}\in\mathbb R$, $j=1,\cdots, k$ are the constants in \eqref{6-21-4}.
Moreover, if $b_{l,j}\neq 0$ for some $j\in \{1,\cdots,k\}$, then it holds
\begin{equation}\label{1-25-01}
\lambda_{p,l}=1+\frac{6}{p}\big(1+o(1)\big)~\,\mbox{as}~~p\to +\infty.
\end{equation}
\end{Prop}
\begin{proof}

We will use Lemma~\ref{lem2-9} to prove this proposition. Firstly, we estimate
 $A_{p,l,m}$ and  $B_{p,l,m,i}$ defined in
\eqref{6-21-21} and \eqref{6-21-22}.

\vskip 0.1cm

 Since $\lambda_{\infty,l}=1$,  using  \eqref{6-21-4}, we have
 \begin{equation}\label{firstApproxAp}
	A_{p,l,m}=\frac{1}{p} \left( b_{l,m} \int_{\R^2} e^{U(z)}\frac{8-|z|^2}{8+|z|^2}  dz+o(1)\right)=0+o\Big(\frac{1 }{p}\Big)=o\Big(\frac{1 }{p}\Big).
	\end{equation}
Also, by scaling, we find
\begin{equation}\label{3-16-4t}
\begin{split}
B_{p,l,m,i}=&\frac{\e_{p,m}}{p} \int_{B_{\frac{d}{\e_{p,m}}(0)}} z_i \Big(1+\frac{w_{p,m}(z)}{p}\Big)^{p-1}   \widetilde{v}_{p,l,m}(z)dz
\\=&
\frac{\e_{p,m}}{p} \left( \int_{\R^2} z_i e^{U(z)}V_{l,m}(z)   dz+o(1)\right)
\\=&  \frac{\e_{p,m}}{p} \left( a_{i,l,m}\int_{\R^2}\frac{z_i^2}{8+|z|^2} e^{U(z)}   dz+o(1)\right)
\\
=&  a_{i,l,m}2\pi    \frac{\e_{p,m}}{p}+o\Big(\frac{\e_{p,m}}{p} \Big)
=O\Big(\frac{\e_{p,m}}{p} \Big).
\end{split}
\end{equation}
We  show that:
\begin{equation}\label{3-16-4s}
	A_{p,l,m} =
	-\frac{8\pi }{p^2} \Big( b_{l,m}
	+o(1)\Big).
\end{equation}
Observe that by \eqref{3-16-4t} and  Lemma \ref{lem2-9}  one has
	\begin{equation}\label{260bis}v_{p,l}(x)= p\lambda_{p,l} \left(\sum^k_{m=1} G(x_{p,m},x) A_{p,l,m}
		+O\Big(\frac{\e_p}{p}\Big)\right)\,\,\,\,\,\,\,\,\,\,\, \mbox{in}~
		C^1_{loc}\Big(\overline{\Omega}\backslash \big\{x_{\infty,1},\cdots,x_{\infty,k}\big\}\Big).
	\end{equation}
Also by Lemma \ref{lem2-5}, we  get
\begin{equation}\label{6-21-41}
\begin{split}
 \Big(p\lambda_{p,l}-1\Big)\int_{B_{d}(x_{p,m})} u_p^p  v_{p,l}=&
\int_{\partial B_{d}(x_{p,m})} \left(\frac{\partial u_p}{\partial \nu}  v_{p,l} - \frac{\partial v_{p,l}}{\partial \nu}  u_p\right).
\end{split}\end{equation}
From which,  using \eqref{luo-1}, \eqref{firstApproxAp} and \eqref{260bis}, we get
\begin{equation}\label{6-21-35}
\begin{split}
 \int_{B_{d}(x_{p,m})} u_p^p  v_{p,l}& =
O\Big(\frac{1}{p}\sum^k_{m=1}\big|A_{p,l,m}\big|\Big)+o\left(\frac{1}{p}\sum^k_{m=1}\big|A_{p,l,m}\big|\right)+ o\left( \frac{\varepsilon_p}{p}\right)
=o\Big(\frac{1}{p^2}\Big).
\end{split}\end{equation}

On the other hand, we know
\begin{equation}\label{3-16-4}
\begin{split}
A_{p,l,m}& u_p(x_{p,m})- \int_{B_{d}(x_{p,m})} u_p^p  v_{p,l}dx\\=&
\int_{B_{d}(x_{p,m})} u_p^{p-1}(x)\Big(
u_p(x_{p,m})-u_p(x)\Big)  v_{p,l}(x)dx\\=&
- \frac{u_p(x_{p,m})}{p^2} \int_{B_{\frac{d}{\e_{p,m}}(0)}}  w_{p,m} \Big(1+\frac{w_{p,m}(z)}{p}\Big)^{p-1}   \widetilde{v}_{p,l,m}(z)dz\\=&
-\frac{\sqrt{e}}{p^2} \left(
\int_{\R^2} e^{U(z)} U(z)\Big(\sum^2_{q=1}\frac{a_{q,l,m}z_q}
{8+|z|^2}+b_{l,m}\frac{8-|z|^2}{8+|z|^2}\Big)
 dz+o(1)\right)\\=&
 -\frac{8\pi \sqrt{e}}{p^2} \Big( b_{l,m}
 +o(1)\Big).
\end{split}
\end{equation}
Now by \eqref{6-21-35} and \eqref{3-16-4}, we deduce
\begin{equation*}
A_{p,l,m} \Big(\sqrt{e}
 +o(1)\Big)+  o\Big(\frac{1}{p^2}\Big)
=
 -\frac{8\pi \sqrt{e}}{p^2} \Big( b_{l,m}
 +o(1)\Big),
\end{equation*}
which gives us \eqref{3-16-4s}.\\

In conclusion  we deduce \eqref{3-12-03} by  \eqref{3-16-4s} and \eqref{260bis}.

\vskip 0.2cm

In the following we prove \eqref{1-25-01}. We will use the Pohozaev identity \eqref{3-12-04}. Let $P_{j}$ be the quadratic form defined in \eqref{07-08-20}, using the expansions of $u_{p}$ and $v_{p,l}$ in  \eqref{luo-1} and \eqref{3-12-03} respectively, we obtain
\begin{equation}\label{6-21-61}
\begin{split}
P_{j}(u_{p},v_{p,l})
=&
-\frac{64\pi^2\sqrt{e}}{p^2} \sum^k_{s=1}\sum^k_{m=1} b_{l,m}P_j\big(G(x_{p,s},x),G(x_{p,m},x)\big)+
o\Big(\frac{1}{p^2}\Big)\\=&
 \frac{32\pi\sqrt{e}b_{l,j}}{p^2} +
o\Big(\frac{1}{p^2}\Big)~\quad~\mbox{ for any}~~~~j=1,\cdots,k,
\end{split}
\end{equation}
where the last equality follows from \eqref{1-1}.
Next, using \eqref{11-14-03N}, we obtain
\begin{equation}\label{6-21-62}
\begin{split}
\int_{\partial B_{d}(x_{p,j})} u_p^p  v_{p,l} ( x-x_{p,j})\cdot\nu
=O\Big(\frac{C^{p}}{p^{p}}\Big).
\end{split}\end{equation}
From \eqref{3-16-4s} and \eqref{6-21-35}, we find
\begin{equation}\label{6-21-63}
\begin{split}
 \int_{B_{d}(x_{p,j})} u_p^p  v_{p,l}= O\Big(\frac{1}{p^3}\Big).
\end{split}\end{equation}
We define $\widetilde{u}_{p,j}(x):=u_{p}(x_{p,j}+\varepsilon_{p,j}x)$, then $\nabla \widetilde{u}_{p,j}=\frac{u_{p}(x_{p,j})}{p}\nabla w_{p,j}$, where $w_{p,j}$ is the rescaled function in \eqref{defwpj} and so by \eqref{5-8-2}, one has
\begin{equation}
\label{convUtilde}
\nabla\widetilde{u}_{p,j}=\frac{\sqrt e}{p}\big(\nabla U+ o(1)\big)\ \mbox{ in }C_{loc}(\mathbb R^{2}) \ \mbox{  as }p\rightarrow +\infty.
\end{equation}
After scaling, using \eqref{convUtilde},  the convergence  in \eqref{5-8-2} for $w_{p,m}$, the convergence in  \eqref{convRisc} and \eqref{6-21-4}
for $\widetilde{v}_{p,l,m}$ and the estimate in Lemma \ref{llma},   passsing to the limit as $p\rightarrow +\infty$ by  Lebesgue's Theorem and \eqref{luoluo2}, we obtain
\begin{equation}\label{6-21-64}
\begin{split}
\int_{B_{d}(x_{p,j})} &u_p^{p-1} v_{p,l} ( x-x_{p,j})\cdot\nabla u_p  \\=&
\frac{1}{p}
\int_{B_{\frac{d}{\e_{p,j}}(0)}} \Big(1+\frac{w_{p,j}(z)}{p}\Big)^{p-1}
 \widetilde{v}_{p,l,j}(z) \ z\cdot\nabla \widetilde{u}_{p,j}(z)  dz\\=&
\frac{\sqrt{e}}{p^2} \left(
\int_{\R^2} e^{U(z)}\Big(\sum^2_{q=1}\frac{a_{q,l,j}z_q}{8+|z|^2}+b_{l,j}\frac{8-|z|^2}{8+|z|^2}\Big)
 \ z\cdot\nabla U(z)  dz+o(1)\right)\\=&
\frac{16 \pi \sqrt{e}}{3 p^2} \Big( b_{l,j} +o(1)\Big).
\end{split}\end{equation}
Finally, substituting \eqref{6-21-61},
\eqref{6-21-62}, \eqref{6-21-63} and \eqref{6-21-64} into the Pohozaev identity \eqref{3-12-04}, we get
\begin{equation}\label{3-12-05}
 \frac{32\pi\sqrt{e}b_{l,j}}{p^2} +
o\Big(\frac{1}{p^2}\Big) = (\lambda_{p,l}-1) \frac{16 \pi \sqrt{e}}{3 p} \Big( b_{l,j} +o(1)\Big).
\end{equation}
Since by assumption $b_{l,j}\neq 0$ for some $j\in \{1,\cdots,k\}$, then
 \eqref{1-25-01} follows by  \eqref{3-12-05}.
\end{proof}
\section{The computation of the Morse index of $u_p$}\label{section:Morse}

In this section,  we study  the eigenvalues $\lambda_{p,l}$ and
 the eigenfunctions $v_{p,l}$ of the linearized problem \eqref{03-16-1}, for $l=1,\cdots,3k+1$, as $p\rightarrow +\infty$.
 Our main results are collected in:
 \begin{itemize}
 \item Proposition \ref{prop:ExpansionsFirstGroup}, for the case $1\leq l\leq k$;\vskip 0.1cm
\item Proposition \ref{propSecondoPezzo1} and Proposition \ref{propSecondoPezzo2}, for the case $k+1\leq l\leq 3k$;\vskip 0.1cm
\item Proposition \ref{prop:expansionLastGroup}, for the case $l=3k+1$.
\end{itemize}
 Theorem \ref{th1.1} can be deduced from them. This is done at the end of the section, where we compute the Morse index of the solution $u_{p}$, for $p$ large.

\vskip 0.2cm

Fixed  $r>0$ small such that, for $1\leq i\neq j\leq k$,
$$B_{4r}(x_{\infty,i})\subset \Omega~~\mbox{ and }~~B_{4r}(x_{\infty,i}) \bigcap B_{4r}(x_{\infty,j})=\emptyset,$$
where  $\big\{x_{\infty,1},\cdots,x_{\infty,k}\big\}\subset\R^{2k}$ is the set of concentration points as in Theorem A.
For each $1\leq i\leq k$, we set $\phi_i(x)=\phi(x-x_{{\infty,}i})$, for $\phi\in C^{\infty}_0\big(B_{3r}(0)\big)$ satisfying $\phi=1$ in $B_{2r}(0)$,
$\phi=0$ in $\mathbb R^2 \backslash B_{3r}(0)$, and $0\le \phi\le 1$.

\vskip 0.1cm

We define
\begin{eqnarray}
&&u_{p,i}:=\phi_iu_p,\qquad\psi_{p,i,q}:=\phi_i\frac{\partial u_p}{\partial x_q}\quad (q=1,2) \label{defpsipi}\label{defupi}.
\end{eqnarray}
Then we have following identities.
\begin{Lem}
It holds
\begin{equation}\label{6-25-43}
\begin{split}
 \int_{\Omega}\big|\nabla u_{p,i}\big|^2dx
=  \int_{\Omega}u_p^2|\nabla \phi_i|^2dx+\int_{\Omega}u^{p+1}_p\phi^2_idx.
\end{split}
\end{equation}
Furthermore, for any $\beta_{p,i},\gamma_{p,i}\in\mathbb R$,  setting $z_{p,i}:=\beta_{p,i}\psi_{p,i,1}+\gamma_{p,i}\psi_{p,i,2}$, one has
\begin{equation}\label{6-25-51}
\begin{split}
 \int_{\Omega}\big|\nabla z_{p,i}\big|^2dx
=  \int_{\Omega}s_{p,i}^2|\nabla \phi_i|^2dx+p \int_{\Omega}u^{p-1}_ps_{p,i}^2\phi^2_idx,
\end{split}
\end{equation}
and
\begin{equation}\label{6-25-52}
\begin{split}
 \int_{\Omega} \nabla  u_{p,i} \cdot \nabla z_{p,i} dx
=  \int_{\Omega}u_ps_{p,i}|\nabla \phi_i|^2dx
+\frac{p+1}{2}\int_{\Omega}u^{p}_ps_{p,i}\phi^2_idx,
\end{split}
\end{equation}
where $s_{p,i}:=\beta_{p,i}\frac{\partial u_p}{\partial x_1}+\gamma_{p,i}\frac{\partial u_p}{\partial x_2}$.

\end{Lem}
\begin{proof}
Multiplying $-\Delta u_p=u_p^p$ in $\Omega$ by $\phi^2_iu_{p}$ and integrating, we have
\[
 \int_{\Omega}  \phi^2_i  u_p^{p+1}dx=
 \int_{\Omega}\phi_i^2 \big|\nabla u_p\big|^{2}dx+ 2  \int_{\Omega}  \phi_i u_{p}\nabla \phi_i\cdot\nabla u_p  dx.
\]
Then  \eqref{6-25-43} follows observing that $|\nabla (\phi_{i}u_{p})|^{2}=u_{p}^{2}|\nabla \phi_{i}|^{2} + \phi_{i}^{2}|\nabla u_{p}|^{2} +2 \phi_{i}u_{p}\nabla\phi_{i}\cdot\nabla u_{p}$.
Observe that $s_{p,i}$ solves \begin{equation}\label{03-05-2}
-\Delta s_{p,i}=pu^{p-1}_ps_{p,i}~~\mbox{in}~~\Omega.
\end{equation}
Multiplying it by $\phi^2_is_{p,i}$ and integrating, we have
\[
\int_{\Omega} \phi^2_i \big|\nabla s_{p,i}\big|^2dx+
2 \int_{\Omega}\phi_i s_{p,i} \nabla \phi_i\cdot\nabla s_{p,i} dx=
 p \int_{\Omega} u^{p-1}_p \phi_i^2 s_{p,i}^2dx,
\]
then \eqref{6-25-51}  follows observing that $z_{p,i}=\phi_{i}s_{p,i}$ and that $$|\nabla (\phi_{i}s_{p,i})|^{2}=s_{p,i}^{2}|\nabla \phi_{i}|^{2} + \phi_{i}^{2}|\nabla s_{p,i}|^{2} +2 \phi_{i}s_{p,i}\nabla\phi_{i}\cdot\nabla s_{p,i}.$$
Multiplying \eqref{03-05-2} by $\phi^2_iu_p$ and integrating, we have
\[
p \int_{\Omega} \phi^2_i u_p^p s_{p,i}   dx= \int_{\Omega}\phi_i^2 \nabla u_p\cdot\nabla s_{p,i} dx+ 2  \int_{\Omega} u_p \phi_i \nabla \phi_i\cdot\nabla s_{p,i} dx.
\]
Also multiplying $-\Delta u_p=u_p^p$ in $\Omega$ by $\phi^2_is_{p,i}$ and integrating, we have
\[
 \int_{\Omega} \phi^2_i u_p^p s_{p,i}   dx=
 \int_{\Omega}\phi_i^2 \nabla u_p\cdot \nabla s_{p,i} dx+ 2  \int_{\Omega} s_{p,i} \phi_i \nabla \phi_i\cdot \nabla u_p  dx.
\]
Then it follows
\[
\begin{split}
\frac{p+1}{2} \int_{\Omega} \phi^2_i u_p^p s_{p,i}  dx= &\int_{\Omega}\phi_i^2 \nabla u_p\cdot\nabla s_{p,i} dx+    \int_{\Omega} u_p \phi_i \nabla \phi_i\cdot \nabla s_{p,i}dx
 + \int_{\Omega} s_{p,i} \phi_i \nabla \phi_i\cdot \nabla u_p  dx.
\end{split}\]
Since $\nabla (\phi_i u_p)\cdot\nabla (\phi_i s_{p,i}) =|\nabla \phi_i|^2u_ps_{p,i}+\phi_i^2 \nabla u_p\cdot \nabla s_{p,i}+
u_{p}\phi_i\nabla \phi_i\cdot\nabla s_{p,i} +s_{p,i} \phi_i \nabla \phi_i\cdot\nabla u_p
$, then one gets \eqref{6-25-52}.

\end{proof}

\subsection{The estimate of $\lambda_{p,l}$   in the case $1\leq l\leq k$.}
\begin{Prop}\label{prop:ExpansionsFirstGroup}
For $1\leq l\leq k$ it holds:
\begin{equation}\label{6-20-1}
\lambda_{p,l}\le \frac{1}{p}+O\Big(\frac{1}{p^2}\Big).
\end{equation}

\end{Prop}
\begin{proof}

Firstly, let $V$ be a vector space whose basis consists of $\{u_{p,i}:1\leq i\leq k\}$, where the functions $u_{p,i}$ have been defined in \eqref{defupi}, then we have $dim~V=k$. By the Courant-Fischer-Weyl min-max principle, we have
\begin{equation}\label{6-25-11}
\lambda_{p,k}=\min_{\begin{subarray}{lr}
W\subset H^1_0(\Omega)
\\
\dim W=k
\end{subarray}}
\max_{f\in W\backslash\{0\}}\frac{\displaystyle\int_{\Omega}|\nabla f|^2dx}{p\displaystyle\int_{\Omega}f^2u^{p-1}_pdx}
\leq \max_{f\in V\backslash\{0\}}\frac{\displaystyle\int_{\Omega}|\nabla f|^2dx}{p\displaystyle\int_{\Omega}f^2u^{p-1}_pdx}.
\end{equation}
Taking $f=\displaystyle\sum^k_{i=1}a_iu_{p,i}$ for some $(a_1,\cdots,a_k)\neq (0,\cdots,0)$, then the fact that $u_{p,i}$ and $u_{p,j}$ have disjoint supports for any $1\leq i\neq j\leq k$ implies
\begin{equation}\label{6-25-12}
\begin{split}
\frac{\displaystyle\int_{\Omega}|\nabla f|^2dx}{p \displaystyle\int_{\Omega}f^2u^{p-1}_pdx}
= &\frac{\displaystyle\sum^k_{i=1}\displaystyle\int_{\Omega}\big|\nabla (a_iu_{p,i})\big|^2dx}{p\displaystyle\sum^k_{i=1} \displaystyle\int_{\Omega}(a_iu_{p,i})^2u^{p-1}_pdx}
 \leq
\max_{1\leq i\leq k} \frac{ \displaystyle\int_{\Omega}\big|\nabla u_{p,i}\big|^2dx}{p \displaystyle\int_{\Omega}u_{p,i}^2u^{p-1}_pdx}
\\
= &
\max_{1\leq i\leq k} \frac{ \displaystyle\int_{\Omega}\big|\nabla u_{p,i}\big|^2dx}{p \displaystyle\int_{\Omega}\phi_{i}^2u^{p+1}_pdx}=
\frac{1}{p}+
\max_{1\leq i\leq k} \frac{ \displaystyle\int_{\Omega}|\nabla \phi_i|^2u_p^2dx}{p \displaystyle\int_{\Omega}\phi_{i}^2u^{p+1}_pdx},
\end{split}
\end{equation}
where we have used \eqref{6-25-43} to get the last equality.
On the other hand,  we have
\begin{equation}\label{6-25-13}
\begin{split}
\int_{\Omega}\phi_i^2u^{p+1}_p
 \ge \int_{B_{r}(x_{p,i})}u_{p}^{p+1}\ge \frac{c_0}p.
\end{split}
\end{equation}
Furthermore, by the definition of $\phi_i$ and \eqref{11-14-03N}, it holds
\begin{equation}\label{6-25-14}
\int_{\Omega}|\nabla \phi_i|^2u_p^2\leq C\int_{B_{3r}(x_{\infty,i})\backslash B_{2r}( x_{\infty,i})}u^2_p =O\Big(\frac{1}{p^2}\Big).
\end{equation}
Hence  from \eqref{6-25-11}, \eqref{6-25-12}, \eqref{6-25-13} and \eqref{6-25-14}, we deduce \eqref{6-20-1}.

\end{proof}

\vskip 0.3cm

\subsection{Asymptotic behavior of $\lambda_{p,l}$ and $v_{p,l}$  in the case $k+1\leq l\leq 3k$.}~

\vskip 0.2cm

We study the asymptotic behavior of the eigenpairs $(\lambda_{p,l},v_{p,l})$, for $k+1\leq l\leq 3k$. The main results of this section are collected in Proposition \ref{propSecondoPezzo1} and Proposition \ref{propSecondoPezzo2} below.

\begin{Lem}\label{lemma:primo}
For $k+1\leq l\leq 3k$, one has
\begin{equation}\label{newlambdalim}
\lambda_{\infty,l}\geq 1
\end{equation}
and \begin{equation}\label{03-05-1}
\lambda_{p,l}\leq 1+O\Big(\e^2_p\Big).
\end{equation}
\end{Lem}
\begin{proof}
 In order to obtain  \eqref{newlambdalim}, it is enough to prove that
\begin{equation}\label{tesilambda}
\lambda_{\infty,k+1}:=\lim_{p\to \infty}\lambda_{p,k+1}\geq 1.
\end{equation}
Suppose by contradiction  that
\begin{equation*}
\lambda_{\infty,k+1}<1.
\end{equation*}
Then by Lemma \ref{lem2-6}, the rescaled eigenfunctions $\widetilde v_{p,k+1,j}$, for  $j=1,\cdots, k$, converge to  solutions $V_{k+1,j}$ of the limit eigenvalue problem \eqref{3-23-1}.
Furthermore by Lemma \ref{lemma:limitEigenvalueProblem} we have $\lambda_{\infty,k+1}=0$.
Hence from Proposition \ref{prop3-3}-(1), we get  that $V_{k+1,j}= c_{k+1,j}\in \mathbb R$,  with  $\mathbb R^{k}\ni\textbf{c}_{k+1}:=(c_{k+1,1},\cdots, c_{k+1,k})\neq \textbf{0}$. Furthermore one has
\begin{equation*}
\textbf{c}_{k+1} \ \bot\ span \{\textbf{c}_{1},\cdots,\textbf{c}_{k}\}.
\end{equation*}
 Since $dim~~span \{\textbf{c}_{1},\cdots,\textbf{c}_{k}\}=k$, hence $span \{\textbf{c}_{1},\cdots,\textbf{c}_{k}\}=\mathbb R^{k}$ and so  we have a contradiction.
As a consequence we get \eqref{tesilambda}, and so  \eqref{newlambdalim} follows.

\vskip 0.3cm

The rest of the proof is devote to get \eqref{03-05-1}. In fact, it is enough to prove that \begin{equation}\label{317:vera}\lambda_{p,3k}\leq 1+O\big(\e^2_p\big).\end{equation}
We basically follow the proof of Proposition 3.1 in \cite{DGIP2019}, where the single-spike solution is considered. For reader's convenience, we give the proof of \eqref{317:vera} here (the multi-spike case). Specifically, we divide the proof of \eqref{317:vera} into three steps.

\vskip 0.3cm

\noindent {\sl Step 1. Using the variational characterization}, we get some basic estimate on $\lambda_{p,3k}$.

\vskip 0.3cm

By the variational characterization of the eigenvalues of the linearized problem, we have
\begin{equation*}
\lambda_{p,3k}=\min_{\begin{subarray}{lr}W\subset H^1_0(\Omega)\\\dim W=3k
\end{subarray}}\max_{v\in W\backslash\{0\}}\frac{\displaystyle\int_{\Omega}|\nabla v|^2dx}{p\displaystyle\int_{\Omega}u^{p-1}_pv^2dx}.
\end{equation*}
We define
\begin{equation*}
W:=span\Big\{u_{p,i},\psi_{p,i,1},\psi_{p,i,2},~i=1,\cdots,k\Big\},
\end{equation*}
where the functions $u_{p,i}, \psi_{p,i,1},\psi_{p,i,2}$ have been introduced in \eqref{defupi}.
Then we can verify that $dim ~W=3k$ and  any $v_{p}\in W$ can be written as
\begin{equation*}
v_p=f_{p} +g_{p},
\end{equation*}
where $f_p:=\displaystyle\sum^k_{i=1}\alpha_{p,i}u_{p,i}$, for  $\alpha_{p,i}\in \mathbb R$ and $g_p:=
\displaystyle\sum^k_{i=1}z_{p,i}$, with $z_{p,i}:=\beta_{p,i}\psi_{p,i,1}+\gamma_{p,i}\psi_{p,i,2}$, for  $\beta_{p,i},\gamma_{p,i}\in\mathbb R$. As a consequence
\begin{equation}\label{6-25-41}
\lambda_{p,3k}\leq \max_{\mathfrak{L}} \frac{\displaystyle\int_{\Omega}\big|\nabla (f_p+g_p)\big|^2dx}{p\displaystyle\int_{\Omega}u^{p-1}_p(f_p+g_p)^2dx},
\end{equation}
where $$\mathfrak{L}:=\Big\{\big(\alpha_{p,1},\cdots,\alpha_{p,k},\beta_{p,1},\cdots,\beta_{p,k},
\gamma_{p,1},\cdots,\gamma_{p,k}\big)\in \R^{3k},~~~\displaystyle\sum^k_{i=1}
\big(\alpha^2_{p,i}+\beta^2_{p,i}+\gamma^2_{p,i}\big)=1\Big\}.$$

Using \eqref{6-25-43}, we obtain
\begin{equation}\label{6-25-55}
\begin{split}
\int_{\Omega}\big|\nabla f_p \big|^2dx=&
\int_{\Omega}\big|\sum^k_{i=1}\alpha_{p,i}\nabla u_{p,i}\big|^2dx
=
\sum^k_{i=1}\alpha^2_{p,i} \int_{\Omega}\big|\nabla u_{p,i} \big|^2dx\\=&
\sum^k_{i=1}\alpha^2_{p,i} \int_{\Omega}\big|\nabla  \phi_i \big|^2u_p^2dx+\sum^k_{i=1}\alpha^2_{p,i}\int_{\Omega}\phi_i^2u_p^{p+1}dx.
\end{split}
\end{equation}
Setting  $s_{p,i}:=\beta_{p,i}\frac{\partial u_{p}}{\partial x_{1}}+\gamma_{p,i}\frac{\partial u_{p}}{\partial x_{2}}$, from \eqref{6-25-51}, we can find
\begin{equation}\label{6-25-56}
\int_{\Omega}  \big|\nabla g_{p}\big|^2dx= \sum^k_{i=1}  \Big( \int_{\Omega} \big|\nabla \phi_i\big|^2 s_{p,i}^2dx+
 p \int_{\Omega} \phi_i^2 u^{p-1}_p  s_{p,i}^2dx\Big).\end{equation}
Similarly, using \eqref{6-25-52}, we have
\begin{equation}\label{6-20-2}
\begin{split}
 \int_{\Omega} \nabla f_p \cdot \nabla g_p dx=&
 \sum^k_{i=1} \alpha_{p,i} \int_{\Omega} \nabla u_{p,i} \cdot \nabla z_{p,i} dx \\=
 &
 \sum^k_{i=1}\alpha_{p,i}\int_{\Omega} \big|\nabla \phi_i \big|^2
u_p  s_{p,i} dx+ \frac{p+1}{2}  \sum^k_{i=1}\alpha_{p,i}\int_{\Omega} \phi^2_i u_p^p s_{p,i}   dx .
\end{split}\end{equation}
From  \eqref{6-25-55}, \eqref{6-25-56} and  \eqref{6-20-2}, we then find
\begin{equation}\label{6-25-61}
\begin{split}
\int_{\Omega}\big|\nabla (f_p+g_p)\big|^2dx=&
 \underbrace{\sum^k_{i=1}\alpha^2_{p,i} \int_{\Omega}\phi_i^2u_p^{p+1}dx}_{:=J_1}
 +(p+1)\underbrace{\sum^k_{i=1}\alpha_{p,i}\int_{\Omega} u_p^p s_{p,i} \phi^2_i  dx}_{:=J_2}+p\underbrace{ \sum^k_{i=1}\int_{\Omega} u^{p-1}_p \phi_i^2 s_{p,i}^2dx}_{:=J_3}\\&
 +
 \underbrace{\sum^k_{i=1}\Big(\alpha^2_{p,i} \int_{\Omega}\big|\nabla  \phi_i \big|^2u_p^2dx+2\alpha_{p,i}\int_{\Omega} s_{p,i} u_p \big|\nabla \phi_i \big|^2 dx + \int_{\Omega} \big|\nabla \phi_i\big|^2 s_{p,i}^2dx\Big)}_{:=J_4}\\=&
 J_1+(p+1)J_2+pJ_3+J_4.
\end{split}\end{equation}
Furthermore, by direct computations, we have
\begin{equation}\label{6-25-62}
\begin{split}
\displaystyle\int_{\Omega}u^{p-1}_p(f_p+g_p)^2dx=&
\displaystyle\int_{\Omega}u^{p-1}_p\big(f^2_p+g^2_p+2f_pg_p\big)dx\\=&
\sum^k_{i=1}\alpha_{p,i}^2\int_{\Omega}u^{p+1}_p\phi_i^2dx+2\sum^k_{i=1}\alpha_{p,i}\int_{\Omega} u_p^p s_{p,i} \phi^2_i  dx+\sum^k_{i=1}\int_{\Omega} u^{p-1}_p \phi_i^2 s_{p,i}^2dx\\[1mm]=&
J_1+2J_2+J_3.
\end{split}\end{equation}
Hence from \eqref{6-25-41},  \eqref{6-25-61} and \eqref{6-25-62}, we obtain
\begin{equation}\label{6-25-87}
\begin{split}
 \lambda_{p,3k}\leq \max_{\mathfrak{L}} \frac{\displaystyle\int_{\Omega}\big|\nabla (f_p+g_p)\big|^2dx}{p\displaystyle\int_{\Omega}u^{p-1}_p(f_p+g_p)^2dx}=
 & \max_{\mathfrak{L}} \left(
 \frac{J_1+(p+1)J_2+pJ_3+J_4}{pJ_1+2pJ_2+pJ_3}\right)
 \\=&1+\max_{\mathfrak{L}} \left( \frac{(1-p)J_1+(1-p)J_2+J_4}{pJ_1+2pJ_2+pJ_3}\right).
\end{split}\end{equation}
\noindent {\sl Step 2. We compute the terms $J_1,\cdots,J_4$ in \eqref{6-25-87}.}

\vskip 0.2cm

Firstly, as in  \eqref{6-25-13} we have
\begin{equation}\label{6-25-81}
\begin{split}
J_1= &{\sum^k_{i=1}\alpha^2_{p,i} \int_{\Omega}\phi_i^2 u_p^{p+1}dx}
=
\frac{8\pi e}{p} \sum^k_{i=1} \alpha^2_{p,i} \Big(1+o(1)\Big).
\end{split}\end{equation}
Moreover,  using  \eqref{11-14-03N} and the divergence Theorem, we get
\begin{equation}\label{6-25-82}
\begin{split}
J_2= &\sum_{i=1}^{k} \alpha_{p,i} \int_{\Omega}\phi_i^2 u_p^{p}s_{p,i}dx
\\=
&\sum^k_{i=1}\alpha_{p,i}\int_{B_r(x_{p,i})}   u_p^p
 \Big(\beta_{p,i}\frac{\partial u_{p}}{\partial x_1}+\gamma_{p,i}\frac{\partial u_{p}}{\partial x_2} \Big)dx+ O\Big(\frac{C^{p+1}}{p^{p+1}}\Big)\\=&
-\frac{1}{p+1}\sum^k_{i=1}\alpha_{p,i}\int_{\partial B_r(x_{p,i})}
 \Big(\beta_{p,i}  u_p^{p+1}\nu_1 +\gamma_{p,i} u_p^{p+1}\nu_2 \Big)d\sigma+ O\Big(\frac{C^{p+1}}{p^{p+1}}\Big)\\=&
 O\Big(\frac{C^{p+1}}{p^{p+1}}\Big),
\end{split}\end{equation}
where $(\nu_1,\nu_2)$ is the unit outer normal vector to $\partial B_r(x_{p,i})$.
Next, from \eqref{11-14-03N}, \eqref{defwpj} and Fatou's lemma, it holds
\begin{equation}\label{6-25-83}
\begin{split}
J_3=
&
\sum_{i=1}^{k}  \int_{\Omega}\phi_i^2 u_p^{p-1}s_{p,i}^{2}dx
\\=
&\sum^k_{i=1} \int_{B_r(x_{p,i})}   u_p^{p-1}
 \Big(\beta_{p,i}\frac{\partial u_{p}}{\partial x_1}+\gamma_{p,i}\frac{\partial u_{p}}{\partial x_2} \Big)^2dx+ O\Big(\frac{C^{p+1}}{p^{p+1}}\Big)\\=&
 \sum^k_{i=1} \frac{u_{p}(x_{p,i})}{p^3\e_{p,i}^2}\int_{B_{\frac{r}{\e_{p,i}}}(0)}   \Big(1+\frac{w_{p,i}}{p}\Big)^{p-1}
 \Big(\beta_{p,i}\frac{\partial w_{p,i}}{\partial x_1}+\gamma_{p,i}\frac{\partial w_{p,i}}{\partial x_2} \Big)^2dx+ O\Big(\frac{C^{p+1}}{p^{p+1}}\Big)\\ \geq &
 \sum^k_{i=1} \frac{u_{p}(x_{p,i})}{p^3\e_{p,i}^2}\int_{\R^2}
 e^{U(x)}
 \Big(\beta_{p,i}\frac{\partial U(x)}{\partial x_1}+\gamma_{p,i}\frac{\partial U(x)}{\partial x_2} \Big)^2dx+ O\Big(\frac{C^{p+1}}{p^{p+1}}\Big)
 \\=&
  \sum^k_{i=1} \frac{u_{p}(x_{p,i})}{p^3\e_{p,i}^2} \left(\beta^2_{p,i}  \int_{\R^2}
 e^{U(x)} \Big|\frac{\partial U(x)}{\partial x_1}\Big|^2+
 \gamma^2_{p,i} \int_{\R^2}
 e^{U(x)} \Big|\frac{\partial U(x)}{\partial x_2}\Big|^2\right)
 + O\Big(\frac{C^{p+1}}{p^{p+1}}\Big)
 \\=&
\sum^k_{i=1} \frac{e\Big(\beta^2_{p,i} +\gamma^2_{p,i} \Big)}{2p^3\e_{p,i}^2}\left( \int_{\R^2} e^{U(x)}\big|\nabla U(x)\big|^2 dx +o(1)\right)+ O\Big(\frac{C^{p+1}}{p^{p+1}}\Big).
\end{split}
\end{equation}
Now we estimate $J_4$. Firstly, we have
\begin{equation}\label{6-25-71}
J_4= \sum^k_{i=1}  \int_{\Omega}\big|\nabla  \phi_i \big|^2\big(\alpha_{p,i} u_p+s_{p,i}\big)^2dx\geq 0.
\end{equation}
Also, by the definition of $\phi_i$ and \eqref{11-14-03N}, it holds
\begin{equation}\label{6-25-72}
\begin{split}
J_4= &\sum^k_{i=1}  \int_{\Omega}\big|\nabla  \phi_i \big|^2\big(\alpha_{p,i} u_p+s_{p,i}\big)^2dx\\
=&\sum^k_{i=1}  \int_{B_{3r}(x_{\infty,i})\backslash B_{2r}(x_{\infty,i})}\big|\nabla  \phi_i \big|^2\left(\alpha_{p,i} u_p+\beta_{p,i}\frac{\partial u_{p}}{\partial x_1}+\gamma_{p,i}\frac{\partial u_{p}}{\partial x_2}\right)^2dx =O\Big(\frac{1}{p^2}\Big).
\end{split}
\end{equation}
Hence from \eqref{6-25-71} and \eqref{6-25-72}, we have
\begin{equation}\label{6-25-84}
0\leq J_4\leq \frac{C}{p^2}.
\end{equation}

\vskip 0.2cm

\noindent {\sl Step 3. We complete the proof of \eqref{317:vera}.}

\vskip 0.2cm

Here we divide into two cases:

\vskip 0.2cm

\noindent\textup{(1)}
If $\Big(\displaystyle\sum^k_{i=1}\alpha^2_{p,i}\Big)p^2\rightarrow +\infty$, then from \eqref{6-25-81}, \eqref{6-25-82} and \eqref{6-25-84}, we have
\begin{equation*}
p^2\Big((1-p)J_1+(1-p)J_2+J_4\Big) \to -\infty,~~\mbox{as}~~p\to \infty,
\end{equation*}
which implies
\begin{equation}\label{6-25-85}
 (1-p)J_1+(1-p)J_2+J_4<0~~\mbox{for large}~~p.
\end{equation}
Also from \eqref{6-25-81}, \eqref{6-25-82} and \eqref{6-25-83}, we have
\begin{equation}\label{6-25-86}
\begin{split}
pJ_1+2pJ_2+p{J_3}\geq  &\sum^k_{i=1} \frac{e\Big(\beta^2_{p,i} +\gamma^2_{p,i} \Big)}{2p^2\e_{p,i}^2}\Big(\int_{\R^2} e^{U(x)}\big|\nabla U(x)\big|^2 dx +o(1)\Big)\\&+ 8\pi e \sum^k_{i=1} \alpha^2_{p,i} \Big(1+o(1)\Big)
+ O\Big(\frac{C^{p+1}}{p^{p}}\Big)>0.
\end{split}
\end{equation}
Hence from \eqref{6-25-85} and \eqref{6-25-86}, we can get
\begin{equation}\label{6-25-93}
\begin{split}
\frac{(1-p)J_1+(1-p)J_2+J_4}{pJ_1+2pJ_2+pJ_3}\leq 0.
\end{split}\end{equation}

\vskip 0.1cm

\noindent\textup{(2)}
If $\Big(\displaystyle\sum^k_{i=1}\alpha^2_{p,i}\Big)p^2\leq C$, then from
$\displaystyle\sum^k_{i=1}
\big(\alpha^2_{p,i}+\beta^2_{p,i}+\gamma^2_{p,i}\big)=1$, we have
\begin{equation*}
\sum^k_{i=1} \Big(\beta^2_{p,i} +\gamma^2_{p,i} \Big)\to 1,~~\mbox{as}~~p\to \infty.
\end{equation*}
So in this case, from \eqref{6-25-81}, \eqref{6-25-82}, \eqref{6-25-83} and \eqref{6-25-84}, we have
\begin{equation}\label{6-25-91}
 (1-p)J_1+(1-p)J_2+J_4=O\Big(\frac{1}{p^2}\Big),
\end{equation}
and
\begin{equation}\label{6-25-92}
pJ_1+2pJ_2+p{J_3}\geq \frac{C_1}
{p^2\e_{p}^2},~~\mbox{for some}~~C_1>0.
\end{equation}
Hence \eqref{6-25-91} and \eqref{6-25-92} imply
\begin{equation}\label{6-25-94}
\begin{split}
\frac{(1-p)J_1+(1-p)J_2+J_4}{pJ_1+2pJ_2+pJ_3}\leq C_0\varepsilon^2_p~~\mbox{for some}~~C_0>0.
\end{split}\end{equation}
Finally, from \eqref{6-25-87}, \eqref{6-25-93} and \eqref{6-25-94}, we can obtain
\eqref{317:vera}.

\end{proof}

\begin{Lem}\label{lemma:tildevo}
For $k+1\leq l\leq 3k$, it holds
\begin{equation}\label{6-24-1}
\lambda_{\infty,l}=1
\end{equation}
and
\begin{equation}\label{3-12-1}
\widetilde{v}_{p,l,j}(y)
=\sum^2_{q=1}\frac{a_{q,l,j}y_q}{8+|y|^2}+o(1),\,~\mbox{in}~C^1_{loc}(\R^2),
~~\mbox{as}~p\to +\infty,
\end{equation}
for some vectors $\textbf{a}_{l,j}:= \big(a_{1,l,j},a_{2,l,j}\big)$ in $\R^2$, for any $j=1,\cdots, k$, and
\begin{equation}\label{6-26-11xx}
\textbf{a}_{l,j}\cdot \textbf{a}_{l',j}=0,\qquad ~~\mbox{ for }~~ k+1\leq l,l'\leq 3k,\ l\neq l'.
\end{equation}
Furthermore for any $l=k+1,\cdots, 3k$ there exists $j\in\{1,\cdots,k\}$ such that
\begin{equation}\label{anozero}
\textbf{a}_{l,j} \neq \textbf{0}.\end{equation}
Moreover for any fixed $j=1,\cdots, k$ one can find $\ell_{1},\ell_{2}\in\{k+1,\cdots, 3k\}$, $\ell_{1}\neq\ell_{2}$ such that
\begin{equation}\label{eachPointGivesContribution2}
\begin{split}
&\textbf{a}_{\ell_{1},j}\neq 0,\,\,\,\textbf{a}_{\ell_{2},j}\neq 0.
\end{split}
\end{equation}
\end{Lem}
\begin{proof}
Firstly, \eqref{6-24-1} is a consequence of Lemma \ref{lemma:primo}. Next we derive \eqref{3-12-1} and \eqref{6-26-11xx}.
By \eqref{6-24-1} one can apply Proposition \ref{prop3-3}-(2) for any $l=k+1,\cdots,3k$, namely  there exists $(a_{1,l,j},a_{2,l,j},b_{l,j})\in\mathbb R^{3}$ such that
\[
\widetilde{v}_{p,l,j}(y)
=\sum^2_{q=1}\frac{a_{q,l,j}y_q}{8+|y|^2}+b_{l,j}\frac{8-|y|^2}{8+|y|^2}+o(1),\,~\mbox{in}~C^1_{loc}(\R^2),
~~\mbox{as}~p\to +\infty.
\]
and $\textbf{a}_{l,j}\cdot \textbf{a}_{l',j}+32b_{l,j}b_{l',j}=0$,
for any  $j=1,\cdots,k$, and $(a_{1,l,j},a_{2,l,j},b_{l,j})\neq (0,0,0)$ for some $j$. Next we show that \begin{equation}
\label{bzero}b_{l,j}=0,\quad\mbox{ for any }j=1,\cdots, k,\end{equation}
from which, \eqref{3-12-1} and \eqref{6-26-11xx} follow.

\vskip 0.1cm

Assume by contradiction that $b_{l,j_0}\neq 0$, then Proposition \ref{lemma:autofunz2casi} applies by  \eqref{6-24-1}, so we have from \eqref{1-25-01} that for large $p$,
\begin{equation*}
\lambda_{p,l}\geq 1+\frac{3}{p},
\end{equation*}
which is a contradiction with \eqref{03-05-1}. Hence \eqref{bzero} holds.\\

Lastly, we prove \eqref{eachPointGivesContribution2} assuming by contradiction that there exists $\widehat j\in\{1,\cdots, k\}$ such that \eqref{eachPointGivesContribution2} is not true. Then \eqref{anozero} implies that there must be (at least) an index $j_{\ast}\in\{1,\cdots, k\}$, $j_{\ast}\neq \widehat j$ and (at least) $3$ indexes  $l_{1},l_{2}, l_{3}\in \{k+1,\cdots, 3k\}$, $l_{i}\neq l_{h}$ for $i\neq h$, $i,h=1,2,3$, such that
\[\textbf{a}_{l_{1},j_{\ast}}\neq 0,\quad \textbf{a}_{l_{2},j_{\ast}}\neq 0, \quad \textbf{a}_{l_{3},j_{\ast}}\neq 0,\]
but this is not possible because of \eqref{6-26-11xx}.
\end{proof}

\vskip 0.4cm

Next we derive the asymptotic behavior of the eigenfunctions:
\begin{Prop}\label{propSecondoPezzo1}
For $k+1\leq l\leq 3k$, one has
\begin{equation}\label{1-25-02}
{v}_{p,l}(x)
=2\pi\sum^k_{m=1}\sum^2_{q=1} a_{q,l,m}\frac{\partial G}{\partial x_q} \big(x,x_{p,m}\big) \e_{p,m} +o\Big(\e_{p}\Big)~\,~\mbox{in}~
C^1_{loc}\Big(\overline{\Omega}\backslash \{x_{\infty,1},\cdots,x_{\infty,k}\}\Big),
\end{equation}
where the constants $a_{q,l,m}$  are the same as in Lemma \ref{lemma:tildevo} and $\varepsilon_{p}$ is defined in \eqref{defepsilon}.
\end{Prop}
\begin{proof}
Let us observe that  $\lambda_{\infty,l}=1$ by \eqref{6-24-1}, hence Proposition \ref{lemma:autofunz2casi} could be applied but, since $b_{l,j}=0$ for any $j=1,\cdots, k$ (see the proof of Lemma \ref{lemma:tildevo}), then  \eqref{3-12-03}  would reduce to
	\[v_{p,l}=o(\frac{1}{p})
	~\,~\mbox{in}~
	C^1_{loc}\Big(\overline{\Omega}\backslash \{x_{\infty,1},\cdots,x_{\infty,k}\}\Big),
\]
which is not enough to deduce \eqref{1-25-02}.  For this reason, differently from the case $1\leq l\leq k$, we need now
to go back to  the proof of \eqref{3-12-03} in order to improve it and  get a better expansion of $v_{p,l}$.

\vskip 0.3cm

We divide the proof of \eqref{1-25-02} into five steps.

\vskip 0.3cm

\noindent \emph{Step 1.  We prove the following preliminary expansion  of $v_{p,l}$:
\begin{equation}\label{3-12-03e}
v_{p,l}(x)= p \lambda_{p,l}\sum^k_{m=1}A_{p,l,m} G(x_{p,m},x)+ 2\pi \lambda_{p,l}\sum^k_{m=1}\sum^2_{q=1} a_{q,l,m}\frac{\partial G}{\partial x_q} \big(x_{p,m},x\big) \e_{p,m} +o\big(\e_{p}\big),
\end{equation}
in $C^1_{loc}\Big(\overline{\Omega}\backslash \{x_{\infty,1},\cdots,x_{\infty,k}\}\Big)$.}

\vskip 0.2cm

From Lemma \ref{lem2-9} and \eqref{6-24-1}, we have
\begin{equation}\label{6-26-11}
v_{p,l}(x)= p\lambda_{p,l}\sum^k_{m=1}A_{p,l,m} G(x_{p,m},x)+ p\lambda_{p,l} \sum^k_{m=1}\sum^2_{q=1} \frac{\partial G(x_{p,m},x) }{\partial x_q} B_{p,l,m,q}+
o\big( \varepsilon_{p} \big),
\end{equation}
in $C^1\Big(\overline{\Omega}\backslash \{x_{\infty,1},\cdots,x_{\infty,k}\}\Big)$.
From \eqref{3-16-4t} it follows that
\begin{equation}\label{5-8-14}
\begin{split}
B_{p,l,m,q}=&
\frac{2\pi}{p} a_{q,l,m}  \e_{p,m}+o\Big(\frac{\varepsilon_{p}  }{p}\Big).
\end{split}
\end{equation}
Hence from \eqref{6-26-11} and \eqref{5-8-14} we deduce \eqref{3-12-03e}.
\\\\
\noindent
\emph{Step 2. We show that
that for any $j=1, \cdots, k$,
\begin{equation}\label{3-14-3}
A_{p,l,j}= o\Big(\frac{\varepsilon_{p}}{p}\Big)+o\Big(\frac{|1-\lambda_{p,l}|}{p}\Big).
\end{equation}}We will use the Pohozaev identity \eqref{3-12-04}. For any $j=1,\cdots,k$ let $P_{j}$ be the quadratic form defined in \eqref{07-08-20}, using the expansions of $u_{p}$ and $v_{p,l}$ in  \eqref{luo-1} and \eqref{3-12-03e} respectively, and recalling that $A_{p,l,m}=o(\frac{1}{p^2})$ by \eqref{3-16-4s}    (since $b_{l,m}=0$ for any $m=1,\cdots, k$, see the proof of Lemma \ref{lemma:tildevo}) and that $C_{p,s}=\frac{8\pi\sqrt e}{p}+o(\frac{1}{p})$ (see \eqref{luoluo1}),
	we obtain
			\[\begin{split}
				P_{j}(u_{p},v_{p,l})
				=&  p\lambda_{p,l}
				\sum_{s=1}^k\sum_{m=1}^kC_{p,s}A_{p,l,m}P_{j}\Big(G(x_{p,s},x), G(x_{p,m},x) \Big)\\
				&+\frac{16\pi^2\sqrt e}{p} \lambda_{p,l} \sum^k_{s=1} \sum^k_{m=1}\sum^2_{q=1} a_{q,l,m}
				P_{j}\Big(G(x_{p,s},x), \partial_q G \big(x_{p,m},x\big)\Big) \e_{p,m} +o\Big(\frac{\e_p}{p}\Big).
			\end{split}
	\]
Hence from \eqref{1-1} and \eqref{abb1-1}, and recalling the expression of the derivatives of the Kirchhoff-Routh function $\Psi_{k}$ (see the Appendix), one has
\begin{equation}\label{6-26-21}
\begin{split}
P_{j}(u_{p},v_{p,l})
				=&
\Big(-\frac{C_{p,j}p\lambda_{p,l}}{2\pi}A_{p,l,j}\Big)+o\Big(\frac{\varepsilon_{p}}{p}\Big)
\\
&
+\frac{16\pi^2\sqrt e}{p} \lambda_{p,l}\sum^2_{q=1} a_{q,l,j} \left(\frac{1}{2}\frac{\partial \Psi_{k}\big(x_{p,1},\cdots,x_{p,k}\big)}{\partial y_{2j-2+q}} \e_{p,j}\right)+o\Big(\frac{\varepsilon_{p}}{p}\Big)
\\
=&
-\Big(4\sqrt e+o\big(1\big)\Big)A_{p,l,j}+o\Big(\frac{\varepsilon_{p}}{p}\Big),
\end{split}\end{equation}
where the last equality follows from \eqref{luoluo1}, the fact that $\lambda_{p,l}=1+o(1)$ and
Proposition 3.12 in \cite{GILY2021}.
\vskip 0.2cm

Now using \eqref{11-14-03N} and \eqref{3-12-03e}, we obtain
\begin{equation}\label{lostpiece}
\begin{split}
\int_{\partial B_{d}(x_{p,j})} u_p^p  v_{p,l} (x-x_{p,j})\cdot\nu
=O\Big(\frac{C^{p}}{p^{p-1}} (\sum_{m=1}^{k}|A_{p,l,m}|+\frac{\varepsilon_{p}}{p})\Big)=O\Big(\frac{1}{p}\sum^k_{m=1}\big|A_{p,l,m}\big|+\frac{\e_p}{p^2}\Big).
\end{split}\end{equation}
Also from \eqref{luo-1}, \eqref{6-21-41} and \eqref{3-12-03e}, we get
\begin{equation}\label{6-26-22}
\begin{split}
 \int_{B_{d}(x_{p,j})} u_p^p  v_{p,l}=&
 \frac{1}{p\lambda_{p,l}-1}\left(
\int_{\partial B_{d}(x_{p,j})} \frac{\partial u_p}{\partial \nu}  v_{p,l} -\int_{\partial B_{d}(x_{p,j})} \frac{\partial v_{p,l}}{\partial \nu}  u_p\right)\\=&
O\Big(\frac{1}{p}\sum^k_{m=1}\big|A_{p,l,m}\big|+\frac{\e_p}{p^2}\Big).
\end{split}\end{equation}
With the same computation as in \eqref{6-21-64} (and recalling that now  $b_{l,j}=0$), we have
\begin{equation}\label{6-26-23}
\int_{B_{d}(x_{p,j})} u_p^{p-1}v_{p,l} ( x-x_{p,j})\cdot\nabla u_p =
o\Big(\frac{1}{p^2}\Big).
\end{equation}
Hence substituting  \eqref{6-26-21}, \eqref{lostpiece}, \eqref{6-26-22} and \eqref{6-26-23} into the Pohozaev identity \eqref{3-12-04}, we find \eqref{3-14-3} and conclude the proof of {\sl Step 2.}\\\\
%
%
%
%
%

\noindent
\emph{Step 3. We get the following identity
\begin{equation}\label{3-14-1}
\begin{split}
 \frac{  (\lambda_{p,l}-1)}{\e_{p,[\frac{t+1}{2}]}}\Big(\frac{\sqrt{e}\pi}{3}{\alpha_t}+o(1)\Big) = 8\pi^2\sqrt{e} \sum^{2k}_{m=1}\frac{\partial^2\Psi_{k}(\boldsymbol{\mathrm x}_{\infty})}{\partial y_t\partial y_m} {\alpha_m}\e_{p,[\frac{m+1}{2}]}
+o\big( \e_p \big)+o\Big({p} \e_p\sum^k_{m=1}|A_{p,l,m}|\Big),
\end{split}
\end{equation}
 for  $t=1,2,\cdots,2k$, where
$\alpha_m:=
\begin{cases}
a_{1,l,\frac{m+1}{2}},~~&\mbox{if}~~m~\mbox{is odd},\\
a_{2,l,\frac{m}{2}},~~&\mbox{if}~~m~\mbox{is even}.
\end{cases}
$}
\vskip 0.4cm

\noindent
We will use the Pohozaev identity \eqref{3-12-04a}. For any $j=1,\cdots,k$ let $Q_{j}$ be the quadratic forms defined in \eqref{abd}. Using the expansions of $u_{p}$ and $v_{p,l}$ in  \eqref{luo-1} and \eqref{3-12-03e} respectively, and recalling that $C_{p,s}=\frac{8\pi\sqrt e}{p}+o(\frac{1}{p})$ (see \eqref{luoluo1}) and that $\lambda_{p,l}=1+o(1)$ by \eqref{6-24-1},
 we have
\begin{equation}\label{6-26-25}
\begin{split}
Q_j(u_p,v_{p,l})=&p\lambda_{p,l}\sum^k_{s,m=1}A_{p,l,m}Q_j\Big(G(x_{p,s},x),G(x_{p,m},x)\Big)C_{p,s}
\\&
+ \frac{16\pi^2\sqrt e}{p} \sum^k_{s,m=1}\sum^2_{q=1}a_{q,l,m}Q_j\Big(G(x_{p,s},x),\partial_qG(x_{p,m},x)\Big)\e_{p,m}
\\&+o\big(\frac{\e_p}{p}\big)+o\Big(\e_p\sum^k_{m=1}|A_{p,l,m}|\Big)
.
\end{split}\end{equation}
Using also \eqref{aluo1}, recalling the expression of the derivatives of $\Psi_{k}$ (see the Appendix)  and that $A_{p,l,m}=o(\frac{1}{p^2})$ (by \eqref{3-16-4s}) one can compute
\begin{equation}\label{Explanation}
\begin{split}
 & p\lambda_{p,l}\sum^k_{s,m=1}A_{p,l,m}Q_j\Big(G(x_{p,s},x),G(x_{p,m},x)\Big)C_{p,s}
\\=& o\Big(\frac{1}{p^2}\Big)\sum^k_{s,m=1}Q_j\Big(G(x_{p,s},x),G(x_{p,m},x)\Big)
=
o\Big(\frac{1}{p^2}\Big)\frac{\partial \Psi_{k} (x_{p,1},\cdots, x_{p,k})}{\partial y_{2j-2+i}}=o\Big(\frac{\varepsilon_p}{p}\Big),
\end{split}\end{equation}
where the last equality follows from Proposition 3.12 in \cite{GILY2021}, similarly as in
\eqref{6-26-21}.

\vskip 0.2cm

Next, by \eqref{aluo41} and recalling the expression of the second derivatives of $\Psi_{k}$ (see \eqref{SecondDer} in the Appendix), we can compute
\begin{equation}\label{6-26-26}
\begin{split}
& \sum^k_{s,m=1}\sum^2_{q=1}a_{q,l,m}Q_j\Big(G(x_{p,s},x),\partial_qG(x_{p,m},x)\Big) \e_{p,m}\\=&
\sum^2_{q=1} a_{q,l,j}\Big(-\frac{1}{2}\frac{\partial^2R(x_{p,j})}{\partial x_i\partial x_q} +\sum_{s\neq j}D^2_{x_ix_q}
 G(x_{p,s},x_{p,j}) \Big) \e_{p,j}
 +\sum^2_{q=1}\sum_{s\neq j}
 a_{q,l,s} D_{x_i}\partial_{q}
 G(x_{p,s},x_{p,j}) \e_{p,s}
 \\=&
 -\frac{1}{2}\sum^2_{q=1}
  \sum_{s=1}^{k}
 a_{q,l,s}  \frac{\partial^{2}\Psi_{k}(x_{p,1},\cdots, x_{p,k})}{\partial y_{2j-2+i}\partial y_{2s-2+q}}
 \e_{p,s}
 \\
 =&
- \frac{1}{2}\sum^{2k}_{m=1}\frac{\partial^2\Psi_{k}(\boldsymbol{\mathrm x}_{\infty})}{\partial y_t\partial y_m}\alpha_m \e_{p,[\frac{m+1}{2}]}+o\big(\e_{p}\big),
\end{split}
\end{equation}
where we have defined $t=2j+i {-2}$ and
$\alpha_m=
\begin{cases}
a_{1,l,\frac{m+1}{2}},~~&\mbox{if}~~m~\mbox{is odd},\\
a_{2,l,\frac{m}{2}},~~&\mbox{if}~~m~\mbox{is even},
\end{cases}
$ for $m=1,2,\cdots,2k$.

\vskip 0.2cm

Hence from \eqref{6-26-25}, \eqref{Explanation} and \eqref{6-26-26}, we have
\begin{equation}\label{6-26-27}
Q_j(u_p,v_{p,l})=
-\frac{8\pi^2\sqrt{e}}{p} \sum^{2k}_{m=1}\frac{\partial^2\Psi_{k}(\boldsymbol{\mathrm x}_{\infty})}{\partial y_t\partial y_m} \alpha_m\e_{p,[\frac{m+1}{2}]}
+o\big(\frac{\e_p}{p}\big)+o\Big(\e_p\sum^k_{m=1}|A_{p,l,m}|\Big).
\end{equation}
On the other hand, scaling and passing to the limit similarly as in \eqref{6-21-64},
we obtain
\begin{equation}\label{6-26-28}
\begin{split}
(\lambda_{p,l}-1)p\int_{B_r(x_{p,j})}u_p^{p-1}  v_{p,l} \frac{\partial u_p}{\partial x_i}\, dx
&=
(\lambda_{p,l}-1)
\frac{u_{p}(x_{p,j})}{\varepsilon_{p,j}p}\int_{ B_{\frac{r}{\varepsilon_j}(0)}}
\!\!\!\!\left( 1+\frac{w_{p,j}}{p}\right)^{p-1}\!\!\!\!\!\!\!\!\widetilde v_{p,l,j}\frac{\partial w_{p,j}}{\partial y_{i}}\, dy
 \\
&=
(\lambda_{p,l}-1)\frac{\sqrt e}{\varepsilon_{p,j}p}
\left(\int_{\mathbb R^{2}} e^{U}\sum_{q=1}^{2}\frac{a_{q,l,j}y_{q}}{8+|y|^{2}}\frac{\partial U}{\partial x_{i}} +o(1)\right)
 \\
&=
\frac{(\lambda_{p,l}-1)}{p\e_{p,j}}\Big(- \frac{\sqrt{e}\pi}{3}a_{i,l,j}+o(1)\Big),
\end{split}
\end{equation}
and, similarly as in \eqref{lostpiece}, using \eqref{11-14-03N}, \eqref{3-12-03e}, \eqref{3-14-3} and Step 2., we obtain
\begin{equation}\label{otherPiece}
\int_{\partial B_{r}(x_{p,j})} u_p^p  v_{p,l} \nu_id\sigma=O\Big(\frac{1}{p}\sum^k_{m=1}\big|A_{p,l,m}\big|\Big)
+o\Big(\frac{\e_p}{p}\Big)= o\Big(\frac{|1-\lambda_{p,l}|}{p^{2}}\Big)+o\Big(\frac{\e_p}{p}\Big).
\end{equation}
Thus, substituting \eqref{6-26-27}, \eqref{6-26-28} and \eqref{otherPiece} into the Pohozaev identity \eqref{3-12-04a}, we find
\eqref{3-14-1}, concluding the proof of {\sl Step 3.}

\vskip 0.4cm

\noindent \emph{Step 4. We show that
\begin{equation}\label{3-14-2}
	\lambda_{p,l}-1=O\Big(\e^2_p\Big)+o\Big(p\e_p^2\sum^k_{m=1}|A_{p,l,m}|\Big).
\end{equation}}

The proof follows from Step 3, observing that since there exists some $t_0\in \{1,2,\cdots,2k\}$ such that $\alpha_{t_0}\neq 0$, then \eqref{3-14-1} implies \eqref{3-14-2}.\\

\vskip 0.2cm

\noindent \emph{Step 5. We complete the proof of \eqref{1-25-02}.}

\vskip 0.2cm

Firstly, from \eqref{3-14-3}, using \eqref{3-14-2}, we have
	\begin{equation}\label{6-26-29}
		A_{p,l,m}= o\Big(\frac{\varepsilon_{p}}{p}\Big),~~\mbox{for}~~m=1,\cdots,k.
	\end{equation}
Then we  can deduce \eqref{1-25-02} by \eqref{3-12-03e} and \eqref{6-26-29}.

\end{proof}

\vskip 0.4cm
Next we derive the asymptotic expansion of the eigenvalues.

\begin{Prop} \label{propSecondoPezzo2}
For $k+1\leq l\leq 3k$, it holds
\begin{equation}\label{1-25-03}
\lambda_{p,l}=1+24\pi \e_{p}^{2}\theta_{l-k}+o\Big(\e^2_p\Big),
\end{equation}
where $\theta_1\leq \theta_2\leq \cdots\leq \theta_{2k}$ are the eigenvalues of the hessian matrix $\mathbf{H}  \Big( D^2 \Psi_k(\boldsymbol{\mathrm x}_{\infty}) \Big) \mathbf{H}$ with
\begin{equation*}
\mathbf{H}=diag
\Big(
e^{-2\pi\Psi_{k,1}(\boldsymbol{\mathrm x}_{\infty})},e^{-2\pi\Psi_{k,1}(\boldsymbol{\mathrm x}_{\infty})}, e^{-2\pi\Psi_{k,2}(\boldsymbol{\mathrm x}_{\infty})}, e^{-2\pi\Psi_{k,2}(\boldsymbol{\mathrm x}_{\infty})},\cdots ,\cdots,e^{-2\pi\Psi_{k,k}(\boldsymbol{\mathrm x}_{\infty})}, e^{-2\pi\Psi_{k,k}(\boldsymbol{\mathrm x}_{\infty})}
\Big),
\end{equation*} and $\varepsilon_{p}$ is defined in \eqref{defepsilon}.
\end{Prop}
\begin{proof}It is a consequence of \emph{Step 3} and \emph{Step 5} of the proof of the previous Proposition \ref{propSecondoPezzo1}. Indeed
 from \eqref{3-14-1} and combining  with \eqref{6-26-29}, we have
\begin{equation}\label{6-26-30}
\begin{split}
\frac{  (\lambda_{p,l}-1)}{\e_{p,[\frac{t+1}{2}]}} \Big(\alpha_t+o(1)\Big)
 = 24\pi  \sum^{2k}_{m=1}\frac{\partial^2\Psi_{k}(\boldsymbol{\mathrm x}_{\infty})}{\partial y_t\partial y_m} \alpha_m\e_{p,[\frac{m+1}{2}]}
+o\Big( \e_p \Big),
~~\mbox{for}~~t=1,2,\cdots,2k,
\end{split}
\end{equation}
which implies
\begin{equation}\label{6-26-31}
\begin{split}
 (\lambda_{p,l}-1) \big(\alpha_1,\cdots,\alpha_{2k}\big)^T
 = 24\pi \mathbf{H}_p \Big( D^2 \Psi_k(\boldsymbol{\mathrm x}_{\infty}) \Big) \mathbf{H}_p \big(\alpha_1,\cdots,\alpha_{2k}\big)^T
+o\Big( \e_p^2 \Big),
\end{split}
\end{equation}
where the matrix $\mathbf{H}_p:= diag \big(\e_{p,1},\e_{p,1},\e_{p,2},\e_{p,2},\cdots,\cdots,\e_{p,k},\e_{p,k} \big)$.
\vskip 0.2cm
Hence using \eqref{6-26-32}, we can write $\mathbf{H}_p= \e_{p} \mathbf{H}+o(\e_{p})$  with $\mathbf{H}$ defined by \eqref{6-26-30}.
And then \eqref{1-25-03} follows by \eqref{6-26-31}.
\end{proof}

\subsection{The estimate of $\lambda_{p,3k+1}$.}

\begin{Prop}\label{prop:expansionLastGroup}
It holds
\begin{equation}\label{03-08-1}
  \lambda_{p,3k+1}\geq 1+\frac{6}{p}\big(1+o(1)\big).
\end{equation}
\end{Prop}

\begin{proof}
If $\displaystyle\limsup_{p\to \infty}\lambda_{p,3k+1}>1$, then \eqref{03-08-1} holds. Now we prove that
\eqref{03-08-1} holds when $\displaystyle\limsup_{p\to \infty}\lambda_{p,3k+1}=1$.

\vskip 0.1cm

Now if $\displaystyle\limsup_{p\to \infty}\lambda_{p,3k+1}=1$, then  Proposition \ref{prop3-3}-(2) applies, namely
there exists $j_0\in\{1,\cdots, k\}$ and $(a_{1,3k+1,j_0},a_{2,3k+1,j_0},b_{3k+1,j_0})\in\mathbb R^{3}\neq (0,0,0)$  such that
\begin{equation*}
\widetilde{v}_{p,3k+1,j_0}(y)
=\sum^2_{q=1}a_{q,3k+1,j_0}\frac{y_q}{8+|y|^2}+b_{3k+1,j_0}\frac{8-|y|^2}{8+|y|^2}+o(1),\quad~\mbox{in}~C^1_{loc}(\R^2),
~~\mbox{as}~p\to +\infty.
\end{equation*}
Next we show that
\begin{equation}\label{aisZero}(a_{1,3k+1,j_0},a_{2,3k+1,j_0})=(0,0),
\end{equation}
 which implies that $b_{3k+1,j_0}\neq 0$. Hence \eqref{03-08-1} follows from \eqref{1-25-01}.

\vskip 0.1cm
 Indeed let us
observe that, since $\lambda_{\infty,\ell}=1$ for any $k+1\leq \ell\leq 3k$ (see \eqref{6-24-1} in Lemma \ref{lemma:tildevo}), then  \eqref{5-25-40} in
Proposition \ref{prop3-3}
holds, namely
$\big(a_{1,3k+1,j_0},a_{2,3k+1,j_0}\big)\cdot \big(a_{1,\ell,j_0},a_{2,\ell,j_0}\big)+24b_{3k+1,j_0}b_{\ell,j_0}=0$.
But recall that $b_{\ell,j_0}=0$ (see \eqref{bzero} in the proof of Lemma \ref{lemma:tildevo}), hence it follows that
\begin{equation*}
\big(a_{1,3k+1,j_0},a_{2,3k+1,j_0}\big)\cdot \big(a_{1,\ell,j_0},a_{2,\ell,j_0}\big)=0 ,\qquad ~~\mbox{ for }~~ k+1\leq \ell\leq 3k.
\end{equation*}

On the other hand, by \eqref{eachPointGivesContribution2}, we can find
$\ell_{1}\neq \ell_{2}$, $k+1\leq \ell_{1},\ell_{2}\leq 3k$, such that
$\big(a_{1,\ell_{1},j_0},a_{2,\ell_{1},j_0}\big)\ne 0$,
$\big(a_{1,\ell_{2},j_0},a_{2,\ell_{2},j_0}\big)\ne 0$,  and

\begin{equation*}
\big(a_{1,\ell_{1},j_0},a_{2,\ell_{1},j_0}\big)\cdot\big(a_{1,\ell_{2},j_0},
a_{2,\ell_{2},j_0}
\big)=0.
\end{equation*}
This gives \eqref{aisZero}.
\end{proof}

\vskip 0.2cm

\subsection{The Proof of Theorem \ref{th1.1}}
\begin{proof}[\bf{Proof of Theorem \ref{th1.1}}]
From Proposition \ref{prop:ExpansionsFirstGroup} and Proposition \ref{prop:expansionLastGroup}, we have that
\begin{equation*}
\lambda_{p,k}<1~\mbox{and}~\lambda_{p,3k+1}>1.
\end{equation*}
Hence, using also Proposition \ref{propSecondoPezzo2} for $k+1 \leq l\leq 3k$, we have
\begin{equation*}
 k+
\sharp \Big\{
\begin{subarray}{c}\mbox{ the negative eigenvalues }\\[1mm]
\mbox{of $\ \mathbf{H}  \big( D^2 \Psi_k(\boldsymbol{\mathrm x}_{\infty}) \big) \mathbf{H}$}
\end{subarray}
\Big\}
\leq\ \sharp \{l,~\lambda_{p,l}<1\} \ \leq k+
\sharp \Big\{ \begin{subarray}{c}\mbox{ the non-positive eigenvalues }\\[1mm]
\mbox{of $\ \mathbf{H}  \big( D^2 \Psi_k(\boldsymbol{\mathrm x}_{\infty}) \big) \mathbf{H}$}
\end{subarray}\Big\}
\end{equation*}
and the conclusion follows observing that
\begin{equation*}
\sharp \Big\{
\begin{subarray}{c}\mbox{ the negative eigenvalues }\\[1mm]
\mbox{of $\ \mathbf{H}  \big( D^2 \Psi_k(\boldsymbol{\mathrm x}_{\infty}) \big)  \mathbf{H}$}
\end{subarray}
\Big\}
=
\sharp \Big\{
\begin{subarray}{c}\mbox{ the negative eigenvalues }\\[1mm]
\mbox{of $\  D^2 \Psi_k(\boldsymbol{\mathrm x}_{\infty}) $}
\end{subarray}
\Big\}
=
m\Big(\Psi_k(\boldsymbol{\mathrm x}_{\infty})\Big).
\end{equation*}
and that
\begin{equation*}
\sharp \Big\{
\begin{subarray}{c}\mbox{ the non-positive eigenvalues }\\[1mm]
\mbox{of $\ \mathbf{H}  \big( D^2 \Psi_k(\boldsymbol{\mathrm x}_{\infty}) \big)  \mathbf{H}$}
\end{subarray}
\Big\}
=
\sharp \Big\{
\begin{subarray}{c}\mbox{ the non-positive eigenvalues }\\[1mm]
\mbox{of $\  D^2 \Psi_k(\boldsymbol{\mathrm x}_{\infty}) $}
\end{subarray}
\Big\}
=
m_{0}\Big(\Psi_k(\boldsymbol{\mathrm x}_{\infty})\Big).
\end{equation*}
\end{proof}

\begin{Rem}
Proposition \ref{prop:ExpansionsFirstGroup} and Proposition \ref{prop:expansionLastGroup} could  be improved getting
\begin{equation*}
\lambda_{p,l}=
\begin{cases}
  \frac{1}{p}+O\big(\frac{1}{p^2}\big),~~&\mbox{for}~1\leq l\leq k,\\[2mm]
  1+\frac{6}{p}\big(1+o(1)\big),~~&\mbox{for}~l=3k+1,
\end{cases}
\end{equation*}
as $p\rightarrow +\infty$. Furthermore the asymptotic behavior of the corresponding eigenfunctions could be easily derived. We omit these results since they go beyond the scope of this work.
\end{Rem}

\section{Improved asymptotic expansion of $u_{p}$}\label{Section:ImprovedAsymptoticup}

Here we show that any  multi-spike solution $u_p$ of problem \eqref{1.1} must have the form of the solutions constructed in
\cite{EMP2006}.
Firstly, we improve the asymptotic expansion of $u_{p}$ around each point $x_{p,j}$ (see Proposition \ref{prop:summaryImporvedAsympt}, and Remark \ref{remak:expansion:up} below). Then, we exploit this analysis to get the main result, which is collected in Proposition \ref{prop:upHasTheFormInEMP} below.

\vskip 0.1cm

Let $u_p$ be a $k$-spike solution of problem \eqref{1.1} and $\boldsymbol{\mathrm x_{\infty}}=(x_{\infty,1},\cdots,x_{\infty,k})\in\Omega^{k}$ be its  concentration point.
Let $w_{p,j}$ be the rescaled function of $u_{p}$  around $x_{p,j}$ as  defined in \eqref{defwpj}.
From \eqref{5-8-2} we know that as $p\rightarrow + \infty$,
\[
w_{p,j}=U+\frac{w_{0}}{p}+O\left(\frac{1}{p^{2}}\right)
\,\,\mbox{ in }C^{2}_{loc}(\mathbb R^{2}),\]
where $U$ is given  in \eqref{def:U}, and $w_0$ is a solution of \eqref{1-2-4}.
In \cite{GILY2021}, it does not show that $w_0$ is radially symmetric. In this
section,
  we
improve this expansion as follows.
\begin{Prop}\label{prop:summaryImporvedAsympt}
Let    $w_0$ be the  radial solution of the non-homogeneous linear equation
\begin{equation}\label{6-4-2}
-\Delta w_0-e^{U}w_0=-\frac{U^2}{2}e^{U}\quad\mbox{ in}~\R^2
\end{equation} satisfying $w_0(0)=0$, $w_{1}$ be
 the radial solution of the non-homogeneous linear equation
\begin{equation}\label{10-19-8}
-\Delta w_1-e^{U}w_1=\Bigl(\frac12  w_0^2
-\frac12    w_0U^2-    w_0U
    +\frac13   U^3   +\frac18    U^4\Bigr)e^U,\quad\mbox{ in}~\R^2,
\end{equation}
satisfying $w_1(0)=0$. Let
 \[	
w_{p,j}=U+\frac{w_{0}}{p}+\frac{w_{1}}{p^2}+\frac{q_{p,j}}{p^3}
.\]
Then for any small fixed $d_0,\tau_1>0$, there exists $C>0$ such that
\begin{equation}\label{1-23-8}
|q_{p,j}(x)|\leq C(1+|x| )^{\tau_1}~\mbox{ in}~ B_{\frac{d_0}{\e_{p,j}}}(0).
\end{equation}
As a consequence
\begin{equation}	\label{refinedExpansion}
w_{p,j}=U+\frac{w_{0}}{p}+\frac{w_{1}}{p^2}+O\left(\frac{1}{p^{3}}\right)
\,\,\mbox{ in }C^{2}_{loc}(\mathbb R^{2}),\end{equation}
as $p\rightarrow +\infty$.
\end{Prop}
\begin{proof} The proof of Proposition \ref{prop:summaryImporvedAsympt} is postponed to the next subsection. \end{proof}
\begin{Rem}\label{remak:expansion:up}
Recall that
\[
\e^{-2}_{p,j}=p \big(u_p(x_{p,j})\big)^{p-1},
\]
thus
\begin{equation}\label{nn1-26-8}
\frac{u_p(x_{p,j})}p=\frac1{ p^{\frac p{p-1}}\e_{p,j}^{\frac2{p-1}}}.
\end{equation}
So by the definition of $w_{p,j}$, together with \eqref{1-23-8}, \eqref{refinedExpansion} translates into the following expansion for $u_{p}$:
\begin{equation}\label{1-26-8}
u_p(x)=W_{p, j}(x)+O\left(\frac{1+\bigl( \frac{x-x_{p,j}}{\varepsilon_{p,j}}  \bigr)^{\tau_1}}{p^{4}}\right)
\,\,\mbox{ in }\; C^{2}_{loc} \big(B_{d_0}(x_{p,j})\big),
\end{equation}
for some small fixed $\tau_1>0$, as $p\rightarrow +\infty$, where
\begin{equation}\label{2-26-8}
W_{p, j}(x):=\frac1{ p^{\frac p{p-1}}\e_{p,j}^{\frac2{p-1}}} \left[p+U\bigl(\frac{x-x_{p,j}}{\e_{p,j} }\bigr)+\frac{w_{0}\bigl(\frac{x-x_{p,j}}{\e_{p,j} }\bigr)}{p}+\frac{w_{1}\bigl(\frac{x-x_{p,j}}{\e_{p,j} }\bigr)}{p^2}\right].
\end{equation}
\end{Rem}

\begin{Rem}\label{remarkEstimateswi}
 It is easy to prove (see for instance Lemma~2.1 in \cite{EMP2006}) that there are constants $C_0$ and $C_1$, such that
as $r\to +\infty$,
\begin{equation}\label{estimatewi}
w_i(r)= C_i \log r +O\bigl(\frac{\log^2 r }{r^2}\bigr),
\end{equation}
and
\begin{equation}\label{0-29-8}
w'_i(r)= \frac{C_i }r  +O\bigl(\frac{\log^2 r }{r^3}\bigr).
\end{equation}
This will be helpful in the next sections.
\end{Rem}

\

In order to state the main result of this section we need to introduce some notations.

\vskip 0.1cm

For any given $\boldsymbol{\mathrm \xi}=(\xi_1,\cdots, \xi_k)\in\Omega^{k}$ close to $\boldsymbol{\mathrm x}_{\infty}=(x_{\infty,1},\cdots,x_{\infty,k})$,
we solve the following equations to obtain $\boldsymbol{\bar\mu}_{p}=(\bar \mu_{p,1},\cdots,\bar\mu_{p,k})\in\mathbb R^{k}   $:
\begin{equation}\label{10-30-8}
\begin{split}
F_{p,j}(\boldsymbol{\bar\mu},\boldsymbol{\xi}):=&-\log (64 \bar\mu_j^4) -8\pi\bigl(1-\frac{C_0}{4p}- \frac{C_1}{4p^2} \bigr) H(\xi_j,\xi_{j})+\bigl(C_0 +\frac{C_1}{p}\bigr)\frac{\log (\bar\mu_je^{-\frac p4})}{p}
\\
&\,\,\,\,\,+8\pi\sum_{l\ne j}\frac{ \bar\mu_j^{\frac2{p-1}} }{\bar\mu_l^{\frac2{p-1} } }
 \bigl(1-\frac{C_0}{4p}- \frac{C_1}{4p^2} \bigr) G(\xi_j,\xi_{l})=0, \qquad j=1,\cdots,k,
\end{split}
\end{equation}
where $C_{i}$, $i=0,1$ are the constants in Remark \ref{remarkEstimateswi}.
Note that for $p>0$ large, we can apply the implicit function theorem to
uniquely solve \eqref{10-30-8}. Moreover, $(\bar \mu_{p,1},\cdots,\bar\mu_{p,k})   $
is a $C^1$ function of $(\xi_1,\cdots, \xi_k)$. A direct computation shows that for $p$ large, $\bar \mu_{p,j}$ satisfies
\begin{equation}\label{MissingEstimateBarMu}
\bar \mu_{p,j}=e^{-(2\pi\Psi_{k,j}(\boldsymbol{\mathrm \xi})+\frac{3}{4})}\left(1+O(\frac{1}{p}) \right),
\end{equation}
where $\Psi_{k,j}$ is defined in \eqref{stts}. In particular, for any fixed  $\tau_{0}>0$ there exists $c>0$ such that
$0<\frac{1}{c}\leq\bar \mu_{p,j}\leq c $,
for every $\xi_j\in B_{\tau_{0}}(x_{\infty,j})$, for $p$ large.
\\

For $j=1,\cdots,k$, we set
\begin{equation}\label{def:barepsilon}\bar\e_{p,j}:=\bar \mu_{p,j} e^{-\frac p4},\end{equation}
and define
\begin{equation}\label{12-30-8}
\bar{ W}_{ p, j}(x):=\frac1{ p^{\frac p{p-1}}\bar\e_{p,j}^{\frac2{p-1}}} \left[p+U\bigl(\frac{x-\xi_{j}}{\bar \e_{p,j} }\bigr)+\frac{w_{0}\bigl(\frac{x-\xi_{j}}{\bar\e_{p,j} }\bigr)}{p}+\frac{w_{1}\bigl(\frac{x-\xi_{j}}{\bar\e_{p,j} }\bigr)}{p^2}\right].
\end{equation}
where $U$, $w_{0}$ and $w_{1}$ are as in Proposition \ref{prop:summaryImporvedAsympt}.
For $\boldsymbol{\mathrm\alpha}:=(\alpha_{1},\cdots,\alpha_{k})\in \mathbb R^{k}$ we also define
\begin{equation}\label{def:Wpalpha}
\bar W_{\boldsymbol{\mathrm\alpha},p}:=\sum_{j=1}^k  \alpha_j P\bar W_{p, j},
\end{equation}
where for $v\in C^2(\bar\Omega)$, we define its \emph{projection}
 $P v\in H_0^1(\Omega)$ as the solution of
\begin{equation}\label{defProjection}
\Delta (P v)=\Delta  v\quad \text{in}\;\Omega.
\end{equation}
We consider the set

\begin{equation}\label{35-30-8}
S_{p}=\left\{ (\boldsymbol{\mathrm\alpha}, \boldsymbol{\xi}, \omega)\in \mathbb R^{k}\times \Omega^{k}\times \Big( H^{1}_{0}(\Omega)\cap L^{\infty}(\Omega)\Big) \ \ : \ \
\begin{split}
&|\alpha_j-1|\le \frac{1}{p^{\frac52}},\; |\xi_j-x_{\infty, j}|\le \tau_0 ,\\
&
\|\omega\|\le \frac1{p^{2+\tau_0}},~~~\|\omega\|_{L^{\infty}(\Omega)}\le \frac1{p^{2+ \tau_0}},
\\
&\bigl\langle  P\bar W_{p, j}, \omega\bigr\rangle = \bigl\langle  \frac{
\partial  P\bar W_{p, j}}{\partial {\xi_{j,h}}}, \omega\bigr\rangle=0,
\\ & j=1,\cdots, k,\; h=1,2,
\end{split}
\right\},
\end{equation}
where $\tau_0>0$ is any fixed small constant. Then, we have
\begin{Prop}\label{prop:upHasTheFormInEMP}
Let $u_p$ be a $k$-spike solution of problem \eqref{1.1} and let  $\boldsymbol{\mathrm x_{\infty}}:=(x_{\infty,1},\cdots,x_{\infty,k})\in\Omega^{k}$, $x_{\infty,i}\neq x_{\infty,j}$ for $i\neq j$, be its  concentration point.
Then for $p$ sufficiently large,
there exists $(\boldsymbol{\mathrm\alpha}_{p}, \boldsymbol{\xi}_{p}, \omega_{p})$  in the interior of $  S_{p}$,
 such that
\begin{equation}\label{36-30-8}
u_p=\bar W_{\boldsymbol{\mathrm\alpha_{p}},p}+\omega_{p}.
\end{equation}

\end{Prop}
\begin{proof} The proof of Proposition \ref{prop:upHasTheFormInEMP} is postponed to  Section \ref{subsection:ProofImprovedFormInEMP}. \end{proof}
\begin{Rem}
\eqref{36-30-8} tells us that any multi-spike solution $u_{p}$ of problem \eqref{1.1} has the form of the solutions constructed in \cite{EMP2006}. We will use \eqref{36-30-8} to compute the degree of all the solutions of problem \eqref{1.1} concentrating at $(x_{\infty,1},\cdots, x_{\infty,k})$ (see Section \ref{section:degreeAndUniqueness}).
\end{Rem}

\

\subsection{The proof of Proposition \ref{prop:summaryImporvedAsympt}}\label{subsection:betterAsymptu}$\,$

\vskip 0.2cm

The proof of Proposition \ref{prop:summaryImporvedAsympt} is a direct consequence of Lemma \ref{prop3-2}, Lemma \ref{blem3-1a} and Lemma \ref{p1-23-8} below.
\begin{Lem}\label{prop3-2}
Let
\begin{equation}\label{dsa}
v_{p,j}:=p\big(w_{p,j}-U\big).
\end{equation}
Then for any small fixed $d_0,\tau>0$, there exists $C>0$ such that
\begin{equation}\label{bbbvpj}
|v_{p,j}(x)|\leq C(1+|x| )^{\tau}~\mbox{ in}~ B_{\frac{d_0}{\e_{p,j}}}(0),
\end{equation}
and it holds
\begin{equation}\label{convergencevpjtow0}
\lim_{p\rightarrow +\infty} v_{p,j}=w_0 ~\mbox{in}~C^2_{loc}(\R^2),
\end{equation}
where
 $w_0$ is as in Proposition \ref{prop:summaryImporvedAsympt}, furthermore
\begin{equation*}
| w_{0}(x)|\leq C(1+|x| )^\tau.
\end{equation*}
\end{Lem}
\begin{proof}
These results improve the ones in   \cite[Proposition~3.6]{GILY2021}, we keep the same notations as in \cite{GILY2021}. \\
In particular the uniform bound \eqref{bbbvpj} is equivalent to
\[
N_{p}:=\max_{|x|\leq \frac{d_0}{\e_{p,j}}}\frac{| v_{p,j}(x)|}{(1+|x|)^\tau}\le C,
\]
which is proved in \cite[Proposition 3.5]{GILY2021}. The convergence \eqref{convergencevpjtow0} to a solution  $w_0$ of \eqref{6-4-2} then  follows from it,  passing
to the limit into the equation satisfied by $v_{p,j}$ (see \cite[Proposition~3.6]{GILY2021}).
The fact that $w_{0}$ is radially symmetric is instead the novelty here, and it is a consequence of the fact that
\begin{equation}\label{novelty}
N_p^*:=\max_{|x|\leq \frac{d_0}{\e_{p,j}}}\max_{|x'|=|x|}\frac{|v_{p,j}(x)-v_{p,j}(x')|}{(1+|x|)^{\tau}}\to 0, \mbox{ as }p\to +\infty.
\end{equation}
%
Next we prove \eqref{novelty}. We adapt the arguments in  Step 1. in  the proof of \cite[Proposition~3.5]{GILY2021}.
We argue by contradiction, assuming that
there exists $c_0>0$ such that $N^*_p\geq c_0$.
Let $x'_p$ and
$x''_p$ satisfy $|x'_p|=|x''_p| {\leq\frac{d_{0}}{\varepsilon_{p,j}}}$ and
\begin{equation*}
N_p^*=  \frac{|v_{p,j}(x'_p)-v_{p,j}(x''_p)|}{(1+|x'_p|)^{\tau}}.
\end{equation*}
Without loss of generality, we may assume that $x'_p$ and
$x''_p$ are symmetric with respect to the $x_1$ axis. Set
\begin{equation*}
\omega^*_p(x):=v_{p,j}(x)-v_{p,j}(x^-),~\,~x^-:=(x_1,-x_2),\  \mbox{ for }x=(x_{1},x_{2}),  ~\,x_2>0.
\end{equation*}
 Hence $x^{''}_{p}={x^{'}_{p}}^{-}$ and
$N_p^*=\frac{ \omega^*_p(x'_{p}) }{ (1+|x'_{p}| )^{\tau} }$.
Let us define
\begin{equation*}
\omega_p(x):=\frac{\omega^*_p(x)}{(1+x_2)^{\tau}}.
\end{equation*}
Also let $x^{**}_p$ satisfy $|x^{**}_p|\leq\frac{d_{0}}{\varepsilon_{p,j}}$, $x^{**}_{p,2}\geq 0$ and
\begin{equation*}
|\omega_p(x^{**}_p)|= N^{**}_p:=\max_{|x|\leq \frac{d_0}{\e_{p,j}},~x_2\geq 0} |\omega_p(x)|.
\end{equation*}
Then it follows
\begin{equation*}
N^{**}_p\geq \frac{|v_{p,j}(x'_p)-v_{p,j}(x''_p)|}{(1+|x'_p|)^{\tau}}=N^*_p\ge c_0>0.
\end{equation*}
Hence, arguing in exactly the same way as in Step 1
of  the proof of \cite[Proposition~3.5]{GILY2021} (see Part 1 and Part 2),  we can prove that
\begin{equation}\label{1h}
|x^{**}_p|\leq C.
\end{equation}
%
%
Now let \[\omega_p^{**}(x):=\frac{\omega_p(x)}{N_p^{**}}.\]
Then  $\omega_p^{**}(x)$ solves
\begin{equation*}
-\Delta \omega^{**}_p-\frac{2\tau}{1+x_2}\frac{\partial \omega^{**}_p}{\partial x_2}+
\frac{\tau(1-\tau)}{(1+x_2)^2}\omega^{**}_p=\frac{g_p(x)}{N_p^{**}(1+x_2)^\tau}+
\frac{\widetilde{{g}}_p(x)}{N_p^{**}(1+x_2)^\tau},
\end{equation*}
where
\begin{equation*}
\widetilde{g}_p(x):=
p\Big(e^{w_{p,j}(x)}-e^{w_{p,j}(x^-)}\Big),~\quad\mbox{for}~~|x|\leq \frac{d_0}{\e_{p,j}},
\end{equation*}
$g_p(x):=g^*_p(x)-g^*_p(x^-)$, and $g^*_p$ is defined as
\begin{equation}\label{St}
g^*_p(x):=p e^{w_{p,j}}
\Big(e^{p\log \big(1+\frac{w_{p,j}}{p}\big)-w_{p,j}}-1\Big),~\quad\mbox{for}~~|x|\leq \frac{d_0}{\e_{p,j}}.
\end{equation}
Moreover $|\omega^{**}_{p}(x)|\leq 1$ and
$|\omega^{**}_{p}(x^{**}_p)|=1$. Hence $\omega^{**}_p(x)\to \omega$ uniformly
in any compact subset of $\R^2$. And   $\omega\not\equiv0$ because $\omega_p^{**}(x_{p}^{**})=1$ by definition and  $|x_{p}^{**}|\le C$ by \eqref{1h}.
We now prove that
\begin{equation*}
\frac{g_p(x)}{N_p^{**}(1+x_2)^{\tau}}\to 0~\mbox{uniformly in any compact set of}~\big\{|x|\leq \frac{d_0}{\e_{p,j}}\big\}~\mbox{as}~p\to +\infty.
\end{equation*}
In fact, by \eqref{St} and using the Taylor expansion, we see
\begin{equation}\label{2-19-8}
g^*_p(x)=-\frac12 e^{ w_{p,j} } w^2_{p,j}+ O\bigl( \frac{ e^{ w^3_{p,j} } w_{p,j} } p \bigr).
\end{equation}
This gives that for $|x|\le R$,
\begin{equation}\label{3-19-8}
g_p(x)= O\bigl(| w_{p,j}(x)-w_{p,j}(x^-)|\bigr)+O\bigl(\frac{ 1 } p \bigr)
= O\bigl(\frac1p |\omega_p^*(x)|\bigr)+O\bigl(\frac{ 1 } p \bigr).
\end{equation}
So we see that
\begin{equation*}
\frac{|g_p(x)|}{N_p^{**}(1+x_2)^{\tau}}\le \frac Cp |\omega_p^{**}(x)|+\frac{ C } p\to 0.
\end{equation*}
Finally  we can argue exactly as at the end of the proof of Step 1 of  \cite[Proposition~3.5]{GILY2021} to get that $\bar \omega:=(1+x_{2})^{{\tau}} \omega=0$. Hence we obtain a contradiction (see page~164 in \cite{GILY2021} for more details). This concludes the proof of \eqref{novelty}.
\end{proof}

\begin{Lem}\label{blem3-1a}
Let \begin{equation}\label{sat}
k_{p,j}:=p\big(v_{p,j}-w_0\big),
\end{equation}
where $v_{p,j}$ is defined in \eqref{dsa} and $w_{0}$ is its limit function as $p\rightarrow +\infty$ (see \eqref{convergencevpjtow0}). Then for any small fixed $d_0,\tau_1>0$, there exists $C>0$ such that
\begin{equation}\label{blpy1}
|k_{p,j}(x)|\leq C(1+|x| )^{\tau_1}~\mbox{ in}~ B_{\frac{d_0}{\e_{p,j}}}(0),
\end{equation}
and it holds
\begin{equation}\label{convergencekpjtow1}
\lim_{p\rightarrow +\infty} k_{p,j}=w_1 ~\mbox{in}~C^2_{loc}(\R^2),
\end{equation}
where
 $w_1$  is as in Proposition \ref{prop:summaryImporvedAsympt}.
\end{Lem}
\begin{proof}
This result is an improvement of Proposition~3.9 in \cite{GILY2021}.
We keep the same notations as in \cite{GILY2021}. \\
In particular the uniform bound \eqref{blpy1} is already in \cite[Proposition~3.9]{GILY2021}. The convergence \eqref{convergencekpjtow1} to a solution  $w_1$ of \eqref{10-19-8} is instead the novelty here. This follows from
the bound \eqref{blpy1}, passing
to the limit into the equation satisfied by $k_{p,j}$ and using the convergences in \eqref{5-8-2} and in \eqref{convergencevpjtow0}.  Indeed
for
 any small fixed $d_0>0$, it holds
\begin{equation}
\label{b5-7-33}
\begin{split}
-\Delta k_{p,j}=  k_{p,j}e^{U}+  h^*_{1,p}+h^*_{2,p}~\,\,\,\mbox{ in}~~\Big\{|x|\leq \frac{d_0}{\e_{p,j}}\Big\},
\end{split}
\end{equation}
where
\begin{equation*}
h^*_{1,p}:=p\Big(p\big(e^{w_{p,j}}-e^{U} \big)-e^{U}w_0-\big( v_{p,j}-w_0\big) e^{U}\Big),
\,\,\,\,\,
h^*_{2,p}:=p\Big(g^*_p+\frac{U^2}{2}e^{U}\Big),
\end{equation*}
and $g^*_p$ is the same as in   \eqref{St}
(see the direct computations in the proof of \cite[Lemma~3.8]{GILY2021}).
Furthermore,  by Taylor expansion and \eqref{convergencevpjtow0} we have that
\begin{equation}\label{11-22-8}
\begin{split}
h^*_{1,p}
=& e^U\bigl(\frac12 v_{p,j}^2+O(\frac1p)\bigr)\to \frac12 e^U w^2_0,
\end{split}
\end{equation}
and
\begin{equation}\label{12-22-8}
\begin{split}
h^*_{2,p}
=&p\Big(pe^{ w_{p,j} }\bigl(e^{ p\log ( 1+\frac{  w_{p,j}  }p) - w_{p,j}  } -1   \bigr)+\frac{U^2}{2}e^{U}\Big)\\
=&p\Big(pe^{ w_{p,j} }\bigl(e^{-\frac12  \frac{  w^2_{p,j}  }p+\frac13  \frac{  w^3_{p,j}  }{p^2} +O(\frac1{p^3})} -1   \bigr)+\frac{U^2}{2}e^{U}\Big)\\
=&p\Big(pe^{ w_{p,j} }\bigl(-\frac12  \frac{  w^2_{p,j}  }p+\frac13  \frac{  w^3_{p,j}  }{p^2}+\frac12 (-\frac12  \frac{  w^2_{p,j}  }p+\frac13  \frac{  w^3_{p,j}  }{p^2})^2 + O(\frac1{p^3})
   \bigr)+\frac{U^2}{2}e^{U}\Big)\\
  =&p\Big(e^{ w_{p,j} }\bigl(-\frac12    w^2_{p,j}  +\frac13  \frac{  w^3_{p,j}  }{p}+\frac18  \frac{  w^4_{p,j}  }{p} + O(\frac1{p^2})
   \bigr)+\frac{U^2}{2}e^{U}\Big) \\
   =&p\Big(-\frac12e^{ w_{p,j} }    w^2_{p,j}
   \bigr)+\frac{U^2}{2}e^{U}\Big) +\frac13   w^3_{p,j}  +\frac18    w^4_{p,j}  + O(\frac1{p})\\
   =&-\frac12e^{ U}    v_{p,j}U^2-e^{ U}    v_{p,j}U
    +\frac13   w^3_{p,j}  +\frac18    w^4_{p,j}  + O(\frac1{p})\\
    \to & \Bigl(-\frac12    w_0U^2-    w_0U
    +\frac13   U^3  +\frac18    U^4\Bigr)e^{ U},
\end{split}
\end{equation}
where the last convergence follows from \eqref {5-8-2} and again from \eqref{convergencevpjtow0}.
Combining \eqref{b5-7-33}, \eqref{11-22-8} and \eqref{12-22-8}
we find that $w_1$ satisfies \eqref{10-19-8}.

\vskip 0.1cm

Finally we prove that $w_{1}$ is radially symmetric. This is a  consequence of the fact that
\begin{equation}\label{M_p^*to0}
M_p^*:=\max_{|x|\leq \frac{d_0}{\e_{p,j}}}\max_{|x'|=|x|}\frac{|k_{p,j}(x)-k_{p,j}(x')|}{(1+|x|)^{{\tau_1}}}\to 0, \quad\mbox{ as }p\to +\infty.
\end{equation}
Next we prove \eqref{M_p^*to0} by contradiction, similarly as in the proof of \eqref{novelty}.
Suppose that  there exists $c_0>0$ such that $M^*_p\geq c_0$. Let $x'_p$ and
$x''_p$ satisfy $|x'_p|=|x''_p| {\leq\frac{d_{0}}{\varepsilon_{p,j}}}$ and
\begin{equation*}
M_p^*=  \frac{|k_{p,j}(x'_p)-k_{p,j}(x''_p)|}{(1+|x'_p|)^{{\tau_1}}}.
\end{equation*}
Without loss of generality, we may assume that $x'_p$ and
$x''_p$ are symmetric with respect to the $x_1$ axis. Set
\begin{equation*}
l^*_p(x):=k_{p,j}(x)-k_{p,j}(x^-),~\,~x^-:=(x_1,-x_2),\  \mbox{ for }x=(x_{1},x_{2}),  ~\,x_2>0.
\end{equation*}
 Hence $x^{''}_{p}={x^{'}_{p}}^{-}$ and
$M_p^*=\frac{ l^*_p(x'_{p}) }{ (1+|x'_{p}| )^{\tau_1} }$.
Let us define
\begin{equation}\label{assfd}
l_p(x):=\frac{l^*_p(x)}{(1+x_2)^{\tau_1}}.
\end{equation}
Then direct calculations show that $l_p$ satisfies
\begin{equation}\label{fsd}
-\Delta l_p-\frac{2{\tau_1}}{1+x_2}\frac{\partial l_p}{\partial x_2}+
\frac{{\tau_1}(1-{\tau_1})}{(1+x_2)^2}l_p=\frac{h_p(x)}{(1+x_2)^{\tau_1}}+
e^{U(x)}l_p,
\end{equation}
where $h_p(x):=\big(h^*_{1,p}(x)-h^*_{1,p}(x^-)\big)+\big(h^*_{2,p}(x)-h^*_{2,p}(x^-)\big)$.

\vskip 0.1cm

Also let $y^{**}_p$ satisfy
\begin{equation}\label{assf1d}
\mbox{$|y^{**}_p|\leq\frac{d_{0}}{\varepsilon_{p,j}}$, $y^{**}_{p,2}\geq 0\ $ and }~~~ |l_p(y^{**}_p)|= M^{**}_p:=\max_{|x|\leq \frac{d_0}{\e_{p,j}},~x_2\geq 0} |l_p(x)|.
\end{equation}
Then it follows
\begin{equation*}
M^{**}_p\geq \frac{|k_{p,j}(x'_p)-k_{p,j}(x''_p)|}{(1+|x'_p|)^{\tau_1}}=M^*_p\ge c_0>0.
\end{equation*}
Using the radial symmetry of $w_0$, we can argue in exactly the same
way as in the proof of Proposition~3.9 in \cite{GILY2021} that
$|y^{**}_p|\le C$.

\vskip 0.1cm

Now let \[l_p^{**}(x):=\frac{l_p(x)}{M_p^{**}},\]
where $l_p$ is defined in \eqref{assfd} and $M_p^{**}$ is defined in \eqref{assf1d}.
From \eqref{fsd}, it follows that $l_p^{**}(x)$ solves
\begin{equation*}
-\Delta l^{**}_p-\frac{2\tau_1}{1+x_2}\frac{\partial l^{**}_p}{\partial x_2}+
\frac{\tau_1(1-\tau_1)}{(1+x_2)^2}l^{**}_p=\frac{h_p(x)}{M_p^{**}(1+x_2)^{\tau_1}}
+e^{U(x)}l^{**}_p.
\end{equation*}
Moreover $|l^{**}_{p}(x)|\leq 1$ and
$|l^{**}_{p}(y^{**}_p)|=1$, hence $l^{**}_p(x)\to \gamma(x)$ uniformly
in any compact subset of $\R^2$.
We now prove that
\begin{equation}\label{6-22-8}
\frac{h_p(x)}{M_p^{**}(1+x_2)^{\tau_1}}\to 0~\mbox{uniformly in any compact set of}~\big\{|x|\leq \frac{d_0}{\e_{p,j}}\big\}~\mbox{as}~p\to +\infty.
\end{equation}
In fact, using $w_{p,j}= U +\frac{v_{p,j}}p$, we find
\[
\begin{split}
h^*_{1,p}(x)-h^*_{1,p}(x^-)=&p\Big(p\big(e^{w_{p,j}(x)}- e^{w_{p,j}(x^-)}\big)-\big( v_{p,j}(x)-v_{p,j}(x^-)\big) e^{U}\Big)=O\Big(\frac{e^U v^2_{p,j}}p\Big).
\end{split}
\]
So we obtain
\begin{equation}\label{7-22-8}
\frac{h^*_{1,p}(x)-h^*_{1,p}(x^-)}{M_p^{**}(1+x_2)^{\tau_1}}\to 0~\mbox{uniformly in any compact set of}~\big\{|x|\leq \frac{d_0}{\e_{p,j}}\big\}~\mbox{as}~p\to +\infty.
\end{equation}
Similarly, using \eqref{2-19-8} and \eqref{3-19-8}, we find
\[
\begin{split}
h^*_{2,p}(x)-h^*_{2,p}(x^-)=p\Big(g^*_p(x)- g^*_p(x^-)\Big)
=O\Big(\frac1p\Big).
\end{split}
\]
So, it holds
\begin{equation}\label{8-22-8}
\frac{h^*_{2,p}(x)-h^*_{2,p}(x^-)}{M_p^{**}(1+x_2)^{\tau_1}}\to 0~\mbox{uniformly in any compact set of}~\big\{|x|\leq \frac{d_0}{\e_{p,j}}\big\}~\mbox{as}~p\to +\infty.
\end{equation}
Combining \eqref{7-22-8} and \eqref{8-22-8}, we prove \eqref{6-22-8}.
Now we can argue as in pages~207 and 208 in \cite{GILY2021} to obtain a contradiction. This concludes the proof of \eqref{M_p^*to0}.
\end{proof}
\begin{Lem}\label{p1-23-8}
Let
\[
q_{p,j}:=p\big(k_{p,j}-w_1\big),
\]
where $k_{p,j}$ is as in \eqref{sat} and $w_{1}$ is its limit profile as $p\rightarrow +\infty$ (see  \eqref{convergencekpjtow1}).
Then for any small fixed $d_0,\tau_1>0$, there exists $C>0$ such that
\begin{equation}\label{1-23-8lemma}
|q_{p,j}(x)|\leq C(1+|x| )^{\tau_1}~\mbox{ in}~ B_{\frac{d_0}{\e_{p,j}}}(0).
\end{equation}
\end{Lem}
\begin{proof} The proof of \eqref{1-23-8lemma} is similar to the one of Proposition~3.5
in \cite{GILY2021}, hence we omit it.
\end{proof}

\

\subsection{The proof of Proposition \ref{prop:upHasTheFormInEMP} }\label{subsection:ProofImprovedFormInEMP}$\,$

\vskip 0.2cm

The proof of Proposition \ref{prop:upHasTheFormInEMP} is postponed at the end of the subsection, after collecting some intermediate results. Let us denote
\begin{equation}\label{def:Wp}
W_p:=\sum_{j=1}^k P W_{p, j},
\end{equation}
where $W_{p, j}$ is defined in \eqref{2-26-8}. Moreover let
\begin{equation}\label{0-4-4}
\mu_{p,j}:=\e_{p,j} e^{\frac p4},
\end{equation}
then by \eqref{nn3-29-03},  $0<\frac{1}{c}\leq \mu_{p,j}\leq c$.

 \vskip 0.2cm

Recall that $W_{p, j}$ approximates $u_{p}$ near $x_{p,j}$ (see Remark \ref{remak:expansion:up}). We will show that  also $W_p$ approximates $u_{p}$.
 \vskip 0.1cm

In order to estimate $ W_{p}- u_p$, we first estimate $PU\bigl(\frac{x-x_{p,j}}{\e_{p,j} }\bigr)$,
$Pw_0\bigl(\frac{x-x_{p,j}}{\e_{p,j} }\bigr)$ and $Pw_1\bigl(\frac{x-x_{p,j}}{\e_{p,j} }\bigr)$, where $U$, $w_{0}$ and $w_{1}$ are the profiles in the definition of  $W_{p, j}$ (see \eqref{2-26-8}).
\begin{Lem}\label{l1-29-8}
It holds
\begin{equation}\label{1-29-8}
PU\bigl(\frac{x-x_{p,j}}{\e_{p,j} }\bigr)= U\bigl(\frac{x-x_{p,j}}{\e_{p,j} }\bigr)
-\log (64\e_{p,j}^4) -8\pi H(x,x_{p,j})+O(\e_{p,j}^2),
\end{equation}
\begin{equation}\label{2-29-8}
Pw_i\bigl(\frac{x-x_{p,j}}{\e_{p,j} }\bigr)= w_i\bigl(\frac{x-x_{p,j}}{\e_{p,j} }\bigr)
-C_i\log \frac1{\e_{p,j}} +2\pi C_iH(x,x_{p,j})+O(\e_{p,j}), \quad i=0,1,
\end{equation}
as $p\rightarrow +\infty$, where $C_{i}$, $i=0,1$ are the constants in Remark \ref{remarkEstimateswi}.
In particular, we have
\begin{equation}\label{tta}
PW_{p, j}= \frac{8\pi}{ p^{\frac p{p-1}}\e_{p,j}^{\frac2{p-1}}} \Big(1-\frac{C_0}{4p}-\frac{C_1}{4p^2}\Big)G(x,x_{p,j})+O\Big(\frac{\e_{p,j}}{p}\Big),~~\mbox{for any}~~
x\in \Omega\backslash B_d(x_{p,j}).
\end{equation}

\end{Lem}
\begin{proof}
Let
\[
\varphi(x)=U\bigl(\frac{x-x_{p,j}}{\e_{p,j} }\bigr)-PU\bigl(\frac{x-x_{p,j}}{\e_{p,j} }\bigr).
\]
Then, $\varphi$ is harmonic in $\Omega$ and on $\partial\Omega$,
\[
\varphi(x)=U\bigl(\frac{x-x_{p,j}}{\e_{p,j} }\bigr)
=\log (64\e_{p,j}^4) +\log\frac1{|x-x_{p,j}|^4} +O(\e_{p,j}^2).
\]
This gives
\[
\varphi(x)=\log (64\e_{p,j}^4) +8\pi H(x,x_{p,j})+O(\e_{p,j}^2),
\]
namely \eqref{1-29-8}.
Similarly, on $\partial\Omega$, it holds
\[
w_i\bigl(\frac{x-x_{p,j}}{\e_{p,j} }\bigr)= C_i\log \frac1{\e_{p,j}}-C_i\log
\frac1{|x-x_{p,j}|} +O(\e_{p,j}),
\]
and thus, \eqref{2-29-8} follows.

\vskip 0.1cm

Similar to Remark 2.1 of \cite{EMP2006}, \eqref{tta} can be deduced by \eqref{2-26-8}, \eqref{1-29-8} and \eqref{2-29-8}.
\end{proof}

\begin{Lem}  For  $x\in B_d(x_{p,j})$,
\begin{equation}\label{3-29-8}
\begin{split}
& W_{p}(x)- \sum_{j=1}^k W_{p, j}(x)\\
=& \frac1{ p^{\frac p{p-1}}\e_{p,j}^{\frac2{p-1}}} \Bigl[
-\log (64\mu_{p,j}^4) -8\pi\bigl(1-\frac{C_0}{4p}- \frac{C_1}{4p^2} \bigr) H(x,x_{p,j})+\bigl(C_0 +\frac{C_1}{p}\bigr)\frac{\log (\mu_{p,j}e^{-p/4})}{p}\Bigr]\\
&\qquad + \frac1{ p^{\frac p{p-1}}\e_{p,j}^{\frac2{p-1}}}8\pi\sum_{l\ne j}\left(\frac{\mu_{p,j}}{ \mu_{p,l}}\right)^{\frac2{p-1}}\bigl(1-\frac{C_0}{4p}- \frac{C_1}{4p^2} \bigr) G(x,x_{p,l})
+O(\frac{\e_{p,j}}p),\end{split}
\end{equation}
as $p\rightarrow +\infty$, where $W_{p, j}$ is defined in \eqref{2-26-8}.
\end{Lem}

\begin{proof}
It follows from
Lemma~\ref{l1-29-8}.
\end{proof}
 We prove the following result.
\begin{Lem}\label{l2-29-8}
Let
\[\begin{split}
\boldsymbol{\mu}_{p}:=(\mu_{p,1},\cdots,\mu_{p,k}),\;\;
\boldsymbol{x}_{p}:=(x_{p,1},\cdots,x_{p,k}),
\end{split}\]
where $x_{p,j}$ and $\mu_{p,j}$ for $j=1,\cdots,k$ are given in \eqref{def:xpj} and \eqref{0-4-4} respectively. Then
it holds
\begin{equation}\label{4-29-8}
\begin{split}
F_{p,j}(\boldsymbol{\mu}_{p},\boldsymbol{x}_{p})
=O\big(\frac1{p^3}\big),
\end{split}
\end{equation}
as $p\rightarrow +\infty$, where $F_{p,j}$ is defined in \eqref{10-30-8}.
In particular,
\begin{equation}\label{n3-29-8}
 W_{p}(x)- \sum_{j=1}^k W_{p, j}(x)=\frac1p
O\big(\frac{1}{p^3}+|x-x_{p,j}|\big),\quad x\in  B_d(x_{p,j}),\; j=1,\cdots,k.
\end{equation}
\end{Lem}
\begin{proof}
\eqref{n3-29-8} follows from \eqref{4-29-8}, observing that from \eqref{3-29-8}   one has
\[W_{p}(x_{p,j})- \sum_{j=1}^k W_{p, j}(x_{p,j})= \frac1{ p^{\frac p{p-1}}\e_{p,j}^{\frac2{p-1}}} F_{p,j}(\boldsymbol{\mu}_{p},\boldsymbol{x}_{p}) +O(\frac{\e_{p,j}}p).\]
Next we prove \eqref{4-29-8}.
By the Green representation theorem and \eqref{11-14-03N}, we have
\begin{equation}\label{s5-8-3}
\begin{split}
u_p(x_{p,j})=& \int_{\Omega}G(y,x_{p,j})
 u_{p}^p(y)\,dy=
 \sum^k_{l=1}\int_{B_d(x_{p,l})}G(y,x_{p,j})u_{p}^p(y)\,dy
+O\Big(\frac{C^p}{p^p}\Big).
\end{split}
\end{equation}
Firstly, we estimate
$\displaystyle\int_{B_d(x_{p,j})}u_{p}^p(y)\,dy$.
We have the following expansion
\begin{equation}\label{5-29-8}
\begin{split}
&\Bigl( 1+ \frac{U\bigl(\frac{x-x_{p,j}}{\e_{p,j} }\bigr)}p+\frac{w_{0}\bigl(\frac{x-x_{p,j}}{\e_{p,j} }\bigr)}{p^2}+\frac{w_{1}\bigl(\frac{x-x_{p,j}}{\e_{p,j} }\bigr)}{p^3}
+O\bigl(\frac{1}{p^4}\bigr) \Bigr)^p\\
=& e^{ U\bigl(\frac{x-x_{p,j}}{\e_{p,j} }\bigr)  }\left[1+\frac1p\Bigl(
w_{0}\bigl(\frac{x-x_{p,j}}{\e_{p,j} }\bigr)-\frac12 U^2\bigl(\frac{x-x_{p,j}}{\e_{p,j} }\bigr)\Bigr)\right.\\
&\qquad \left.+\frac1{p^2} \Bigl(w_{1}\bigl(\frac{x-x_{p,j}}{\e_{p,j} }\bigr)-
U\bigl(\frac{x-x_{p,j}}{\e_{p,j} }\bigr)w_{0}\bigl(\frac{x-x_{p,j}}{\e_{p,j} }\bigr)+\frac13 U^3\bigl(\frac{x-x_{p,j}}{\e_{p,j} }\bigr)+\frac12
w_{0}^2\bigl(\frac{x-x_{p,j}}{\e_{p,j} }\bigr)\right.\\
& \left. \qquad +\frac18 U^4\bigl(\frac{x-x_{p,j}}{\e_{p,j} }\bigr)-\frac12 U^2\bigl(\frac{x-x_{p,j}}{\e_{p,j} }\bigr)w_{0}\bigl(\frac{x-x_{p,j}}{\e_{p,j} }\bigr)
\Bigr)+O\Bigl( \frac{ \log^6 (2+ \frac{|x-x_{p,j}|}{\e_{p,j} }) }{p^3} \Bigr)
\right].
\end{split}
\end{equation}
For simplicity, we denote
\begin{equation}\label{6-29-8}
f_0(r):=w_{0}(r)-\frac12 U^2(r)
\end{equation}
and
\begin{equation}\label{7-29-8}
f_1(r):= w_{1}(r)-
U(r)w_{0}(r)+\frac13 U^3(r)+\frac12
w_{0}^2(r)
+\frac18 U^4(r)-\frac12 U^2(r)w_{0}(r).
\end{equation}
And then  \eqref{6-4-2} and \eqref{10-19-8} may be rewritten as \[-\Delta w_{0}=e^{U}f_{0}\quad\mbox{ in }\mathbb R^{2}~~~~\mbox{and}~~~~-\Delta w_{1}=e^{U}f_{1}\quad\mbox{ in }\mathbb R^{2}.\]
Hence using \eqref{1-26-8}  and \eqref{5-29-8}, we have
\begin{equation}\label{8-29-8}
\begin{split}
\int_{B_d(x_{p,j})}u_{p}^p(y)\,dy
 =&u^p_p(x_{p,j})\int_{B_d(x_{p,j})}e^{ U\bigl(\frac{x-x_{p,j}}{\e_{p,j} }\bigr)  }\left[1+\frac1p
f_0\bigl(\frac{x-x_{p,j}}{\e_{p,j} }\bigr)
+\frac1{p^2} f_1\bigl(\frac{x-x_{p,j}}{\e_{p,j} }\bigr)+O\Bigl( \frac{ (2+ \frac{|x-x_{p,j}|}{\e_{p,j} })^{\tau_1} }{p^3} \Bigr)
\right] \\
 =&u^p_p(x_{p,j})\e_{p,j}^2\left[\int_{B_{d \e_{p,j}^{-1}}(0)}e^{ U  }\Bigl(1+\frac1p
f_0
+\frac1{p^2} f_1\Bigr)+O\Bigl( \frac{1 }{p^3} \Bigr)\right] \\
 =&\frac1{ p^{\frac p{p-1}}\e_{p,j}^{\frac2{p-1}}}\int_{\mathbb R^2}e^{ U  }\Bigl(1+\frac1p
f_0
+\frac1{p^2} f_1\Bigr)+O\Bigl( \frac{1 }{p^4} \Bigr).
\end{split}
\end{equation}
We also have
\[
\int_{\mathbb R^2}e^{ U  }=8\pi.
\]
Moreover, it follows from \eqref{0-29-8},
\[
\int_{\mathbb R^2}e^{ U  } f_i=\lim_{R\to +\infty}\int_{B_R(0)}e^{ U  } f_i = -\lim_{R\to +\infty}\int_{\partial B_R(0)} w_i'=-2\pi C_i.
\]
So we have proved that
\begin{equation}\label{9-29-8}
\begin{split}
\int_{B_d(x_{p,j})}u_{p}^p(y)\,dy
=\frac1{ p^{\frac p{p-1}}\e_{p,j}^{\frac2{p-1}}}
\Bigl(8\pi -\frac{2C_0\pi}p -\frac{2C_1\pi}{p^2}
\Bigr)+O\Bigl( \frac{1 }{p^4} \Bigr).
\end{split}
\end{equation}
It follows from \eqref{s5-8-3} and \eqref{9-29-8} that
\begin{equation}\label{10-29-8}
\begin{split}
u_p(x_{p,j})=& \int_{\Omega}G(y,x_{p,j})
 u_{p}^p(y)\,dy\\=&
 \sum^k_{l=1}\int_{B_d(x_{p,l})}G(y,x_{p,j})u_{p}^p(y)\,dy
+O\Big(\frac{C^p}{p^p}\Big)\\
=&\frac1{2\pi} \int_{B_d(x_{p,j})}\log\frac1{|y-x_{p,j}|}u_{p}^p(y)\,dy-H(x_{p,j},x_{p,j})\int_{B_d(x_{p,j})}u_{p}^p(y)\,dy
\\&+
\sum_{l\ne j}\int_{B_d(x_{p,l})}G(x_l,x_{p,j})u_{p}^p(y)\,dy
+O(\e_{p,j})\\
=&\frac1{2\pi} \int_{B_d(x_{p,j})}\log\frac1{|y-x_{p,j}|}u_{p}^p(y)\,dy
\\&+
\Bigl(8\pi -\frac{2C_0\pi}p -\frac{2C_1\pi}{p^2}
\Bigr)\left[
\sum_{l\ne j}\frac1{ p^{\frac p{p-1}}\e_{p,l}^{\frac2{p-1}}}G(x_{p,l},x_{p,j})-\frac1{ p^{\frac p{p-1}}\e_{p,j}^{\frac2{p-1}}}H(x_{p,j},x_{p,j})\right]
+O\big(\e_{p,j}\big).
\end{split}
\end{equation}
Similar to \eqref{8-29-8}, we can prove that
\begin{equation}\label{11-29-8}
\begin{split}
&\int_{B_d(x_{p,j})}\log\frac1{|y-x_{p,j}|}u_{p}^p(y)\,dy\\
&=u^p_p(x_{p,j})\int_{B_d(x_{p,j})}\log\frac1{|y-x_{p,j}|}e^{ U\bigl(\frac{y-x_{p,j}}{\e_{p,j} }\bigr)  }\left[1+\frac1p
f_0\bigl(\frac{y-x_{p,j}}{\e_{p,j} }\bigr)
+\frac1{p^2} f_1\bigl(\frac{y-x_{p,j}}{\e_{p,j} }\bigr)+O\Bigl( \frac{\log^6 (2+ \frac{|y-x_{p,j}|}{\e_{p,j} }) }{p^3} \Bigr)
\right]\\
&=u^p_p(x_{p,j})\e_{p,j}^2\left[\int_{B_{d \e_{p,j}^{-1}}(0)}\Bigl[\log \frac1{\e_{p,j}}+\log\frac1{|y|}\Bigr]e^{ U  }\Bigl(1+\frac1p
f_0
+\frac1{p^2} f_1\Bigr)+O\Bigl( \frac{1 }{p^3} \Bigr)\right]\\
=&\frac1{ p^{\frac p{p-1}}\e_{p,j}^{\frac2{p-1}}}\log\frac1{\e_{p,j}}
\Bigl(8\pi -\frac{2C_0\pi}p -\frac{2C_1\pi}{p^2}
\Bigr)
+\frac1{ p^{\frac p{p-1}}\e_{p,j}^{\frac2{p-1}}}\int_{\mathbb R^2}\log\frac1{|y|}e^{ U  }\Bigl(1+\frac1p
f_0
+\frac1{p^2} f_1\Bigr)+O\Bigl( \frac{1 }{p^4} \Bigr).\\
\end{split}
\end{equation}
Now, we use Proposition~\ref{p1-30-8} to calculate the integrals on the right hand side of \eqref{11-29-8}.
Since
\[
U(x)= \log\frac1{(1+\frac18 |x|^2)^2} =4\log\frac1{ |x|}+\log64 +O\bigl(\frac1{|x|^2}\bigr),\quad\text{as}\; |x|\to +\infty,
\]
using Proposition~\ref{p1-30-8}, we find
\[
U(0)= \frac1{2\pi}\int_{\mathbb R^2}\log\frac1{|y|}e^{ U  }+\log 64,
\]
which, together with $U(0)=0$, gives
\[
 \frac1{2\pi}\int_{\mathbb R^2}\log\frac1{|y|}e^{ U  }=-\log 64.
\]
On the other hand, by \eqref{estimatewi}  and
 Proposition~\ref{p1-30-8}, we find that
\[
\int_{\mathbb R^2}\log\frac1{|y|}e^{ U  }
f_i= w_i(0)
=0.
\]
So we have proved that
\begin{equation}\label{12-29-8}
\begin{split}
&\int_{B_d(x_{p,j})}\log\frac1{|y-x_{p,j}|}u_{p}^p(y)\,dy\\
=&\frac1{ p^{\frac p{p-1}}\e_{p,j}^{\frac2{p-1}}}\left[ \log\frac1{\e_{p,j}}
\Bigl(8\pi -\frac{2C_0\pi}p -\frac{2C_1\pi}{p^2}
\Bigr)-\log 64\right]
+O\Bigl( \frac{1 }{p^4} \Bigr).
\end{split}
\end{equation}
Combining \eqref{10-29-8} and \eqref{12-29-8}, in view of \eqref{nn1-26-8},
we have
\begin{equation}\label{20-29-8}
\begin{split}
1=& \frac1{2\pi p}\left[\log\frac1{\e_{p,j}}
\Bigl(8\pi -\frac{2C_0\pi}p -\frac{2C_1\pi}{p^2}
\Bigr)-\log 64\right]\\
&+\frac1p
\Bigl(8\pi -\frac{2C_0\pi}p -\frac{2C_1\pi}{p^2}
\Bigr)\left[
\sum_{l\ne j}\frac{ p^{\frac p{p-1}}\e_{p,j}^{\frac2{p-1}} }{ p^{\frac p{p-1}}\e_{p,l}^{\frac2{p-1}} }G(x_{p,l},x_{p,j})-H(x_{p,j},x_{p,j})\right]
+O\Big(\frac1{p^4}\Big).
\end{split}
\end{equation}
But
\[
\frac4{ p}\log\frac1{\e_{p,j}}=1+\frac1p \log\frac1{\mu_{p,j}^4},
\]
which, together with \eqref{20-29-8}, gives
\[
\begin{split}
&-\log (64\mu_{p,j}^4)-\Bigl(\frac{C_0}p +\frac{C_1}{p^2}
\Bigr)\log\frac1{\e_{p,j}}
\\
&+
\Bigl(8\pi -\frac{2C_0\pi}p -\frac{2C_1\pi}{p^2}
\Bigr)\left[
\sum_{l\ne j}\frac{ \mu_{p,j}^{\frac2{p-1}} }{\mu_{p,l}^{\frac2{p-1} } }G(x_{p,l},x_{p,j})-H(x_{p,j},x_{p,j})\right]
=O\Big(\frac1{p^3}\Big).
\end{split}
\]
\end{proof}
We now estimate $\|W_{p}- u_p\|_{L^\infty(\Omega)}$.
\begin{Prop}\label{l1-7-9}
Let $W_{p}$ be as in \eqref{def:Wp}. For any small $\tau>0$, it holds
\begin{equation*}
\|W_{p}- u_p\|_{L^\infty(\Omega)}=O\Big(\frac{1}{p^{4-\tau}}\Big).
\end{equation*}
\end{Prop}
\begin{proof}
It follows from \eqref{1-26-8} and \eqref{n3-29-8} that for any small $\tau>0$,
it holds
\begin{equation}\label{2-7-9}
 |W_{p}(x)- u_p(x)|\le \frac C{p^{4-\tau}},\quad x\in \displaystyle\Cup_{j=1}^k B_{2\e_{p,j} p^{3}}(x_{p,j}).
\end{equation}
We also have that for $x\in\Omega\setminus \Cup_{j=1}^k B_{2\e_{p,j} p^{3}}(x_{p,j})$,
\begin{equation}\label{3-7-9}
\begin{split}
u_p(x)=& \int_{\Omega}G(y,x)
 u_{p}^p(y)\,dy=
 \sum^k_{l=1}\int_{B_d(x_{p,l})}G(y,x)u_{p}^p(y)\,dy
+O\Big(\frac{C^p}{p^p}\Big).
\end{split}
\end{equation}
On the other hand,
\begin{equation}\label{4-7-9}
\begin{split}
&\int_{B_d(x_{p,j})\setminus B_{ \e_{p,j} p^{3}}(x_{p,j})}G(y,x)u_{p}^p(y)\,dy\\
\le & Cu^p_p(x_{p,j})\int_{B_d(x_{p,j})\setminus B_{\e_{p,j} p^{3}}(x_{p,j})}|\log\frac1{|y-x|}|e^{ U\bigl(\frac{x-x_{p,j}}{\e_{p,j} }\bigr)  }\\
\le & C u^p_p(x_{p,j})\e_{p,j}^2\left[\int_{B_{d \e_{p,j}^{-1}}(0)\setminus B_{  p^{3}}(0)}e^{ U  }\Bigl(|\log\e_{p,j}|+|\log \frac1{|y-\frac{x-x_{p,j}}
{\e_{p, j}}|}| \Bigr)\right]\\
=&O\Bigl( \frac{1 }{p^4} \Bigr).
\end{split}
\end{equation}
Moreover, for $x\notin B_{ 2\e_{p,j} p^{3}}(x_{p,j})$, similar to \eqref{9-29-8},
we can prove that
\begin{equation}\label{5-7-9}
\begin{split}
& \int_{B_{ \e_{p,j} p^{3}}(x_{p,j})}G(y,x)
 u_{p}^p(y)\,dy\\=&
G(x_{p,j}, x) \int_{B_{ \e_{p,j}  p^{3}}(x_{p,j})}
 u_{p}^p(y)\,dy
+O\Big(\e_{p,j}\Big)\\
=& \frac1{ p^{\frac p{p-1}}\e_{p,j}^{\frac2{p-1}}}
\Bigl(8\pi -\frac{2C_0\pi}p -\frac{2C_1\pi}{p^2}
\Bigr)G(x_{p,j}, x)+O\Bigl( \frac{1 }{p^4} \Bigr).
\end{split}
\end{equation}
Combining \eqref{3-7-9}, \eqref{4-7-9} and \eqref{5-7-9}, we are led to
\begin{equation}\label{6-7-9}
\begin{split}
u_p(x)=& \sum_{j=1}^k\frac1{ p^{\frac p{p-1}}\e_{p,j}^{\frac2{p-1}}}
\Bigl(8\pi -\frac{2C_0\pi}p -\frac{2C_1\pi}{p^2}
\Bigr)G(x_{p,j}, x)+O\Bigl( \frac{1 }{p^4} \Bigr),\quad x\in\Omega\setminus \Cup_{j=1}^k B_{2\e_{p,j} p^{3}}(x_{p,j}).
\end{split}
\end{equation}
It is easy to see from Lemma~\ref{l1-29-8}
that

\begin{equation}\label{1-4-4}
\begin{split}
W_p(x)=& \sum_{j=1}^k\frac1{ p^{\frac p{p-1}}\e_{p,j}^{\frac2{p-1}}}
\Bigl(8\pi -\frac{2C_0\pi}p -\frac{2C_1\pi}{p^2}
\Bigr)G(x_{p,j}, x)+O\Bigl( \frac{1 }{p^4} \Bigr),\quad x\in\Omega\setminus \Cup_{j=1}^k B_{2\e_{p,j} p^{3}}(x_{p,j}).
\end{split}
\end{equation}
Hence the result follows from \eqref{2-7-9}, \eqref{6-7-9} and \eqref{1-4-4}.
\end{proof}
Next, for any given $\boldsymbol{\mathrm \xi}=(\xi_1,\cdots, \xi_k)\in\Omega^{k}$ close to $\boldsymbol{\mathrm x}_{\infty}=(x_{\infty,1},\cdots,x_{\infty,k})$, let
\[\boldsymbol{\bar\mu}_{p}:=(\bar \mu_{p,1},\cdots,\bar\mu_{p,k}),   \]
where $\bar \mu_{p,j}=\bar \mu_{p,j}(\boldsymbol{\mathrm  \xi})$, $j=1,\cdots,k$, are defined in \eqref{10-30-8}. Let also $\bar\varepsilon_{p,j}$, $\bar{W}_{ p, j}$ be as in \eqref{def:barepsilon} and \eqref{12-30-8}, respectively. We define
\begin{equation*}
\bar W_{p}:=\sum_{j=1}^k  P\bar W_{p, j},
\end{equation*}
where $P$ is the projection in \eqref{defProjection}.

\vskip 0.1cm

To avoid confusion, we use $\boldsymbol{\bar\mu}_{p}^*$, $\bar \mu_{p,j}^*$, $\bar\varepsilon_{p,j}^*$, $\bar{W}_{ p, j}^*$ and $\bar W_p^*$ to denote $\boldsymbol{\bar\mu}_{p}$, $\bar \mu_{p,j}$, $\bar\varepsilon_{p,j}$, $\bar{W}_{p,j}$ and $\bar W_p$, respectively, when
 $\boldsymbol{\xi}=\boldsymbol{x}_{p}:=(x_{p,1},\cdots, x_{p,k})$.
Recall the $\mu_{p,j}$ is defined in \eqref{0-4-4}. Using Lemma~\ref{l2-29-8}, we have
\begin{equation*}
|\mu_{p,j}-\bar \mu_{p,j}^*|=O\bigl(\frac1{p^3}\bigr),
\end{equation*}
namely
\begin{equation}\label{11-30-8BIS}
\frac{  \bar\e_{p,j}^* }{ \e_{p,j} }=\frac{  \bar\mu_{p,j}^* }{ \mu_{p,j} }=1+O\bigl(\frac1{p^3}\bigr),
\end{equation}
as $p\rightarrow +\infty$.
Then by Proposition~\ref{l1-7-9} and \eqref{11-30-8BIS}, we can deduce that
\begin{equation}\label{10-7-9}
\begin{split}
\|\bar W^*_{p}-u_p\|_{L^\infty(\Omega)}\le &
\|\bar W^*_{p}-W_p\|_{L^\infty(\Omega)}+\| W_{p}-u_p\|_{L^\infty(\Omega)}
\\
=& \|\bar W^*_{p}-W_p\|_{L^\infty(\Omega)}+ O\bigl(\frac1{p^{4-\tau}}\bigr)= O\bigl(\frac1{p^{4-\tau}}\bigr).
  \end{split}
\end{equation}

Let
$\|u\|:=\bigl(\displaystyle\int_{\Omega} |\nabla u|^2\bigr)^{\frac12}$ for $u\in H_0^1(\Omega)$, we have
\begin{Prop}\label{prop10-30-8}
 For any small $\tau>0$, it holds
\begin{equation*}
\|u_p-\bar W_{p}^*\|^2=O\bigl(
\frac{1}{p^{7-\tau}}\bigr),~~~\mbox{as}~~~p\rightarrow +\infty.
\end{equation*}
\end{Prop}
\begin{proof}
We have
\begin{equation}\label{1-27-8}
\begin{split}
&\|u_p- \bar W_{p}^*\|^2=\int_\Omega\Bigl(-\Delta (u_p-
\bar W_{p}^*)\Bigr)(u_p-
\bar W_{p}^*).
\end{split}
\end{equation}
On the other hand, it follows from \eqref{1-26-8} and \eqref{5-29-8} that
\begin{equation}\label{30-30-8}
\begin{split}
-\Delta u_p
=&\big(u_p(x_{p,j})\big)^{p}\Bigl( 1+ \frac{U\bigl(\frac{x-x_{p,j}}{\e_{p,j} }\bigr)}p+\frac{w_{0}\bigl(\frac{x-x_{p,j}}{\e_{p,j} }\bigr)}{p^2}+\frac{w_{1}\bigl(\frac{x-x_{p,j}}{\e_{p,j} }\bigr)}{p^3}
+O\bigl(\frac{(1+ \frac{|x-x_{p,j}|}{\e_{p,j} })^{\tau_1}}{p^4}\bigr) \Bigr)^p\\
 =&\frac1{ p^{\frac p{p-1}}\e_{p,j}^{\frac2{p-1}+2}}
 e^{ U\bigl(\frac{x-x_{p,j}}{\e_{p,j} }\bigr)  }\left[1+\frac1p f_0\bigl(\frac{x-x_{p,j}}{\e_{p,j} }\bigr) +\frac1{p^2} f_1\bigl(\frac{x-x_{p,j}}{\e_{p,j} }\bigr)+O\Bigl( \frac{ (1+ \frac{|x-x_{p,j}|}{\e_{p,j} })^{\tau_1} }{p^3} \Bigr)
\right],
 \quad \text{in}\; B_{d_0}(x_{p, j}),
\end{split}
\end{equation}
where $f_{0}$ and $f_{1}$ are defined in \eqref{6-29-8} and \eqref{7-29-8}, respectively. But
\begin{equation}\label{3-26-8}
\begin{split}
-\Delta P\bar W_{p,j}^*
=\frac1{ p^{\frac p{p-1}}(\bar\e_{p,j}^*)^{\frac2{p-1}+2}} e^{ U\bigl(\frac{x-x_{p,j}}{\bar\e_{p,j}^* }\bigr)  }\left[1+\frac1p f_0\bigl(\frac{x-x_{p,j}}{\bar\e_{p,j}^* }\bigr) +\frac1{p^2} f_1\bigl(\frac{x-x_{p,j}}{\bar \e_{p,j}^* }\bigr)
\right],
 \quad \text{in}\; B_{d_0}(x_{p, j}).
\end{split}
\end{equation}
Combining \eqref{1-27-8}, \eqref{30-30-8} and \eqref{3-26-8},
using \eqref{n3-29-8} and \eqref{11-30-8BIS}, we are led
to
\begin{equation*}
\begin{split}
\|u_p- \bar W_{p}^*\|^2=\frac1{ p^{\frac p{p-1}}(\bar\e_{p,j}^*)^{\frac2{p-1}}}O\Bigl(
\frac{1}{p^{6-\tau}}\Bigr)=O\Bigl(
\frac{1}{p^{7-\tau}}\Bigr).
\end{split}
\end{equation*}
\end{proof}
\begin{proof}[{\bf Proof of Proposition \ref{prop:upHasTheFormInEMP}}]
We now consider the following problem
\begin{equation}\label{32-30-8}
\begin{split}
m_p:= \inf\Bigl\{ \|u_p- \bar W_{\boldsymbol{\mathrm\alpha},p}\|^2:\;\; |\alpha_j-1|\le \frac{1}{p^{\frac{5}{2}}},\; |\xi_j-x_{p, j}|\le \frac{1}{p^{2} e^{\frac{p}{4} }}\Bigr\},
\end{split}
\end{equation}
where $\bar W_{\boldsymbol{\mathrm\alpha},p}$ is defined as in \eqref{def:Wpalpha}. Then, by Proposition \ref{prop10-30-8}, for any small $\tau>0$, it holds $m_p= O\bigl(
p^{-7+\tau}\bigr)$ as $p\to +\infty$.
Let $\boldsymbol{\mathrm\alpha}_{p}=(\alpha_{p,1},\cdots,\alpha_{p,k})\in\mathbb R^{k}$ and $\boldsymbol{\mathrm \xi}_{p}=( \xi_{p,1},\cdots, \xi_{p,k})\in\Omega^{k}$
be a minimizer of \eqref{32-30-8}.  Now we denote $\bar{\mu}^{**}_{p,j}:=\bar \mu_{p,j}(\boldsymbol{\mathrm  \xi_p})$ and $\bar{\e}^{**}_{p,j}:=\bar{\mu}^{**}_{p,j}e^{-\frac{p}{4}}$. Then by \eqref{10-30-8}, it holds
\begin{equation*}
|\bar{\mu}^{**}_{p,j}-\bar \mu_{p,j}^*|=O\bigl(\|\boldsymbol{\mathrm  \xi_p}-\boldsymbol{\mathrm  x_p}\|\bigr),
\end{equation*}
namely
\begin{equation*}
\frac{  \bar\e_{p,j}^* }{ \bar{\e}^{**}_{p,j} }=\frac{  \bar\mu_{p,j}^* }{ \bar{\mu}^{**}_{p,j} }=1+O\bigl(\|\boldsymbol{\mathrm  \xi_p}-\boldsymbol{\mathrm  x_p}\|\bigr),
\end{equation*}
as $p\rightarrow +\infty$.

\vskip 0.1cm

Also by Proposition \ref{prop10-30-8},
\begin{equation}\label{21-4-9}
\begin{split}
\|\bar W_{\boldsymbol{\mathrm\alpha},p}- \bar W^*_{p}\|^2&=
\|\sum_{j=1}^k \alpha_{p,j} P\bar W_{p, j}- \bar W^*_{p}\|^2\le 2\|u_p- \sum_{j=1}^k \alpha_{p,j} P\bar W_{p, j}\|^2+2\|u_p- \bar W^*_{p}\|^2\\
&\le 4\|u_p- \bar W^*_{p}\|^2
= {O\bigl(
\frac{1}{p^{7-\tau}}\bigr)}.
\end{split}
\end{equation}
%
Furthermore, for $i\ne j$,  by \eqref{3-26-8}, we have
\[
\begin{split}
 -\int_{\Omega} \Delta\bar W_{p, j}   P\bar W^*_{p, i}
=&-\int_{B_d(\xi_j)}\Delta\bar W_{p, j}   P\bar W^*_{p, i}+O\Big(\frac{1}{e^{\frac{p}{4} }}\Big).
\end{split}\]
Also, by \eqref{tta}, we know
\begin{equation*}
 P\bar W^*_{p, i} =  \frac{8\pi}{ p^{\frac p{p-1}}(\bar{\e}^{**}_{p,i})^{\frac2{p-1}}} \Big(1-\frac{C_0}{4p}-\frac{C_1}{4p^2}\Big)G(x,x_{p,i})+O\Big(\frac{1}{e^{\frac{p}{4} }}\Big),\quad \text{in}\; B_d(\xi_j).
\end{equation*}
By Taylor's expansion of $G(x,x_{p,i})$ at $\xi_j$,  we get
\begin{equation}\label{hast}
\begin{split}
 -\int_{\Omega} \Delta\bar W_{p, j}   P\bar W^*_{p, i}
=   \tau_{p,j}
 G\big(\xi_j,x_{p,i}\big)  +O\Big(\frac{1}{e^{\frac{p}{4} }}\Big),
\end{split}
\end{equation}where
\begin{equation}\label{hast1}
\begin{split}
\tau_{p,j}:=&
 \Big(1-\frac{C_0}{4p}-\frac{C_1}{4p^2}\Big)
\frac{64\pi^2}{  p^{\frac {2p}{p-1}}(\bar{\e}_{p,j})^{\frac2{p-1}}  (\bar \e_{p,j}^{**})^{\frac2{p-1}}}
 \underbrace{ \int_{\R^2} \Bigl(\Delta U(x)+\frac{\Delta w_0(x)}{p} +\frac{\Delta w_1(x)}{p^2} \Bigr) dx}_{\text{this is a constant}}.
 \end{split}
\end{equation}
Hence by \eqref{32-30-8}, \eqref{hast} and \eqref{hast1}, it holds
\[
\begin{split}
\int_{\Omega} \Delta\bar W_{p, j}   \bigl( P\bar W_{p, i}-  P\bar W^*_{p, i} \bigr)=&
\tau_{p,j}\Big( G\big(\xi_j,\xi_{i}\big) - G\big(\xi_j,x_{p,i}\big) \Big)
\\=& O\Bigl(\frac{ |\xi_i-x_{p, i}|  }{p^2}
\Bigr) +O\Big(\frac{1}{e^{\frac{p}{4} }}\Big)=O\Big(\frac{1}{e^{\frac{p}{4} }}\Big).
\end{split}
\]By similar computations, we find
\[
\begin{split}
&\int_{\Omega}\bigl(\Delta\bar W_{p, i} -\Delta\bar W^*_{p, i}\bigr) \bigl( P\bar W_{p, j}-  P\bar W^*_{p, j} \bigr)
=O\Big(\frac{1}{e^{\frac{p}{4} }}\Big).
\end{split}
\]
And then we deduce that
\begin{equation*}
 \begin{split}
&\bigl\langle  \alpha_{p,i} P\bar W_{p, i}-  P\bar W^*_{p, i},\;
 \alpha_{p,j} P\bar W_{p, j}-  P\bar W^*_{p, j}\bigr\rangle
\\
=&\int_{\Omega}\Bigl( -\alpha_{p,i}\Delta\bar W_{p, i} +\Delta\bar W^*_{p, i}\Bigr)
\Bigl( \alpha_{p,j} P\bar W_{p, j}-  P\bar W^*_{p, j} \Bigr)\\
=&-\int_{\Omega}\Bigl( (\alpha_{p,i}-1)\Delta\bar W_{p, i} +
\Delta\bar W_{p, i}-\Delta\bar W^*_{p, i}\Bigr)
\Bigl( (\alpha_{p,j}-1) P\bar W_{p, j}+P\bar W_{p, j}-  P\bar W^*_{p, j} \Bigr)\\
=&O\Bigl(\frac{ |\alpha_{p,j}-1|\cdot|\alpha_{p,i}-1| }{p^2}
\Bigr) +O\Big(\frac{1}{e^{\frac{p}{4} }}\Big).
\end{split}
\end{equation*}
Thus, it holds
\begin{equation}\label{33-30-8}
\begin{split}
\|\bar W_{\boldsymbol{\mathrm\alpha},p}- \bar W^*_{p}\|^2
=&
\sum_{j=1}^k\| \alpha_{p,j} P\bar W_{p, j}- P \bar W^*_{p,j}\|^2 +O\Bigl(\sum^k_{j=1}\sum_{i\neq j}\frac{ |\alpha_{p,j}-1|\cdot |\alpha_{p,i}-1| }{p^2}
\Bigr) +O\Big(\frac{1}{e^{\frac{p}{4} }}\Big).
 \end{split}
\end{equation}
Let us observe that
\begin{equation}\label{nn33-30-8}
\begin{split}
\|\alpha_{p,j} P\bar W_{p, j}- P \bar W^*_{p, j}\|^2
= \alpha_{p,j}^2\| P\bar W_{p, j}\|^2+\| P \bar W^*_{p, j}\|^2-2\alpha_{p,j}
\int_\Omega
\nabla P\bar W_{p, j}\nabla P \bar W^*_{p, j}.
\end{split}
\end{equation}
Then \eqref{nn33-30-8} gives
\begin{equation*}
\begin{split}
\|\alpha_{p,j} P\bar W_{p, j}- P \bar W^*_{p, j}\|^2
\ge  \alpha_{p,j}^2\| P\bar W_{p, j}\|^2+\| P \bar W^*_{p, j}\|^2-2\alpha_{p,j}
\| P\bar W_{p, j}\|\cdot\| P \bar W^*_{p, j}\|.
\end{split}
\end{equation*}
Similar to  the proof of (5.6) in \cite{EMP2006}, we can prove by \eqref{32-30-8} that
\begin{equation}\label{31-5-4t}
\| P\bar W_{p, j}\|^2-\| P \bar W^*_{p, j}\|^2=O\bigl(\frac{ |\xi_j-x_{p, j}| }{p^2}\bigr)=O\Bigl(\frac{1}{p^{4} e^{\frac{p}{4} }}\Bigr).
\end{equation}
So we have
\begin{equation}\label{31-5-4}
\begin{split}
\|\alpha_{p,j} P\bar W_{p, j}- P \bar W^*_{p, j}\|^2
\ge  \bigl(\alpha_{p,j}-1)^2\| P\bar W_{p, j}\|^2+O\Bigl(\frac{1}{p^{4} e^{\frac{p}{4} }}\Bigr).
\end{split}
\end{equation}
Combining \eqref{21-4-9}, \eqref{33-30-8} and \eqref{31-5-4}, we obtain
\[
\sum^k_{j=1} \bigl(\alpha_{p,j}-1)^2\| P\bar W_{p, j}\|^2\le C\Bigl(
\frac{1}{p^{7-\tau}}
+\sum^k_{j=1}\frac{ |\alpha_{p,j}-1|^2 }{p^2}
\Bigr),
\]which, together with $\| P\bar W_{p, j}\|^2\sim \frac{1}{p}$, gives
\begin{equation}\label{luo-1mt}
|\alpha_{p,j}-1|=O\Bigl(
\frac{1}{p^{(6-\tau)/2}}\Bigr),~~~\mbox{for}~~~j=1,\cdots,k.
\end{equation}
On the other hand, by \eqref{nn33-30-8}, \eqref{31-5-4t} and \eqref{luo-1mt}, we get
\begin{equation}\label{ttsma}
\begin{split}
\|&\alpha_{p,j} P\bar W_{p, j}- P \bar W^*_{p, j}\|^2
\\= &\alpha_{p,j}^2\| P\bar W_{p, j}\|^2+\| P \bar W^*_{p, j}\|^2-2\alpha_{p,j}
\int_\Omega
\nabla P\bar W_{p, j}\nabla P \bar W^*_{p, j}\\
=&
(\alpha_{p,j}-1)^2\| P\bar W_{p, j}\|^2-2\alpha_{p,j}
\int_\Omega
\nabla P\bar W_{p, j}\nabla \big( P \bar W^*_{p, j}-P\bar W_{p, j}\big)
-{\| P\bar W_{p, j}\|^2+\| P \bar W^*_{p, j}\|^2}\\
=
&
(\alpha_{p,j}-1)^2\| P\bar W_{p, j}\|^2-2\alpha_{p,j}
\int_\Omega
\nabla P\bar W_{p, j}\nabla \big( P \bar W^*_{p, j}-P\bar W_{p, j}\big)
+ {O\Bigl(\frac{ |\xi_j-x_{p, j}|}{p^2}\Bigr)}\\=&
 2\alpha_{p,j}
\int_\Omega \Delta (P \bar W_{p, j}) \big( P \bar W^*_{p, j}-P\bar W_{p, j}\big)
+O\Bigl(\frac{|\alpha_{p,j}-1|^2}{p} + {\frac{ |\xi_j-x_{p, j}|}{p^2}\Bigr)}\\=& 2\alpha_{p,j}
\int_\Omega \Delta (P\bar W_{p, j}) \big( P \bar W^*_{p, j}-P\bar W_{p, j}\big)+
O\Bigl( {
\frac{1}{p^{7-\tau}}}
\Bigr).
\end{split}
\end{equation}
For $x\in B_\delta(\xi_{j})$,
\begin{equation*}
-\Delta (P\bar W_{p,j})  =\frac{1}{  p^{\frac p{p-1}}   (\bar \e_{p,j}^{**})^{\frac2{p-1}+2}}
e^{U\bigl(\frac{x-\xi_j}{\bar  \e_{p,j}^{**} }\bigr)  } \Bigl(1+ \frac1p f_0\bigl(\frac{x-\xi_j}{\bar  \e_{p,j}^{**} }\bigr) + \frac1{p^2}
 f_1\bigl(\frac{x-\xi_j}{\bar  \e_{p,j}^{**} }\bigr)+O\bigl(\frac1{p^3}\bigr)\Bigr),
\end{equation*}
where $f_{0}$ and $f_{1}$ are defined in \eqref{6-29-8} and \eqref{7-29-8}, respectively. And
by \eqref{2-26-8}, \eqref{1-29-8} and \eqref{2-29-8}, for $x\in B_\delta(\xi_{j})$, it holds
\begin{equation}\label{luo-3}
\begin{split}
 P \bar W^*_{p, j}-P\bar W_{p, j}
=&\big(\bar W^*_{p, j}-\bar W_{p, j}\big)+O\big(\frac{|x_{p,j}-\xi_j|}{p} \big)+
O\Big(\frac{1}{  p } \big| \frac{1}{(\bar \e_{p,j}^{**})^{\frac2{p-1}}}-\frac{1}{(\bar \e_{p,j}^{*})^{\frac2{p-1}}}\big| \Big)\\&
+O\Big(\frac{1}{  p } \big| \bar \mu_{p,j}^{**} - \bar \mu_{p,j}^{*}\big| \Big)+O\big(\frac{\bar \e_{p,j}^{*}}{p}\big)\\=&
\frac{1}{  p^{\frac p{p-1}}   (\bar \e_{p,j}^{**})^{\frac2{p-1}}}\left[
V_p\bigl(\frac{x-\xi_j}{\bar  \e_{p,j}^{**} }\bigr)-V_p\bigl(\frac{x-x_{p,j}}{\bar  \e_{p,j}^{**} }\bigr)
\right]
+O\Big(\frac{1}{pe^{\frac{p}{4} }}\Big),
\end{split}
\end{equation}
where
\begin{equation*}
V_p(x):=U\bigl(x\bigr) +  \frac1p   w_0\bigl(x\bigr) + \frac1{p^2} w_1\bigl(x\bigr).
\end{equation*}
Let us point out that $V_p$ is a radial function in $\R^2$.
Hence
\begin{align}\label{tstat1}
&\int_\Omega \Delta (P\bar W_{p, j}) \big( P \bar W^*_{p, j}-P\bar W_{p, j}\big)dx\notag
\\=&-
\frac{1}{  p^{\frac {2p}{p-1}}   (\bar \e_{p,j}^{**})^{\frac4{p-1}+2}}\int_{B_\delta(\xi_j)}
e^{U\bigl(\frac{x-\xi_j}{\bar  \e_{p,j}^{**} }\bigr)  } \left[1+ \frac1p f_0\bigl(\frac{x-\xi_j}{\bar  \e_{p,j}^{**} }\bigr) + \frac1{p^2}
 f_1\bigl(\frac{x-\xi_j}{\bar  \e_{p,j}^{**} }\bigr)+O\bigl(\frac1{p^3}\bigr)\right] \notag\\&
\,\,\,\,\,\,\,\,\,\,\,\,
\times \left[
V_p\bigl(\frac{x-\xi_j}{\bar  \e_{p,j}^{**} }\bigr)-V_p\bigl(\frac{x-x_{p,j}}{\bar  \e_{p,j}^{**} }\bigr)
\right]dx
+O\Big(\frac{1}{pe^{\frac{p}{4} }}\Big)\notag
\\=&
\frac{1}{  p^{\frac {2p}{p-1}}   (\bar \e_{p,j}^{**})^{\frac4{p-1}}}\underbrace{\int_{\R^2}
e^{U\bigl(x\bigr)  } \left[1+ \frac1p f_0\bigl(x\bigr) + \frac1{p^2}
 f_1\bigl(x\bigr)+O\bigl(\frac1{p^3}\bigr)\right]
\times
\left[ V_p\bigl(x+ \frac{\xi_j-x_{p,j}}{\bar  \e_{p,j}^{**} }\bigr) -
V_p\bigl(x\bigr)
\right]dx}_{=:A_{p}}\notag \\&
+O\Big(\frac{1}{pe^{\frac{p}{4} }}\Big).
\end{align}

On the other hand, by Taylor's expansion, we have
\begin{equation}\label{tsta}
\begin{split}
&V_p\bigl(x+ \frac{\xi_j-x_{p,j}}{\bar  \e_{p,j}^{**} }\bigr) -V_p\bigl(x\bigr)\\
=&
\langle \nabla V_p\bigl(x\bigr), \frac{\xi_j-x_{p,j}}{\bar  \e_{p,j}^{**} }\rangle+\frac{1}{2}
  \nabla^2 V_p\bigl(x\bigr)\Big(\frac{\xi_j-x_{p,j}}{\bar  \e_{p,j}^{**} },\frac{\xi_j-x_{p,j}}{\bar  \e_{p,j}^{**} }\Big)
+O\Big(\big|\frac{\xi_j-x_{p,j}}{\bar  \e_{p,j}^{**} }\big|^2\Big).
\end{split}
\end{equation}
Then by \eqref{32-30-8},
 \eqref{tsta}, the symmetry of $f_0(x)$, $f_1(x)$ and $V_p(x)$, we deduce
\begin{equation}\label{tstat}
\begin{split}
A_{p}=&\int_{\R^2}
e^{U\bigl(x\bigr)  } \left[1+ \frac1p f_0\bigl(x\bigr) + \frac1{p^2}
 f_1\bigl(x\bigr)+O\bigl(\frac1{p^3}\bigr)\right]
\times
\left[ V_p\bigl(x+ \frac{\xi_j-x_{p,j}}{\bar  \e_{p,j}^{**} }\bigr) -
V_p\bigl(x\bigr)
\right]dx\\=& \frac{1}{2}
\int_{\R^2}
e^{U\bigl(x\bigr)  } \left[1+ \frac1p f_0\bigl(x\bigr) + \frac1{p^2}
 f_1\bigl(x\bigr)\right]
\times
\left[  \nabla^2 V_p\bigl(x\bigr)\Big(\frac{\xi_j-x_{p,j}}{\bar  \e_{p,j}^{**} },\frac{\xi_j-x_{p,j}}{\bar  \e_{p,j}^{**} }\Big)
\right]dx+O\Bigl(\frac{|\xi_j-x_{p,j}|}{p^3|\bar  \e_{p,j}^{**} |}\Bigr)\\=&
\left[\frac{1}{4}
\int_{\R^2}\frac{1}{\big(1+\frac{|x|^2}{8}\big)^4} +o\big(1\big)\right] \frac{|x_{p,j}-\xi_j|^2}{(\bar  \e_{p,j}^{**} )^2}+O\Bigl(\frac{|\xi_j-x_{p,j}|}{p^3|\bar  \e_{p,j}^{**} |}\Bigr).
\end{split}\end{equation}
Hence from \eqref{tstat1} and \eqref{tstat}, we get
\begin{equation}\label{sbdaa}\begin{split}
& \int_\Omega \Delta (P\bar W_{p, j}) \big( P \bar W^*_{p, j}-P\bar W_{p, j}\big)dx\\=&
 \frac{1}{  p^{\frac {2p}{p-1}}   (\bar \e_{p,j}^{**})^{\frac4{p-1}}}\left[\left(\frac{1}{4}
\int_{\R^2}\frac{1}{\big(1+\frac{|x|^2}{8}\big)^4} +o\big(1\big)\right) \frac{|x_{p,j}-\xi_j|^2}{(\bar  \e_{p,j}^{**} )^2}+O\Bigl(\frac{|\xi_j-x_{p,j}|}{p^3|\bar  \e_{p,j}^{**} |}\Bigr)\right] +O\Big(\frac{1}{pe^{\frac{p}{4} }}\Big).
\end{split}\end{equation}
On the other hand, from \eqref{21-4-9}, \eqref{33-30-8}, \eqref{luo-1mt} and \eqref{ttsma}, we have
\begin{equation}\label{sbda}
\int_\Omega \Delta (P\bar W_{p, j}) \big( P \bar W^*_{p, j}-P\bar W_{p, j}\big)=
O\Bigl( {
\frac{1}{p^{7-\tau}}}
\Bigr).
\end{equation}
Hence from \eqref{sbdaa} and \eqref{sbda}, we get
\begin{align*}
&\frac{1}{p^2}\frac{|x_{p,j}-\xi_j|^2}{(\bar  \e_{p,j}^{**} )^2}=O\Big(\frac{|x_{p,j}-\xi_j|}{p^5\bar  \e_{p,j}^{**}  }\Big)+O\Bigl( {
\frac{1}{p^{7-\tau}}}
\Bigr),
\end{align*}
which  gives $\frac{|x_{p,j}-\xi_j|}{\bar  \e_{p,j}^{**}}\leq  {\frac{C}{p^{(5-\tau)/2}}}$, that is
\begin{equation}\label{otherImprovement}|x_{p,j}-\xi_j|=O\Big(\frac{ {1}}{{p^{(5-\tau)/2}} e^{\frac{p}{4}}}\Big).
\end{equation}
So, \eqref{luo-1mt} and \eqref{otherImprovement} prove that the minimizers  $\boldsymbol{\mathrm\alpha}_{p}=(\alpha_{p,1},\cdots,\alpha_{p,k})$ and $\boldsymbol{\xi}_{p}=( \xi_{p,1},\cdots, \xi_{p,k})$ of  \eqref{32-30-8} are in the interior of the set
 { $\bigl\{  |\alpha_j-1|\le \frac{1}{p^{\frac52}},\; |\xi_j-x_{p, j}|\le \frac{1}{p^{2} e^{\frac{p}{4}}}\bigr\}$}.
As a result,
\begin{equation*}
\begin{split}
\bigl\langle  P\bar W_{p, j}, \omega_{p}\bigr\rangle = \bigl\langle  \frac{
\partial  P\bar W_{p, j}}{\partial \xi_{j,h}}, \omega_{p}\bigr\rangle=0,
\quad j=1,\cdots, k,\; h=1,2,
\end{split}
\end{equation*}
where
\[\omega_{p}:=u_p- \bar W_{\boldsymbol{\mathrm\alpha},p}.\]
 Furthermore, for any small $\tau\in (0,\tau_0)$,
\[
\|\omega_{p}\|\le \frac {C}{p^{\frac{7-\tau}2}}.
\]
By \eqref{10-7-9}, it holds
\begin{equation}\label{luo-11}
\begin{split}\|\omega_{p}\|_{L^\infty(\Omega)}=&
\|u_p- \bar W_{\boldsymbol{\mathrm\alpha},p}\|_{L^\infty(\Omega)}=\|\bar W_{\boldsymbol{\mathrm\alpha},p}- \bar W^*_{p}\|_{L^\infty(\Omega)}+O\big(\frac{1}{p^{4-\tau}} \big)\\=&
 \sum_{j=1}^k \underbrace{|\alpha_{p,j}-1|\cdot \| P\bar W_{p, j}\|_{L^\infty(\Omega)}}_{=
 O\Big({\frac{1}{p^{5/2 }}}\Big) }
+\|\sum_{j=1}^k\big( P \bar W^*_{p, j}-P\bar W_{p, j}\big)\|_{L^\infty(\Omega)}
+O\big(\frac{1}{p^{4-\tau}} \big)\\=&
\|\sum_{j=1}^k\big( P \bar W^*_{p, j}-P\bar W_{p, j}\big)\|_{L^\infty(\Omega)}
+O\Big({\frac{1}{p^{\frac{5}{2} }}}\Big).
\end{split}
\end{equation}
Also, by \eqref{luo-3}, we get
\begin{equation}\label{luo-12}
\begin{split}
\|P \bar W^*_{p, j}-P\bar W_{p, j}\|_{L^\infty(\Omega)}
=\underbrace{\|P \bar W^*_{p, j}-P\bar W_{p, j}\|_{L^\infty\big(B_\delta(\xi_j)\big)}}_{
=O\big({\frac{1}{p^{(7-\tau)/{2}}}}\big)}
+\underbrace{\|P \bar W^*_{p, j}-P\bar W_{p, j}\|_{L^\infty\big(\Omega\backslash B_\delta(\xi_j)\big)}}_{=O\big(\frac{1}{p^4}\big)}=
O\Big({\frac{1}{p^{\frac{7-\tau}{2}}}}\Big).
\end{split}
\end{equation}
Finally, \eqref{luo-11} and \eqref{luo-12} give us that
\begin{equation*}
\begin{split}\|\omega_{p}\|_{L^\infty(\Omega)}=&O\Big({\frac{1}{p^{\frac{5}{2} }}}\Big).
\end{split}
\end{equation*}
\end{proof}

\section{Estimates for the approximations}\label{section:approxPWproblems}
In this section we collect useful estimates on an approximated linearized problem (see Subsection \ref{section:approxLinearized}) and on the approximating profile $\bar W_{\boldsymbol{\mathrm\alpha},p}$ (see Subsection \ref{subsection:estimatesPW}).

\subsection{Estimates for the linear operator $-\Delta-p\bar W_{\boldsymbol{\mathrm\alpha},p}^{p-1}$}\label{section:approxLinearized}$\,$

\vskip 0.1cm

Given $\boldsymbol{\mathrm \xi}=(\xi_1,\cdots, \xi_k)\in\Omega^{k}$ close to $\boldsymbol{\mathrm x}_{\infty}=(x_{\infty,1},\cdots,x_{\infty,k})$, let $\bar \mu_{p,j}$, $\bar \varepsilon_{p,j}$ and $\bar W_{p, j}$, $j=1,\cdots, k$, be the associated parameters and profiles as defined in  \eqref{10-30-8}--\eqref{12-30-8}.

\vskip 0.1cm

Let the profile $\bar W_{\boldsymbol{\mathrm\alpha},p}$ be as in \eqref{def:Wpalpha}, namely
\[
\bar W_{\boldsymbol{\mathrm\alpha},p}:=\sum_{j=1}^k  \alpha_j P\bar W_{p, j},
\]
where, $\boldsymbol{\mathrm\alpha}$ and
$\boldsymbol{\mathrm\xi}$ satisfy $|\alpha_j-1|\le  p^{-\frac{5}{2}}$, and $|\xi_j-x_{\infty,j}|\le \tau_0$ for some small $\tau_0>0$, $j=1,\cdots, k$.
As in the previous section, in order to simplify the notation, we  write $\bar W_{p}$ for the profile $\bar W_{\boldsymbol{\mathrm\alpha},p}$ when $\boldsymbol{\mathrm\alpha}=(\alpha_{1},\cdots,\alpha_{k})=(1,\cdots,1)$, namely
\begin{equation}\label{def:potentialLinearpartial}
\bar W_{p}:= \sum_{j=1}^k  P\bar W_{p, j}.
\end{equation}
Consider
\begin{equation*}
\zeta_{p}:= \inf\biggl\{  \frac{\displaystyle\int_{\Omega} |\nabla \varphi|^2}{
\displaystyle\int_{\Omega} p\bar W_{\boldsymbol{\mathrm\alpha},p}^{p-1} \varphi^2}\, : \,
\varphi\in H^1_0(\Omega),\;\bigl\langle  P\bar W_{p, j}, \varphi\bigr\rangle = \bigl\langle  \frac{
\partial  P\bar W_{p, j}}{\partial \xi_{j,h}} , \varphi\bigr\rangle=0,
\ j=1,\cdots, k,\; h=1,2
\biggr\}.
\end{equation*}
We know that $\zeta_{p}$ is achieved by $\varphi_p$, which satisfies the linear equation
\begin{equation}\label{61-4-9}
 \begin{cases} -\Delta \varphi_p -\zeta_{p} p\bar W_{\boldsymbol{\mathrm\alpha},p}^{p-1} \varphi_p
=h_{p}~~~&\mbox{ in }\Omega,\\[2mm]
\varphi_{p}=0~~~&\mbox{ on }\partial\Omega,\end{cases}
\end{equation}
where
\begin{equation}
\label{def:h}
h_{p}:=\sum_{j=1}^k \beta_j\Delta P\bar W_{p, j}+ \sum_{j=1}^k\sum_{h=1}^2
\gamma_{jh}\Delta \frac{
\partial P \bar W_{p, j}}{\partial \xi_{j,h}}
\end{equation}
for some constants $\beta_j$ and $\gamma_{jh}$ depending on $p$. W.l.g. we may assume that $\|\varphi_p\|_{L^\infty(\Omega)}=1$.
Let us point out that for $|\alpha_j-1|\le p^{-\frac{5}{2}}$, it holds
$\alpha_j^{p-1}= 1+ O(p^{-\frac{3}{2}})$.
So we have
\begin{equation}
\label{10-6-4}
p\bar W_{\boldsymbol{\mathrm\alpha},p}^{p-1}=p\big(1+O(p^{-\frac{3}{2}})\big)\bar W_{p}^{p-1}.
\end{equation}

In this section, we will prove the following result.
\begin{Prop}\label{p10-30-8}
There exists a constant  $c_0>0$, such that
\begin{equation*}
\zeta_{p}\ge 1+\frac {c_0}p.
\end{equation*}
\end{Prop}

To prove Proposition~\ref{p10-30-8}, we will argue by contradiction. Suppose
that up to a subsequence, $\zeta_{p}\le 1+o(\frac {1}p)$  and $\|\varphi_p\|_{L^\infty(\Omega)}=1$.
We first estimate the constants  $\beta_j$ and $\gamma_{jh}$ in \eqref{61-4-9}.
\begin{Lem}\label{lemma:contantsControl}It holds
 \begin{equation}\label{22-5-9}
\beta_j = O(1),\quad \gamma_{jh}=O\bigl(
\frac{1}{pe^{\frac p4}}
\bigr),~~\mbox{
as $p\rightarrow +\infty$}.
\end{equation}
\end{Lem}
\begin{proof}
 From equation \eqref{61-4-9}, it follows that
$\beta_j$ and $\gamma_{jh}$ are determined by
the following systems
\begin{equation}\label{10-5-9}
\begin{split}
& \bigl\langle  P\bar W_{p, i}, \varphi_p\bigr\rangle -\zeta_{p} \int_\Omega  p\bar W_{\boldsymbol{\mathrm\alpha},p}^{p-1} P\bar W_{p, i} \varphi_p\\
=&\sum_{j=1}^k \beta_j\int_\Omega P\bar W_{p, i}\Delta P\bar W_{p, j}+ \sum_{j=1}^k\sum_{h=1}^2\gamma_{jh}\int_\Omega
P\bar W_{p, i}\Delta \frac{
\partial P \bar W_{p, j}}{\partial \xi_{j,h}} ,
\end{split}
\end{equation}
and
\begin{equation}\label{11-5-9}
\begin{split}
& \bigl\langle  \frac{
\partial P \bar W_{p, i}}{\partial \xi_{i,l}} , \varphi_p\bigr\rangle -\zeta_{p}  \int_\Omega  p\bar W_{\boldsymbol{\mathrm\alpha},p}^{p-1}  \frac{
\partial P \bar W_{p, i}}{\partial \xi_{i,l}}  \varphi_p\\
=&\sum_{j=1}^k \beta_j \int_\Omega\frac{
\partial P \bar W_{p, i}}{\partial \xi_{i,l}} \Delta P\bar W_{p, j}+ \sum_{j=1}^k\sum_{h=1}^2
\gamma_{jh}\int_\Omega \frac{
\partial P \bar W_{p, i}}{\partial \xi_{i,l}} \Delta \frac{
\partial P \bar W_{p, j}}{\partial \xi_{j,h}}.
\end{split}
\end{equation}
Direct computations show that
\begin{equation}\label{1-5-9}
\| P\bar W_{p, j}\|^2\sim \frac1p,\,~~~~~~
\| \frac{
\partial P \bar W_{p, j}}{\partial \xi_{j,h}} \|^2\sim \frac{e^{\frac p2}}{p},
\end{equation}
and
\begin{equation}\label{3-5-9}
\bigl\langle \bar W_{p, j},\;\frac{
\partial P \bar W_{p, j}}{\partial \xi_{j,h}}\bigr\rangle =O\Bigl( \frac1{p^4\bar \e_{p,j}}\Big)=O\Bigl(\frac{e^{\frac{p}{4}}}{p^{4}}\Big).
\end{equation}
Moreover, we have
\begin{equation}\label{50-5-9}
\bigl\langle  P\bar W_{p, i}, \varphi_p\bigr\rangle=\bigl\langle  \frac{
\partial P \bar W_{p, i}}{\partial \xi_{i,l}} , \varphi_p\bigr\rangle=0.
\end{equation}
Hence \eqref{10-5-9} and \eqref{11-5-9} become respectively
\begin{equation}\label{10-5-9BIS}
\begin{split}
 \zeta_{p} \int_\Omega  p\bar W_{\boldsymbol{\mathrm\alpha},p}^{p-1} P\bar W_{p, i} \varphi_p
&\sim
\frac{1}{p}\Bigl(\beta_i+o(1)\sum_{j\ne i}|\beta_j|\Bigr)+ O\Bigl( \frac{e^{\frac p4}}{p} \sum_{h=1}^2
\sum_{j=1}^k|\gamma_{ih}|\Big)
\end{split}\end{equation}
and
\begin{equation}\label{11-5-9BIS}
\begin{split}
  \zeta_{p}  \int_\Omega  p\bar W_{\boldsymbol{\mathrm\alpha},p}^{p-1}  \frac{
\partial P \bar W_{p, i}}{\partial \xi_{i,l}} \varphi_p
& \sim
\sum_{j=1}^k |\beta_{j}| O\Bigl( \frac{e^{\frac p4}}{p}\Big)+\frac{e^{\frac p2}}{p}
\Bigl(\gamma_{il} +o(1)
 \sum_{h\ne l, \text{or}\; j\ne i}
|\gamma_{jh}| \Bigr).
\end{split}
\end{equation}

On the other hand, for $x\in B_\delta(\xi_{i})$,
\begin{equation}\label{60-5-9}
 p\bar W_{\boldsymbol{\mathrm\alpha},p}^{p-1}(x)
 =
\frac{\alpha_j^{p-1}}{\bar \e_{p,i}^2} e^{U(\frac{x-\xi_{i}}{\bar \e_{p,{i}}})}\left(
1+\frac1p \left( w_0(\frac{x-\xi_{i}}{\bar \e_{p,{i}}}) -U(\frac{x-\xi_{i}}{\bar \e_{p,{i}}}) -\frac12 U^2(\frac{x-\xi_{i}}{\bar \e_{p,{i}}})  \right)+O\bigl(\frac{\log^4(2+\frac{|x-\xi_{i}|}{\bar \e_{p,{i}}} )}{ p^2}  \bigr)\right)
\end{equation}
and
\begin{equation}\label{61-5-9}
\begin{split}
 P\bar W_{p, i} =\frac1{ p^{\frac p{p-1}}\bar\e_{p,{i}}^{\frac2{p-1}}}
\bigl( p +O(1)\bigr)\sim 1+O\Bigl(\frac1p\Bigr).
\end{split}
\end{equation}
So using \eqref{50-5-9}  \eqref{60-5-9} and \eqref{61-5-9}, we can prove that
\begin{equation}\label{20-5-9}
\begin{split}
\int_\Omega  p\bar W_{\boldsymbol{\mathrm\alpha},p}^{p-1} P\bar W_{p, i} \varphi_p\sim &\frac1{\bar \e_{p,i}^2}\int_{ B_\delta(\xi_{i})} e^{U(\frac{x-\xi_{i}}{\bar \e_{p,{i}}})}\Bigl(
1 +O(\frac1p\bigr)\Bigr)\varphi_p\\
\sim &-\int_\Omega\Delta P\bar W_{p, {i}} \varphi_p+
O\Bigl( \frac{1  }{ p}\Bigr)\|\varphi_p\|_{L^\infty(\Omega)}=O\Bigl( \frac{1  }{ p}\Bigr).
\end{split}
\end{equation}
Similarly,  we can prove that
\begin{equation}\label{21-5-9}
\begin{split}
\int_\Omega  p\bar W_{p}^{p-1}\frac{
\partial P \bar W_{p, {i}}}{\partial \xi_{i,l}} \varphi_p
\sim &\frac1{\bar \e_{p,i}^2}\int_{ B_\delta(\xi_{i})} e^{U(\frac{x-\xi_{i}}{\bar \e_{p,{i}}})}
\Bigl(
1 +O(\frac1p\bigr)\Bigr)\frac1{p\bar \e_{p,{i}} }\frac{
\partial U(\frac{x-\xi_{i}}{\bar \e_{p,{i}}})
}{\partial \xi_{i,l}} \varphi_p +O\Bigl( \frac{1  }{\bar \e_{p,{i}} p^2}\Bigr)\\
\sim & -\frac1{p\bar \e_{p,{i}} }\int_\Omega \Delta  \frac{
\partial P \bar W_{p, i}}{\partial \xi_{i,l}} \varphi_p+
O\Bigl( \frac{1  }{\bar \e_{p,{i}} p^2}\Bigr)=O\Bigl( \frac{e^{\frac p4}  }{ p^2}\Bigr).
\end{split}
\end{equation}
Solving \eqref{10-5-9BIS} and \eqref{11-5-9BIS}, using \eqref{20-5-9}
and \eqref{21-5-9},
we find  \eqref{22-5-9}.
\end{proof}

\

\begin{proof}[{\bf Proof of Proposition~\ref{p10-30-8}}]
We will modify the proof of Proposition~3.1 in \cite{EMP2006} to prove
this proposition.
We argue by contradiction. Suppose that $\zeta_{n}:=\zeta_{p_n}\le 1+\frac1{np_n}$
for some $p_n\to +\infty$.
 We can find $\varphi_n\in H_0^1(\Omega)$, satisfying \eqref{61-4-9} for $p=p_n$, and
\[
 \bigl\langle  P\bar W_{p_n, j}, \varphi_n\bigr\rangle = \bigl\langle  \frac{
\partial  P\bar W_{p_n, j}}{\partial \xi_{j,h}} , \varphi_n\bigr\rangle=0.
\]
We may assume w.l.g. that $\|\varphi_n\|_{L^\infty(\Omega)}=1$. In order to simplify the notation we set $\bar \e_{n,i} :=\bar \e_{p_n,i}$.
Let \[\widetilde\varphi_n(y):= \varphi_{n}(\bar \e_{n,i} y+ \xi_{i}).\]
%
%
%
%
{\sl Step~1.  We show that
\begin{equation}\label{3-6-9}
\widetilde\varphi_{n}\rightarrow a Z_0(y):=a \frac{ (|y|^2-1)}{|y|^2+1}\quad\mbox{ in }C^1_{loc}(\mathbb R^2),
\end{equation}
for a constant $a\geq 0$, as $p\rightarrow +\infty$.}
\\
\\
Let us observe that $\widetilde\varphi_n$ satisfies
\begin{equation*}
\begin{split}
&-\Delta \widetilde\varphi_n - \bar \e^2_{n,i} \zeta_{n}  p_{n}(\bar W_{\boldsymbol{\mathrm\alpha_n},p_{n}})^{p_{n}-1}
  (\bar \e_{n,i} y+ \xi_{i}) \widetilde\varphi_n \\
  = &\bar \e^2_{n,i}\Bigl(  \sum_{j=1}^k \beta_j\Delta P\bar W_{p_n, j}(\bar \e_{n,i} y+ \xi_{i})+ \sum_{j=1}^k\sum_{h=1}^2
\gamma_{jh}\Delta \frac{
\partial P \bar W_{p_n, j}(\bar \e_{n,i} y+ \xi_{i})}{\partial \xi_{j,h}} \Bigr)
\end{split}
\end{equation*}
and $\|\widetilde\varphi_{n}\|_{L^{\infty(\Omega)}}=1$. In view of
\eqref{10-6-4}, \eqref{60-5-9}, Lemma \ref{lemma:contantsControl} and  $\zeta_{n}\in [0,1+\frac1{np_n}]$, we assume that up to a subsequence, $\widetilde\varphi_{n}\to \widetilde\varphi$ in $C^1_{loc}(\mathbb R^2)$, and $\widetilde\varphi$ satisfies
\begin{equation}\label{2-6-9}
\begin{split}
-\Delta \widetilde\varphi -\zeta e^{U} \widetilde\varphi=0,\quad \text{in}\; \mathbb R^2,
\end{split}
\end{equation}
where $\zeta\in [0, 1]$ is a constant.

\vskip 0.1cm

Suppose that $\widetilde\varphi\ne 0$.
Note that \eqref{2-6-9} has no bounded nontrivial solution if $t\in (0, 1)$.
So either $\zeta=0$ or $\zeta=1$.  If $\zeta=0$, $\widetilde\varphi$ is a constant. On the other hand, from
$\bigl\langle  P\bar W_{p_{n}, j},\varphi_{n}\bigr\rangle=0$, we see that
\[
\int_{\mathbb R^2} e^{U} \widetilde\varphi=0.
\]
This gives $\zeta=1$. Moreover, from $\bigl\langle  \frac{
\partial  P\bar W_{p_{n}, j}}{\partial \xi_{j,h}} , \varphi_{n}\bigr\rangle=0$, we deduce that
\[
\widetilde\varphi(y)=a Z_0(y),
\]
for a constant $a$. We may assume $a\geq 0$, otherwise we just replace $\widetilde\varphi_n$ by $-\widetilde\varphi_n$. This concludes the proof of \eqref{3-6-9}.\\

\noindent{\sl Step~2.  We show that in \eqref{3-6-9} the constant $a=0$.}\\

We will follow Step~4 in
the proof of Proposition~3.1 in \cite{EMP2006}. Hence we will find a suitable function $v\in H^1_0(\Omega)$, such that
\begin{equation}\label{4-6-9}
-\Delta v - p_{n}(\bar W_{p_{n}})^{p_{n}-1} v \approx e^{ U(\frac{x-\xi_{i}}{\bar \e_{n,i}}) } Z_0(\frac{x-\xi_{i}}{\bar \e_{n,i}}),
\end{equation}
and then we will  multiply both sides of \eqref{4-6-9} by $v$ and
integrate on $\Omega$ to obtain the result.

\vskip 0.1cm

Let us define $v$. In view of \eqref{60-5-9}, we first get the ODE solution $w$ of
\[
\Delta w+ e^{U} w= e^{U} Z_0,
\]
and the ODE solution $t$ of
\[
\Delta t+ e^{U} t= e^{U}.
\]
Then as $|y|\to +\infty$,
\[
w(y) =\frac43 \log |y| +O\bigl(\frac1{|y|}\bigr),\quad t(y)=O\bigl(\frac1{|y|}\bigr).
\]
Define
\[
v(x) :=w(\frac{x-\xi_{i}}{\bar \e_{n,i}})  +\frac43 \log \bar \e_{n,i} Z_0(\frac{x-\xi_{i}}{\bar \e_{n,i}}) +\frac{8\pi}3 H(\xi_{i}, \xi_{i}) t(\frac{x-\xi_{i}}{\bar \e_{n,i}}).
\]
Then by
\eqref{MissingEstimateBarMu}, we find
\[
Pv = v-\frac{8\pi}3 H(\cdot,\xi_{i}) +O\bigl(e^{-\frac{ p_{n}}{4}}\bigr),
\]
and
\begin{equation}\label{5-6-9}
\begin{split}
&\Delta Pv + p_{n}(\bar W_{p_{n}})^{p_{n}-1} Pv \\
= &e^{ U(\frac{x-\xi_{i}}{\bar \e_{n,i}})+2\log\frac{1}{\bar \e_{n,i}} } Z_0(\frac{x-\xi_{i}}{\bar \e_{n,i}})+
\left(p_{n}(\bar W_{p_{n}})^{p_{n}-1}-e^{ U (\frac{x-\xi_{i}}{\bar \e_{n,i}})  +2\log\frac{1}{\bar \e_{n,i}}  } \right) Pv +R,
\end{split}
\end{equation}
where
\begin{equation*}
R:= \left(Pv- v +\frac{8\pi}3 H(\xi_{i},\xi_{i})\right)e^{ U(\frac{x-\xi_{i}}{\bar \e_{n,i}})+  2\log\frac{1}{\bar \e_{n,i}}   }.
\end{equation*}
Now we multiply \eqref{5-6-9} by $\varphi_n$, integrate by parts and use \eqref{61-4-9} to obtain
\begin{equation}\label{7-6-9}
\begin{split}
&\int_\Omega e^{ U(\frac{x-\xi_{i}}{\bar \e_{n,i}})+  2\log\frac{1}{\bar \e_{n,i}}    } Z_0(\frac{x-\xi_{i}}{\bar \e_{n,i}})\varphi_n+
\int_\Omega
\left(p_{n}(\bar W_{p_{n}})^{p_{n}-1}-e^{ U (\frac{x-\xi_{i}}{\bar \e_{n,i}})  +  2\log\frac{1}{\bar \e_{n,i}}    } \right) Pv\varphi_n\\
=&\int_\Omega\Bigl((1-\zeta_{n}) p_{n}(\bar W_{\boldsymbol{\mathrm\alpha_n}, p_{n}})^{p_{n}-1} \varphi_n  +p_n\bigl(\bar W_{ p_{n}})^{p_{n}-1}-\bar W_{\boldsymbol{\mathrm\alpha_n}, p_{n}})^{p_{n}-1}
\bigr) \varphi_n- h_{p_{n}}
\Bigr) Pv  -\int_\Omega R\varphi_n,
\end{split}
\end{equation}
where $h_{p_{n}}$ is defined in \eqref{def:h}.
From the computations on pages 49 and 50 in \cite{EMP2006}, we have
\begin{equation}\label{8-6-9}
\begin{split}
\int_\Omega e^{ U(\frac{x-\xi_{i}}{\bar \e_{n,i}}) +  2\log\frac{1}{\bar \e_{n,i}}   } Z_0(\frac{x-\xi_{i}}{\bar \e_{n,i}})\varphi_n=\frac{8a\pi}3+o(1),
\end{split}
\end{equation}
\begin{equation}\label{9-6-9}
\begin{split}
\int_\Omega
\left(p_{n}(\bar W_{p_{n}})^{p_{n}-1}-e^{ U (\frac{x-\xi_{i}}{\bar \e_{n,i}})  +  2\log\frac{1}{\bar \e_{n,i}}    } \right) Pv \varphi_n=\frac{8a\pi}3+o(1)
\end{split}
\end{equation}
and
\begin{equation}\label{10-6-9}
\int_\Omega R \varphi_n=O(e^{-\frac{p_{n}}{4}}),
\end{equation}
as $n\rightarrow +\infty$. Similar to the  computations as  on page 49 in \cite{EMP2006}, we   also obtain
\begin{equation}\label{12-6-9}
\begin{split}
&\int_\Omega p_{n}(\bar W_{p_{n}})^{p_{n}-1} \varphi_n  Pv \\
=&\int_\Omega p_{n}(\bar W_{p_{n}})^{p_{n}-1} \varphi_n \Bigl( \frac43 \log \bar \e_{n,i}Z_0(\frac{x-\xi_{i}}{\bar \e_{n,i}})  +O(1)\Bigr)=  \frac43 \log \bar \e_{n,i} \Bigl(
\frac{8a\pi}3+o(1)\Bigr),
\end{split}
\end{equation}
and in view of \eqref{10-6-4}
\begin{equation*}
\begin{split}
\int_\Omega p_n\big|\bigl(\bar W_{ p_{n}})^{p_{n}-1}-\bar W_{\boldsymbol{\mathrm\alpha_n}, p_{n}})^{p_{n}-1}
\bigr) \varphi_n
 Pv\big|  = O\bigl(|\log \bar \e_{n,i}| p^{-\frac32+\tau_0}\bigr),
\end{split}
\end{equation*}
and it is also easy to see from  Lemma \ref{lemma:contantsControl}  that
\begin{equation*}
\begin{split}
\int_\Omega h_{n} Pv =O(e^{-\frac{ p_{n}}{4}}),
\end{split}
\end{equation*}
as $n\rightarrow +\infty$.  Combining \eqref{7-6-9}, \eqref{8-6-9}, \eqref{9-6-9}, \eqref{10-6-9} and \eqref{12-6-9}, we obtain
\begin{equation}\label{13-6-9}
\frac{16\pi a} 3 -\frac43 (1-\zeta_{n})\log \bar \e_{n,i} \bigl(\frac{8a\pi}3+o(1)\bigr)=o(1).
\end{equation}
Since $(1-\zeta_{n})\ge -\frac1{np_{n}}$,  and
$\bar \e_{n,i} \sim c e^{-p_{n}/4}$ by
 \eqref{MissingEstimateBarMu},  we have $-\frac43 (1-\zeta_{n})\log \bar \e_{n,i} \ge
-\frac{C}n$. Hence we obtain a contradiction from \eqref{13-6-9} in view of $a>0$.\\\\
{\sl Step~3. Conclusion. } \\\\Let
\[
\|f\|_*:=\max_{x\in\Omega}\Bigl(\sum_{j=1}^k \frac{\bar\e_{n,j}}{(\bar \e_{n,j}^2
+|x-\xi_{j}|^2)^{\frac32}}\Bigr)^{-1}|f(x)|.
\]
Arguing as in Steps~1 and 2 in the proof of Proposition~3.1
in \cite{EMP2006}, we can prove  that for $x\in \Omega\setminus \Cup_{j=1}^k
B_{R \bar \e_{n,j}}(\xi_{j})$, where $R>0$ is large, there exists a constant $C>0$ such that
\begin{equation}\label{20-6-9}
|\varphi_n(x)|\le C \max_{x\in \cup_{j=1}^k
B_{R \bar \e_{n,j}}(\xi_{j})} |\varphi_n| + C\|h_{p_{n}}\|_*,
\end{equation}
where $h_{p_{n}}$ is defined in \eqref{def:h} with $p=p_{n}$.

\vskip 0.1cm

On the other hand, we have
\[
| \beta_j\Delta P\bar W_{p_{n}, j}|\le \frac{ C}{p_{n}} \frac{\bar\e_{n,j}^2}{(\bar \e_{n,j}^2
+|x-\xi_{j}|^2)^{2}}\le \frac{ C}{p_{n}}  \frac{\bar\e_{n,j}}{(\bar \e_{n,j}^2
+|x-\xi_{j}|^2)^{\frac32}},
\]
and
\[
|\gamma_{jh}\Delta \frac{
\partial P \bar W_{p_{n}, j}}{\partial \xi_{j,h}} |\le \frac C{p_{n}^2} \frac{\bar\e_{n,j}^2}{(\bar \e_{n,j}^2
+|x-\xi_{j}|^2)^{2}}\le \frac C{p_{n}^2}  \frac{\bar\e_{n,j}}{(\bar \e_{n,j}^2
+|x-\xi_{j}|^2)^{\frac32}}.
\]
Thus, $\|h_{p_{n}}\|_*= o(1)$ as $n\rightarrow +\infty$, which, together with \eqref{20-6-9},
 gives
 \[
\|\varphi_n\|_{L^{\infty}(\Omega)}
\le C \max_{x\in \bigcup_{j=1}^k
B_{R \bar \e_{n,j}}(\xi_{j})} |\varphi_n(x)| + o(1)= C \max_{y\in
B_{R} (0)} |\widetilde\varphi_n(y)| + o(1).
\]
As a consequence, {\sl Step 1} and {\sl Step 2} imply
\[
\|\varphi_n\|_{L^{\infty}(\Omega)}=o(1),
\]
as $n\rightarrow +\infty$, a contradiction to $\|\varphi_n\|_{L^\infty(\Omega)}=1$.
\end{proof}
Similarly, we can also prove the following result.
\begin{Prop}\label{p10-7-9}
Let $\hat h\in C(\overline\Omega)$, $t\in [0, 1]$, and $\beta_j$ and $\gamma_{j,h}$ be some constants.
Let us suppose that  that $\omega$ satisfies the linear problem
\begin{equation*}
\begin{cases}
-\Delta \omega -t p\bar W_{\boldsymbol{\mathrm\alpha},p}^{p-1} \omega =\hat h+\displaystyle\sum_{j=1}^k \beta_j\Delta P\bar W_{p, j}+ \displaystyle\sum_{j=1}^k\sum_{h=1}^2
\gamma_{jh}\Delta \frac{
\partial P \bar W_{p, j}}{\partial \xi_{j,h}} &\mbox{ in }\Omega,\\
\omega=0 &\mbox{ on }\partial\Omega,
\end{cases}
\end{equation*}
where $\bar W_{p}$ is defined in \eqref{def:potentialLinearpartial}, and that
\[
\bigl\langle  P\bar W_{p, j}, \omega\bigr\rangle = \bigl\langle  \frac{
\partial  P\bar W_{p, j}}{\partial \xi_{j,h}} , \omega\bigr\rangle=0,
\quad j=1,\cdots, k,\; h=1,2.
\]
 Then for $p$ large, it holds
\[\|\omega\|_{L^\infty(\Omega)}\le C p \|\hat h\|_*,
\]
where \[
\|\hat h\|_*:=\max_{x\in\Omega}\Bigl(\sum_{j=1}^k \frac{\bar\e_{p,j}}{(\bar \e_{p,j}^2
+|x-\xi_{j}|^2)^{\frac32}}\Bigr)^{-1}|\hat h(x)|.
\]
\end{Prop}

\begin{proof}
The proof is very similar to that of Proposition~\ref{p10-30-8}. In
this proposition,   \eqref{13-6-9} is replaced by
\begin{equation*}
\frac{16\pi a} 3 =O\big(\|\hat h\|_*\log \bar \e_{p,i}\big)+o\big(1\big)=o\big(1\big).
\end{equation*}
See also \cite{EMP2006} for details.
\end{proof}

\

\subsection{Some crucial estimates} $\,$\\\label{subsection:estimatesPW}

\begin{Lem}\label{l1-8-9}
 It holds
\[
\begin{split}
&\int_\Omega \nabla \bar W_{\boldsymbol{\mathrm\alpha},p} \nabla  P\bar W_{p, j}-\int_\Omega (\bar W_{\boldsymbol{\mathrm\alpha},p})^{p}P\bar W_{p, j}=
\frac{8\pi (\alpha_j-\alpha_j^p) }{ p^{\frac p{p-1}}\bar\e_{p,j}^{\frac2{p-1}}}
+O\Bigl(\frac1{p^2} \sum_{i=1}^k |\alpha_i-
\alpha_i^p|+\frac1{p^{4}}\Bigr).
\end{split}
\]
\end{Lem}

\begin{proof}
We have
\[
\begin{split}
 \int_\Omega \nabla \bar W_{\boldsymbol{\mathrm\alpha},p} \nabla  P\bar W_{p, j}=&-
\int_\Omega \sum_{i=1}^k \alpha_i\Delta (P\bar W_{p, i})   P\bar W_{p, j}\\
 = &-\sum_{i=1}^k\int_{ B_\delta(\xi_{i})}\alpha_i\Delta (P\bar W_{p, i})   P\bar W_{p, j}+O\bigl(\frac1{p^4}\bigr).
\end{split}
\]
In $B_\delta(\xi_i)$, we have
\begin{equation}\label{60-6-4}
-\Delta (P\bar W_{p, i})  =\frac1{ p^{\frac p{p-1}}\bar\e_{p,i}^{\frac2{p-1}+2}}
e^{U (\frac{x-\xi_i}{\bar \e_{p,i}})  } \Bigl(1+ \frac1p f_0\bigl(\frac{x-\xi_i}{\bar \e_{p,i}}\bigr) + \frac1{p^2}
 f_1\bigl(\frac{x-\xi_i}{\bar \e_{p,i}}\bigr) +O\bigl(\frac1{p^3}\bigr)\Bigr),
\end{equation}
where $f_{0}$ and $f_{1}$ are defined in \eqref{6-29-8} and \eqref{7-29-8}, respectively.
Hence
\begin{equation}\label{2-8-9}
\begin{split}
&\int_\Omega \nabla \bar W_{\boldsymbol{\mathrm\alpha},p} \nabla  P\bar W_{p, j}\\
=& \sum_{i=1}^k \frac{\alpha_i}{ p^{\frac p{p-1}}\bar\e_{p,i}^{\frac2{p-1}}}
\underbrace{\frac1{\e_{p,i}^2}\int_{B_\delta(\xi_i)}
e^{U (\frac{x-\xi_i}{\bar \e_{p,i}})  }\Bigl(1+ \frac1p f_0\bigl(\frac{x-\xi_i}{\bar \e_{p,i}}\bigr) + \frac1{p^2}
 f_1\bigl(\frac{x-\xi_i}{\bar \e_{p,i}}\bigr) \Bigr)P\bar W_{p, j}
 }_{:= h_{p,i,j}  }+O\bigl(\frac1{p^4}\bigr).
 \end{split}
\end{equation}
It is easy to check that
\begin{equation*}
 h_{p,j,j}  =8\pi +O\bigl(\frac1p\bigr),\quad
  h_{p,j,j}  =O\bigl(\frac1p\bigr), \;\;i\ne j.
\end{equation*}

On the other hand,
we can also prove that in $B_\delta(\xi_i)$, it holds
\begin{equation*}
\begin{split}
(\bar W_{\boldsymbol{\mathrm\alpha},p})^p
=\frac{ \alpha_j^p  }{ p^{\frac p{p-1}}\bar\e_{p,j}^{\frac2{p-1}+2}}
e^{U (\frac{x-\xi_j}{\bar \e_{p,j}})  } \Bigl(1+ \frac1p f_0\bigl(\frac{x-\xi_j}{\bar \e_{p,j}} \bigr)+ \frac1{p^2}
 f_1\bigl(\frac{x-\xi_j}{\bar \e_{p,j}})  \bigr)+O\bigl(\frac1{p^3}\bigr)\Bigr).
\end{split}
\end{equation*}
Hence we have
\begin{equation}\label{3-8-9}
\int_\Omega (\bar W_{\boldsymbol{\mathrm\alpha},p})^{p}P\bar W_{p, j}=
\sum_{i=1}^k \frac{\alpha_i^p h_{p,i,j}}{ p^{\frac p{p-1}}\bar\e_{p,i}^{\frac2{p-1}}}+O\bigl(\frac1{p^{4}}\bigr).
\end{equation}

Using  \eqref{2-8-9} and \eqref{3-8-9}, we find
\begin{equation*}
\begin{split}
& -\int_{ B_\delta(\xi_i)}\sum_{i=1}^k \alpha_i\Delta (P\bar W_{p, i})   P\bar W_{p, j}-\int_{ B_\delta(\xi_j)} (\bar W_{\boldsymbol{\mathrm\alpha},p})^pP\bar W_{p, j}\\
=& \sum_{i=1}^k\frac{(\alpha_i-\alpha_i^p) h_{p,i,j}}{ p^{\frac p{p-1}}\bar\e_{p,j}^{\frac2{p-1}}}+O\bigl(\frac1{p^{4}}\bigr)=
\frac{8\pi (\alpha_j-\alpha_j^p) }{ p^{\frac p{p-1}}\bar\e_{p,j}^{\frac2{p-1}}}
+O\bigl(\frac1{p^2} \sum_{i=1}^k |\alpha_i-
\alpha_i^p|+\frac1{p^{4}}\bigr).
\end{split}
\end{equation*}
\end{proof}
Let us define the cut-off function \begin{equation}\label{def:Eta-theta}\eta_j(y):=\eta\big(|y-\xi_j|\big),\end{equation} where
\[
\eta(t)  \begin{cases}
=1 &\mbox{ if }|t|\leq d,\\
=0&\mbox{ if }|t|\ge d+\theta,\\
\in [0,1]& \mbox{ if }d\leq|t| \leq d+\theta,
\end{cases}
 \qquad \eta'(t)\sim -\frac1 \theta\mbox{ for }t\in (d, d+\theta),\]
for a fixed $\theta>0$.
\begin{Lem}\label{l3-8-9}
 It holds
\[
\begin{split}
&\int_{\Omega} \nabla \bar W_{\boldsymbol{\mathrm\alpha},p} \nabla (\eta_j\frac{\partial \bar W_{\boldsymbol{\mathrm\alpha},p}}{
\partial y_h  })-\int_{\Omega}
(\bar W_{\boldsymbol{\mathrm\alpha},p})^p \eta_j\frac{\partial \bar W_{\boldsymbol{\mathrm\alpha},p}}{
\partial y_h  }\\
=& \frac{64\pi^2 e}{p^2}\frac{\partial \Psi_k(\xi_1,
\cdots,\xi_k)}{\partial \xi_{jh}} +O\Bigl(\frac \theta {p^2}+\frac1{p^3}+\sum_{j=1}^k\frac{|\alpha_j-1|}p\Bigr),
\end{split}
\]
where $\theta>0$ is the constant in \eqref{def:Eta-theta}.
\end{Lem}

\begin{proof}
It is easy to see that
\begin{equation}\label{31-9-9}
\int_{\Omega}
(\bar W_{\boldsymbol{\mathrm\alpha},p})^p \eta_j\frac{\partial (\bar W_{\boldsymbol{\mathrm\alpha},p})}{
\partial y_h  }
= O(\bar \e_{p,j}^2).
\end{equation}
On the other hand,
\begin{equation}\label{32-9-9BIS}
\begin{split}
&\int_{\Omega} \nabla \bar W_{\boldsymbol{\mathrm\alpha},p} \nabla (\eta_j\frac{\partial \bar W_{\boldsymbol{\mathrm\alpha},p}}{
\partial y_h  })=\int_{\Omega} \eta_j\nabla \bar W_{\boldsymbol{\mathrm\alpha},p} \nabla \frac{\partial \bar W_{\boldsymbol{\mathrm\alpha},p}}{
\partial y_h  }+\int_{\Omega}\frac{\partial \bar W_{\boldsymbol{\mathrm\alpha},p}}{
\partial y_h  } \nabla \bar W_{\boldsymbol{\mathrm\alpha},p} \nabla\eta_j \\
=&\int_{B_{d+\theta}(\xi_j)} \nabla \bar W_{\boldsymbol{\mathrm\alpha},p} \nabla \frac{\partial \bar W_{\boldsymbol{\mathrm\alpha},p}}{
\partial y_h  }+\int_{B_{d+\theta}(\xi_j)\setminus B_{d}(\xi_j)}\frac{\partial \bar W_{\boldsymbol{\mathrm\alpha},p}}{
\partial y_h  } \nabla \bar W_{\boldsymbol{\mathrm\alpha},p} \nabla\eta_j+O\bigl(\frac{\theta}{p^2}\bigr)
\\
=&\frac12 \int_{\partial B_{d+\theta}(\xi_j)}|\nabla \bar W_{\boldsymbol{\mathrm\alpha},p}|^2 \nu_h-
\int_{\partial B_{d}(\xi_j)}\frac{\partial \bar W_{\boldsymbol{\mathrm\alpha},p}}{
\partial y_h  } \nabla \bar W_{\boldsymbol{\mathrm\alpha},p} +O\bigl(\frac{\theta}{p^2}\bigr).
\end{split}
\end{equation}
Here, we have used
\[
\int_{B_{d+\theta}(\xi_j)\setminus B_{d}(\xi_j)}\frac{\partial \bar W_{\boldsymbol{\mathrm\alpha},p}}{
\partial y_h  } \nabla \bar W_{\boldsymbol{\mathrm\alpha},p} \nabla\eta_j=
-
\int_{\partial B_{d}(\xi_j)}\frac{\partial \bar W_{\boldsymbol{\mathrm\alpha},p}}{
\partial y_h  } \nabla \bar W_{\boldsymbol{\mathrm\alpha},p} +O\bigl(\frac{\theta}{p^2}\bigr).
\]
But on $\partial B_{d+\theta}(\xi_j)$,
\[
\begin{split}
&\bar W_{\boldsymbol{\mathrm\alpha},p}(y)= {\sum_{i=1}^k}\frac{8\pi}{ p^{\frac p{p-1}}\bar\e_{p,i}^{\frac2{p-1}}}
\Bigl(\alpha_i G(y, \xi_i)+O(\bar \e_{p,i}^2)\Bigr)\\
=&\frac{8\pi \sqrt e}p {\sum_{i=1}^k}G(y, \xi_i)+ O\Bigl(\frac1p \sum_{i=1}^k
|\alpha_i-1|+\frac1{p^2}+\bar \e_{p,i}^2
\Bigr).
\end{split}
\]
So by \eqref{31-9-9} and \eqref{32-9-9BIS}, we obtain
\begin{align*}
&\int_{\Omega} \nabla \bar W_{\boldsymbol{\mathrm\alpha},p} \nabla (\eta_j\frac{\partial \bar W_{\boldsymbol{\mathrm\alpha},p}}{
\partial y_h  })-\int_{\Omega}
(\bar W_{\boldsymbol{\mathrm\alpha},p})^p \frac{\partial (\eta_j\bar W_{\boldsymbol{\mathrm\alpha},p})}{
\partial y_h  }
\\
= & \int_{\Omega} \nabla \bar W_{\boldsymbol{\mathrm\alpha},p} \nabla (\eta_j\frac{\partial \bar W_{\boldsymbol{\mathrm\alpha},p}}{
\partial y_h  })+O(\bar \e_{p,j}^2)
\\
=&\frac12 \int_{\partial B_{d+\theta}(\xi_j)}|\nabla \bar W_{\boldsymbol{\mathrm\alpha},p}|^2 \nu_h-
\int_{\partial B_{d}(\xi_j)}\frac{\partial \bar W_{\boldsymbol{\mathrm\alpha},p}}{
\partial y_h  } \nabla \bar W_{\boldsymbol{\mathrm\alpha},p} +O\bigl(\frac{\theta}{p^2}\bigr)
\\
=&\frac{64\pi^2 e}{p^2}\left[
\frac12 \int_{\partial B_{d+\theta}(\xi_j)}|\nabla \sum_{j=1}^k G(y, \xi_j)|^2 \nu_h-
\int_{\partial B_{d}(\xi_j)}\frac{\partial \sum_{j=1}^k G(y, \xi_j)}{
\partial y_h  } \nabla \sum_{j=1}^k G(y, \xi_j)
\right]\\
&+O\Bigl(\frac \theta {p^2}+\frac1{p^3}+\sum_{j=1}^k\frac{|\alpha_j-1|}p\Bigr)
\\
=& \frac{64\pi^2 e}{p^2}\frac{\partial \Psi_k(\xi_1,
\cdots,\xi_k)}{\partial \xi_{jh}} +O\Bigl(\frac \theta {p^2}+\frac1{p^3}+\sum_{j=1}^k\frac{|\alpha_j-1|}p\Bigr).
\end{align*}
\end{proof}
\
\section{The degree counting formula}\label{section:degreeAndUniqueness}

It follows from Proposition~\ref{prop:upHasTheFormInEMP} that
 any  $k$-spike solution $u_p$ of problem \eqref{1.1} has the form
 \[
 u_p=\bar W_{\boldsymbol{\mathrm\alpha_{p}},p}+\omega_{p}
 \]
   for some  $(\boldsymbol{\mathrm\alpha}_{p}, \boldsymbol{\xi}_{p}, \omega_{p})$  in the interior of $  S_{p}$, where the set $S_p$ is defined in \eqref{35-30-8}.
We consider the functional $I_{p}:H^{1}_{0}(\Omega)\rightarrow \mathbb R$
\[
I_{p}(u):=\frac12 \int_\Omega |\nabla u|^2 -\frac{1}{p+1}\int_\Omega |u|^{p+1},
\]
whose critical points are solutions of problem \eqref{1.1}. In this section,
we will compute the Leray-Schauder degree of $I'_p$ in $S_p$.

\vskip 0.1cm

We define the following maps on $(\boldsymbol{\mathrm\alpha}, \boldsymbol{\mathrm\xi}, \omega)\in S_{p}$, where the set $S_{p}$ is defined in \eqref{35-30-8}:
\begin{equation*}
\begin{split}
A_{1,j}(\boldsymbol{\mathrm\alpha}, \boldsymbol{\mathrm\xi}, \omega):=\bigl\langle I_p' (\bar W_{\boldsymbol{\mathrm\alpha},p} +\omega),     P\bar W_{p, j}\bigr\rangle,
\end{split}
\end{equation*}
\begin{equation}\label{31-7-9}
\begin{split}
A_{2,j,h}(\boldsymbol{\mathrm\alpha}, \boldsymbol{\mathrm\xi}, \omega):=\bigl\langle I_p' (\bar W_{\boldsymbol{\mathrm\alpha},p}+\omega),\eta_j \frac{
\partial (\bar W_{\boldsymbol{\mathrm\alpha},p}+\omega)}{\partial y_{h}} \bigr\rangle,
\end{split}
\end{equation}
for $j=1,\cdots, k$, $ h=1,2$, where $\bar W_{\boldsymbol{\mathrm\alpha},p}$ and $P\bar W_{p, j}$ are as in
\eqref{12-30-8} and \eqref{def:Wpalpha}, respectively, and $\eta_j$ is the cut-off function around $\xi_{j}$ defined in \eqref{def:Eta-theta}.
Moreover we also consider the map $A_{3}(\boldsymbol{\mathrm\alpha}, \boldsymbol{\mathrm\xi}, \omega)\in E_{p}$, which is determined by
\begin{equation*}
\begin{split}
\bigl\langle A_{3}(\boldsymbol{\mathrm\alpha}, \boldsymbol{\mathrm\xi}, \omega), \widetilde\omega\bigr\rangle:=\bigl\langle I_p' (\bar W_{\boldsymbol{\mathrm\alpha},p}+\omega),  \widetilde\omega \bigr\rangle,\quad \forall\; \widetilde\omega\in E_{p},
\end{split}
\end{equation*}
where we denote
\begin{equation}\label{def:Ep}
E_{p}:=\Bigl\{\widetilde\omega\in H^1_0(\Omega):\ \bigl\langle  P\bar W_{p, j}, \widetilde\omega\bigr\rangle = \bigl\langle  \frac{
\partial  P\bar W_{p, j}}{\partial \xi_{j,h}} , \widetilde\omega\bigr\rangle=0,
\ j=1,\cdots, k,\ h=1,2\Bigr\}.
\end{equation}
\begin{Prop}\label{prop:equivalence}
The existence of any solution $u_p$ concentrating at the $k$ distinct points $x_{\infty,1},
\cdots, x_{\infty, k}\in\Omega$ is equivalent to
the existence of $(\boldsymbol{\mathrm\alpha},\boldsymbol{\mathrm\xi}, \omega)\in S_{p}$, such that
\[
A(\boldsymbol{\mathrm\alpha},\boldsymbol{\mathrm \xi}, \omega):= \big(A_{1,j}(\boldsymbol{\mathrm\alpha},\boldsymbol{\mathrm\xi}, \omega), A_{2,j, h}(\boldsymbol{\mathrm\alpha}, \boldsymbol{\mathrm\xi}, \omega), A_{3}(\boldsymbol{\mathrm\alpha}, \boldsymbol{\mathrm\xi},\omega)\big)
=0.
\]
And in such a case \[u_{p}=\bar W_{\boldsymbol{\mathrm\alpha},p}+\omega.\]
\end{Prop}

\begin{proof}
By Proposition \ref{prop:upHasTheFormInEMP} we know that if $u_p$ is a solution of problem \eqref{1.1} concentrating at $x_{\infty,1},\cdots,x_{\infty,k}$ then, for $p$ sufficiently large,
there exists $(\boldsymbol{\mathrm\alpha}, \boldsymbol{\xi}, \omega)\in S_{p}$, such that
\[
u_p=\bar W_{\boldsymbol{\mathrm\alpha},p}+\omega,
\]
obviously $A(\boldsymbol{\mathrm\alpha},\boldsymbol{\mathrm \xi}, \omega)=0$.

\vskip 0.1cm

 Next we show that the opposite is also true. Indeed $A_3=0$ implies
\begin{equation}\label{1-9-9}
-\Delta (\bar W_{\boldsymbol{\mathrm\alpha},p}+\omega)- (\bar W_{\boldsymbol{\mathrm\alpha},p}+\omega)^p
=\sum_{j=1}^k \beta_j\Delta P\bar W_{p, j}+ \sum_{j=1}^k\sum_{h=1}^2
\gamma_{jh}\Delta \frac{
\partial P \bar W_{p, j}}{\partial \xi_{j,h}} \quad\mbox{ in }\Omega,
\end{equation}
for some constants $\beta_j$ and $\gamma_{jh}$ (which depend on $p$).

\vskip 0.1cm

On the other hand, $A_{1,i}=0$ and $A_{2,l,m}=0$ with \eqref{1-9-9}, give
\begin{equation}\label{2-9-9}
\sum_{j=1}^k \beta_j\int_\Omega P\bar W_{p, i}\Delta P\bar W_{p, j}+ \sum_{j=1}^k\sum_{h=1}^2\int_\Omega
\gamma_{jh}P\bar W_{p, i}\Delta \frac{
\partial P \bar W_{p, j}}{\partial \xi_{j,h}} =0,
\end{equation}
and
\begin{equation}\label{3-9-9}
\sum_{j=1}^k \beta_j \int_\Omega\eta_l\frac{
\partial (\bar W_{\boldsymbol{\mathrm\alpha},p}+\omega)}{\partial y_{m}}\Delta P\bar W_{p, j}+ \sum_{j=1}^k\sum_{h=1}^2
\gamma_{jh}\int_\Omega \eta_l\frac{
\partial (\bar W_{\boldsymbol{\mathrm\alpha},p}+\omega)}{\partial y_{m}}\Delta \frac{
\partial P \bar W_{p, j}}{\partial \xi_{j,h}} =0.
\end{equation}
With direct computations (see \eqref{1-5-9} and \eqref{3-5-9}), \eqref{2-9-9} gives
\begin{equation}\label{2-9-9NEW}
\beta_i+o\bigl(\sum_{j=1}^k |\beta_j|\bigr)+ O\Bigl( \frac1{p\bar \e_{p,i}}\Big)\sum_{j=1}^k\sum_{h=1}^2
|\gamma_{jh}|
=0.
\end{equation}
Furthermore  for $(\boldsymbol{\mathrm\alpha},\boldsymbol{\mathrm\xi},\omega)\in S_{p}$, using the symmetry of the function $\eta_l\Delta P\bar W_{p, {j}}$, we can prove that

\[
\int_\Omega\eta_l\frac{
\partial (\bar W_{\boldsymbol{\mathrm\alpha},p}+\omega)}{\partial y_{m}}\Delta P\bar W_{p, j}
=-\int_\Omega(\bar W_{\boldsymbol{\mathrm\alpha},p}+\omega)\frac{
\partial }{\partial y_{m}}[\eta_l\Delta P\bar W_{p, {j}}]=o\bigl(\frac1{p\bar \e_{p,l}}\bigr),
\]
and
\[
\begin{split}
&\int_\Omega \eta_l \frac{
\partial (\bar W_{\boldsymbol{\mathrm\alpha},p}+\omega)}{\partial y_{m}}\Delta \frac{
\partial P \bar W_{p, j}}{\partial \xi_{j,h}} \\
=&-\int_\Omega(\bar W_{\boldsymbol{\mathrm\alpha},p}+\omega)\frac{
\partial }{\partial y_{m}}\Bigl[\eta_l \Delta \frac{
\partial P \bar W_{p, {j}}}{\partial \xi_{j,h}} \Bigr]=
\frac{c_0+o(1)}{\bar\e_{p,{l}}^2} \delta_{hm},\quad c_0\ne 0,
\end{split}
\]where $\delta_{hm}$ denotes the Kronecker's symbol. Hence \eqref{3-9-9} gives
\begin{equation}\label{3-9-9NEW}
\sum_{j=1}^k |\beta_l| o\bigl(\frac1{p}\bigr)+ \frac{c_0+o(1)}{\bar\e_{p,l}}
\gamma_{lm}+ o( e^{\frac p4})\sum_{j=1}^k \sum_{h=1}^2
|\gamma_{jh}| =0.
\end{equation}
Thus, solving \eqref{2-9-9NEW} and \eqref{3-9-9NEW} yields $\beta_j=0$,
$\gamma_{jh}=0$.

\vskip 0.1cm

As a consequence, by \eqref{1-9-9}, we have that $u_{p}:=\bar W_{\boldsymbol{\mathrm\alpha},p}+\omega$ is a solution of problem \eqref{1.1}. Furthermore it concentrates  at $(x_{\infty,1},\cdots,x_{\infty,k})$, since $(\boldsymbol{\mathrm\alpha}, \boldsymbol{\xi}, \omega)\in S_{p}$,
and $(x_{\infty,1},\cdots,x_{\infty,k})$ is a non-degenerate critical of
the Kirchhoff-Routh function $\Psi_{k}$ (and thus
is isolated).

\end{proof}
\begin{Rem}
Let us point out that in the right hand side of \eqref{31-7-9} we do not use
\[
\bigl\langle I_p' (\bar W_{\boldsymbol{\mathrm\alpha},p}+\omega),
\frac{\partial P \bar W_{p, j}}{\partial \xi_{j,h}}   \bigr\rangle,
\]
because this term involves $\displaystyle\int_\Omega\bar W_{\boldsymbol{\mathrm\alpha},p}^{p-2}\omega^2 \frac{\partial P \bar W_{p, j}}{\partial \xi_{j,h}}$, which
is extremely difficult to control.
Notice that, by using the integration by parts, we can avoid the estimates
near the concentration point $\xi_j$ for the term in the right
hand side of \eqref{31-7-9}.
\end{Rem}

We have the following degree formula for the map $A$.
\begin{Thm}\label{th1-9-9}
Let $\boldsymbol{\mathrm x}_{\infty}:=(x_{\infty,1},
\cdots, x_{\infty, k}) \in \Omega^{k}$, $x_{\infty,i}\neq x_{\infty,j}$ if $i\neq j$,  be a nondegenerate critical point for the Kirchhoff-Routh function $\Psi_{k}$.
Then we have
\[
\deg(A,0,S_{p}) =(-1)^{k+m(\boldsymbol{\mathrm x}_{\infty},\Psi_k)},
\]
for $p$ large.
\end{Thm}

\begin{Rem}
Theorem \ref{theorem:topological-degree-formula} is an obvious consequence of Theorem \ref{th1-9-9} and Proposition \ref{prop:equivalence}.
\end{Rem}

The rest of the section is devoted to the proof of Theorem \ref{th1-9-9}. In order to compute $\deg(A,0,S_{p})$, we will reduce to a simpler finite dimensional map $B$ homotopic to $A$ and compute the degree of $B$.

\vskip 0.1cm

Let us define the  map
\begin{equation}\label{def:B}
B(\boldsymbol{\mathrm\alpha},\boldsymbol{\mathrm \xi}, \omega):= \big(B_{1,j}(\boldsymbol{\mathrm\alpha},\boldsymbol{\mathrm \xi}), B_{2,j, h}(\boldsymbol{\mathrm\alpha},\boldsymbol{\mathrm\xi}), B_{3}(\omega)\big)
\end{equation} on $(\boldsymbol{\mathrm\alpha}, \boldsymbol{\mathrm\xi}, \omega)\in S_{p}$, as follows
\begin{equation*}
B_{1,j}(\boldsymbol{\mathrm\alpha}, \boldsymbol{\mathrm\xi}):=\bigl\langle I_p' (\bar W_{\boldsymbol{\mathrm\alpha},p} ),     P\bar W_{p, j}\bigr\rangle,~~~~~
B_{2,j,h}(\boldsymbol{\mathrm\alpha},\boldsymbol{\mathrm\xi}):=\bigl\langle I_p' (\bar W_{\boldsymbol{\mathrm\alpha},p}), \eta_j \frac{
\partial (\bar W_{\boldsymbol{\mathrm\alpha},p})}{\partial y_{h}}\bigr\rangle,
\end{equation*}
for $j=1,\cdots, k$, $ h=1,2$, and $B_{3}( \omega)\in E_{p}$ determined by
\begin{equation*}
\bigl\langle B_{3}( \omega), \widetilde\omega\bigr\rangle:=\bigl\langle \omega,  \widetilde\omega \bigr\rangle,\quad \forall\; 	\widetilde\omega\in E_{p},
\end{equation*}
where $E_{p}$ is the same set in \eqref{def:Ep}.

\vskip 0.1cm
It is obvious that the computation of the degree of $B$ in $S_{p}$ is a finite
dimensional problem. In fact, we have
\begin{equation}\label{BBarB-SameDeg}
\deg(B, 0, S_{p})=\deg ( \bar B, 0, D_{p}),
\end{equation}
where
\begin{equation}\label{defbarB}
\bar B (\boldsymbol{\mathrm\alpha}, \boldsymbol{\mathrm\xi}):= \big(B_{1,j}(\boldsymbol{\mathrm\alpha}, \boldsymbol{\mathrm\xi}), B_{2,j, h}(\boldsymbol{\mathrm\alpha}, \boldsymbol{\mathrm\xi})\big),
\end{equation}
and
\begin{equation}\label{def:Dp}
D_{p}:=  \Bigl\{ (\boldsymbol{\mathrm\alpha}, \boldsymbol{\mathrm\xi})\in \mathbb R^{k}\times \Omega^{k}\ :  \;  |\alpha_j-1|\le
\frac1{p^{\frac{5}{2}}},\, |\xi_j-x_{\infty,j}|\le
\tau_0\Bigr\}.
\end{equation}
We prove that the maps $A$ and $B$ are homotopic in $S_{p}$.
\begin{Prop}\label{p20-7-9}
Let $A_t:=  t A  +(1-t)B$.  Then  for any $t\in [0, 1]$, $A_t\ne 0$ on $\partial S_{p}$.
\end{Prop}

\begin{proof}
We argue by contradiction.  Suppose that there exist
$t\in [0, 1]$ and $(\boldsymbol{\mathrm\alpha}, \boldsymbol{\mathrm\xi}, \omega)\in \partial S_{p}$, such that
\begin{equation*}
A_t(\boldsymbol{\mathrm\alpha}, \boldsymbol{\mathrm\xi}, \omega)=0.
\end{equation*}
Then from $t A_3  +(1-t)B_3=0$, we know
\begin{equation}\label{34-7-9}
\begin{split}
 -t\Delta (\bar W_{\boldsymbol{\mathrm\alpha},p}+\omega)- t(\bar W_{\boldsymbol{\mathrm\alpha},p}+\omega)^p-(1-t)\Delta\omega
= \sum_{j=1}^k \beta_j\Delta P\bar W_{p, j}+ \sum_{j=1}^k\sum_{h=1}^2
\gamma_{jh}\Delta \frac{
\partial P \bar W_{p, j}}{\partial \xi_{j,h}},
\end{split}
\end{equation}
for some constant $\beta_j$ and $\gamma_{jh}$. We rewrite \eqref{34-7-9}
as
\begin{equation}\label{35-7-9}
-\Delta \omega- tp (\bar W_{\boldsymbol{\mathrm\alpha},p})^{p-1}\omega+ l +N(\omega)
=\sum_{j=1}^k \beta_j\Delta P\bar W_{p, j}+ \sum_{j=1}^k\sum_{h=1}^2
\gamma_{jh}\Delta \frac{
\partial P \bar W_{p, j}}{\partial \xi_{j,h}},
\end{equation}
where
\begin{equation*}
l:=t\Big[-\Delta \bar W_{\boldsymbol{\mathrm\alpha},p} - (\bar W_{\boldsymbol{\mathrm\alpha},p})^p\Big],
\end{equation*}
and
\begin{equation*}
N(\omega) :=-t \Bigl[(\bar W_{\boldsymbol{\mathrm\alpha},p}+\omega)^p -(\bar W_{\boldsymbol{\mathrm\alpha},p})^p -p(\bar W_{\boldsymbol{\mathrm\alpha},p})^{p-1}\omega   \Bigr].
\end{equation*}
{\sl{Step 1.
We show that
\begin{equation}\label{l2-7-9}
\|\omega\|=
O\Bigl(\frac{1}{p^{\frac{9}{4}+\frac{\tau_0}{2}}}\Bigr), \;\; \|\omega\|_{L^\infty(\Omega)},\; \|\nabla \omega\|_{L^\infty\big(\Omega
\setminus \cup_{j=1}^k B_{\frac d 2}(\xi_j)\big)}
= O\Bigl(\frac1{p^{3}}\Bigr).
\end{equation}

}}

Let us observe that \eqref{35-7-9} implies that
 \begin{equation}\label{38-7-9}
\| \omega\|^2- tp\int_\Omega (\bar W_{\boldsymbol{\mathrm\alpha},p})^{p-1}\omega^2+\int_\Omega\big( l +N(\omega)\bigr)
\omega=0.
\end{equation}
On one hand,  Proposition~\ref{p10-30-8} gives that
\begin{equation}\label{39-7-9}
\| \omega\|^2- tp\int_\Omega (\bar W_{\boldsymbol{\mathrm\alpha},p})^{p-1}\omega^2
\ge \bigl(1-\frac{t}{1+\frac{c_0}p}\bigr)\|\omega\|^2
\ge \frac{c_0}{p+c_0} \|\omega\|^2.
\end{equation}

Next, using \eqref{60-6-4} and \eqref{61-6-4}, we find
\[
\int_{\Omega}|l|=O\Bigl( \frac1p \sum_{j=1}^k |\alpha_j-\alpha_j^p|+\frac1{p^4}  \Bigr)=O\Bigl( \frac1{p^{\frac{5}{2}  }}  \Bigr),
\]
which, together with $\|\omega\|_{L^\infty(\Omega)}  $
\begin{equation}\label{40-7-9}
\big|\int_\Omega l
\omega\,\big|\le C \|\omega\|_{L^\infty(\Omega)}\int_{\Omega}|l|\le\frac
 C{p^{\frac{9}{2}+\tau_0  }},
\end{equation}
and
\begin{equation}\label{41-7-9}
|\int_\Omega N(\omega)
\omega\,|\le C \|\omega\|^3_{L^\infty(\Omega)}p(p-1)\int_{\Omega}(\bar W_{\boldsymbol{\mathrm\alpha},p})^{p-2}\le
C p\|\omega\|^3_{L^\infty(\Omega)}
\le\frac C{p^{5+3\tau_0 }}.
\end{equation}
Combining \eqref{38-7-9}, \eqref{39-7-9}, \eqref{40-7-9} and \eqref{41-7-9}, we obtain
\[
\|\omega\| \le \frac{C}{p^{\frac{9}{4}+\frac{\tau_0}{2}}}.
\]

Now we use  Proposition~\ref{p10-7-9}  to obtain
\begin{equation}\label{ll1}
\|\omega\|_{L^\infty(\Omega)}\le Cp \big(\|l\|_*+\|N(\omega)\|_*\big).
\end{equation}
Furthermore, we have
\begin{equation}\label{ll2}\|l\|_*\leq  \frac{C}{p^{4}},\end{equation}
(see \cite[Proposition 2.1]{EMP2006}), and
\begin{equation}\label{ll3}\|N(\omega)\|_*\leq Cp\|\omega\|_{L^{\infty}(\Omega)}^{2} \leq \frac{C}{p^{1+\tau_{0}}}\|\omega\|_{L^{\infty}(\Omega)},\end{equation}
(see for instance the proof of \cite[Lemma 4.1]{EMP2006}).
 Hence from \eqref{ll1}, \eqref{ll2} and \eqref{ll3}, we get
 \[
\|\omega\|_{L^\infty(\Omega)} \le \frac C{p^{3}}.
\]
As for the estimate of $\nabla\omega$, for $x_0\in \Omega
\setminus \Cup_{j=1}^k B_{\frac d 2}(\xi_j)$,
\[
\|\omega\|_{W^{2,3}\big(B_{\frac d8}(x_0)\big)}\le C\Big(
\|\omega\|_{L^\infty\big(B_{\frac d4}(x_0)\big)}+|f|_{L^3(B_{\frac d4}\big(x_0)\big)}\Big) \le
\frac{C}{p^{3}},
\]
where
\[
f:=-tp \bar W_{\boldsymbol{\mathrm\alpha},p}^{p-1}\omega- l -N(\omega)+\sum_{j=1}^k \beta_j\Delta P\bar W_{p, j}+ \sum_{j=1}^k\sum_{h=1}^2
\gamma_{jh}\Delta \frac{
\partial P \bar W_{p, j}}{\partial \xi_{j,h}}.
\]
This concludes the proof of \eqref{l2-7-9}.\\

{\sl Step 2.
We show that
\begin{equation}\label{l2-8-9}
\alpha_j=1+ O\bigl(\frac1{p^{3}}\bigr).
\end{equation}
}
By {\sl Step 1} we have
 $\|\omega\|_{L^\infty(\Omega)}= O\bigl(\frac1{p^{3} } \bigr)$, hence  \begin{equation}\label{20-7-9}
\begin{split}
A_{1,j}=&\bigl\langle I_p' (\bar W_{\boldsymbol{\mathrm\alpha},p} +\omega),     P\bar W_{p, j}\bigr\rangle\\
=& \int_\Omega \nabla \bar W_{\boldsymbol{\mathrm\alpha},p} \nabla  P\bar W_{p, j}-\int_\Omega (\bar W_{\boldsymbol{\mathrm\alpha},p} +\omega)^pP\bar W_{p, j}\\
=& \int_\Omega \nabla \bar W_{\boldsymbol{\mathrm\alpha},p} \nabla  P\bar W_{p, j}-\int_\Omega (\bar W_{\boldsymbol{\mathrm\alpha},p})^{p}P\bar W_{p, j}
-p\int_\Omega (\bar W_{\boldsymbol{\mathrm\alpha},p})^{p-1}P\bar W_{p, j}\omega + O\Bigl(p\|\omega\|^2_{L^\infty(\Omega)}\Bigr) \\
=&\int_\Omega \nabla \bar W_{\boldsymbol{\mathrm\alpha},p} \nabla  P\bar W_{p, j}-\int_\Omega (\bar W_{\boldsymbol{\mathrm\alpha},p})^{p}P\bar W_{p, j} +
O\bigl(\frac1{p^{3}}\bigr)
\\=
&B_{1,j}+
O\bigl(\frac1{p^{3}}\bigr).
\end{split}
\end{equation}
Since $tA_{1}+(1-t)B_{1}=0$,  \eqref{20-7-9} implies that
\begin{equation}\label{20-7-9BIS}
B_{1,j}=\int_\Omega \nabla \bar W_{\boldsymbol{\mathrm\alpha},p} \nabla  P\bar W_{p, j}-\int_\Omega (\bar W_{\boldsymbol{\mathrm\alpha},p})^{p}P\bar W_{p, j}=
O\bigl(\frac1{p^{3}}\bigr).
\end{equation}
Thus, Lemma~\ref{l1-8-9} and \eqref{20-7-9BIS} give \[
\alpha_j-\alpha_j^p =O\bigl(\frac1{p^{2}}\bigr),
\]
from which, we obtain
\begin{equation*}
\begin{split}
1-\alpha_j^{p-1}=O\bigl(\frac1{p^{2}}\bigr).
\end{split}
\end{equation*}
But $\alpha_j^{p-1}= 1+(p-1)(\alpha_j-1)+O\big(p^2(\alpha_j-1)^2\big)$. So we see
$\alpha_j-1=O\bigl(\frac1{p^{3}}\bigr),$ namely \eqref{l2-8-9}.\\

\

\noindent{\sl Step 3.
We show that
\begin{equation*}
|\xi_j-x_{\infty, j}|\to 0.
\end{equation*}
}

Since by {\sl Step 1},
\[
\|\nabla \omega\|_{L^\infty\big(\Omega
\setminus \cup_{j=1}^k B_{\frac d 2}(\xi_j)\big)}
= O\bigl(\frac1{p^{3}}  \bigr),
\]
and
\[
|\nabla \bar W_{\boldsymbol{\mathrm\alpha},p}|\sim \frac1p,\quad \text{in}\; \Omega
\setminus \Cup_{j=1}^k B_{\frac d 2}(\xi_j),
\]
we can prove that
\begin{equation*}
\begin{split}
&\int_{\Omega} (\bar W_{\boldsymbol{\mathrm\alpha},p}+\omega)^p\eta_j\frac{
\partial (\bar W_{\boldsymbol{\mathrm\alpha},p}+\omega)}{\partial y_{h}}\\
=&
\int_{
 B_{d+\theta}(\xi_j)} (\bar W_{\boldsymbol{\mathrm\alpha},p}+\omega)^p\frac{ \partial (\bar W_{\boldsymbol{\mathrm\alpha},p}+\omega)}{\partial y_{h}}+
 O\Bigl(  \frac1{p^p} \Bigr)
\\
=&\frac1{p+1}\int_{
\partial B_{d+\theta}(\xi_j)} (\bar W_{\boldsymbol{\mathrm\alpha},p}+\omega)^{
p+1}+O\Bigl(  \frac1{p^p} \Bigr)
=O\Bigl(  \frac1{p^p}
\Bigr),
\end{split}
\end{equation*}
where $\eta_{j}$ and $\theta$ are defined in \eqref{def:Eta-theta}.
On the other hand,
\begin{equation}\label{51-9-9}
\begin{split}
&\bigl\langle \nabla (\bar W_{\boldsymbol{\mathrm\alpha},p}+\omega), \nabla [\eta_j\frac{
\partial (\bar W_{\boldsymbol{\mathrm\alpha},p}+\omega)}{\partial y_{h}}] \bigr\rangle\\
=& \int_{\Omega} \Bigl(\eta_j\nabla (\bar W_{\boldsymbol{\mathrm\alpha},p}+\omega)\nabla \frac{
\partial (\bar W_{\boldsymbol{\mathrm\alpha},p}+\omega)}{\partial y_{h}}+\frac{
\partial (\bar W_{\boldsymbol{\mathrm\alpha},p}+\omega)}{\partial y_{h}}
\nabla (\bar W_{\boldsymbol{\mathrm\alpha},p}+\omega)\nabla \eta_j \Bigr)\\
=&\int_{B_{d+\theta}(\xi_j)} \nabla (\bar W_{\boldsymbol{\mathrm\alpha},p}+\omega)\nabla \frac{
\partial (\bar W_{\boldsymbol{\mathrm\alpha},p}+\omega)}{\partial y_{h}}\\
&+ \int_{B_{d+\theta}(\xi_j)} \Bigl((\eta_j-1)\nabla \bar W_{\boldsymbol{\mathrm\alpha},p}\nabla \frac{
\partial \bar W_{\boldsymbol{\mathrm\alpha},p}}{\partial y_{h}}+\frac{
\partial \bar W_{\boldsymbol{\mathrm\alpha},p}}{\partial y_{h}}
\nabla \bar W_{\boldsymbol{\mathrm\alpha},p}\nabla \eta_j \Bigr)
+O\Bigl(\frac1{\theta p^{4}}
\Bigr).
\end{split}
\end{equation}   But
\begin{equation}\label{50-9-9}
\begin{split}
&\int_{B_{d+\theta}(\xi_j)} \nabla (\bar W_{\boldsymbol{\mathrm\alpha},p}+\omega)\nabla \frac{
\partial (\bar W_{\boldsymbol{\mathrm\alpha},p}+\omega)}{\partial y_{h}}\\
=&\frac12 \int_{\partial B_{d+\theta}(\xi_j)} |\nabla (\bar W_{\boldsymbol{\mathrm\alpha},p}+\omega)|^2 \nu_h
= \frac12 \int_{\partial B_{d+\theta}(\xi_j)} |\nabla \bar W_{\boldsymbol{\mathrm\alpha},p}|^2 \nu_h+
O\Bigl(\frac1{p^{4}}
\Bigr).
\end{split}
\end{equation}
Combining \eqref{51-9-9} and \eqref{50-9-9}, we obtain
\begin{equation*}
\begin{split}
\bigl\langle \nabla (\bar W_{\boldsymbol{\mathrm\alpha},p}+\omega), \nabla [\eta_j\frac{
\partial (\bar W_{\boldsymbol{\mathrm\alpha},p}+\omega)}{\partial y_{h}} ]\bigr\rangle
=\bigl\langle \nabla \bar W_{\boldsymbol{\mathrm\alpha},p}, \nabla [\eta_j\frac{
\partial (\bar W_{\boldsymbol{\mathrm\alpha},p}}{\partial y_{h}}] \bigr\rangle
+O\Bigl(\frac1{\theta p^{4}}
\Bigr).
\end{split}
\end{equation*}
Thus, we have proved that
\begin{equation}\label{56-9-9}
\begin{split}
A_{2,j}=&\bigl\langle I_p' (\bar W_{\boldsymbol{\mathrm\alpha},p}+\omega),  \eta_j\frac{
\partial (\bar W_{\boldsymbol{\mathrm\alpha},p}+\omega)}{\partial y_{h}} \bigr\rangle
=\bigl\langle I_p' (\bar W_{\boldsymbol{\mathrm\alpha},p}),  \eta_j\frac{
\partial \bar W_{\boldsymbol{\mathrm\alpha},p}}{\partial y_{h}} \bigr\rangle
+O\Bigl(\frac1{\theta p^{4}}
\Bigr)=
B_{2,j}+O\Bigl(\frac1{\theta p^{4}}\Bigr).
\end{split}
\end{equation}
Since $tA_{2}+(1-t)B_{2}=0$, then \eqref{56-9-9} implies that
\begin{equation}\label{20-7-9BIS}
B_{2,j}=\bigl\langle I_p' (\bar W_{\boldsymbol{\mathrm\alpha},p}),  \eta_j\frac{
\partial \bar W_{\boldsymbol{\mathrm\alpha},p}}{\partial y_{h}} \bigr\rangle=
O\bigl(\frac1{\theta p^{4}}\bigr).
\end{equation}
Using Lemma~\ref{l3-8-9}, we obtain
\begin{equation*}
\begin{split}
\frac{\partial \Psi_k(\xi_1,
\cdots,\xi_k)}{\partial \xi_{jh}} =O\Bigl(\theta+\frac1{p}+\sum_{j=1}^k|\alpha_j-1|p\Bigr).
\end{split}
\end{equation*}
By the non-degeneracy of the critical point, we find
\[
|\xi_j-x_{\infty,j}|=o\big(1\big).
\]

\vskip 0.2cm

\noindent{\sl Step 4. Conclusion.}

\vskip 0.2cm

 {\sl Step 1}, {\sl Step 2} and {\sl Step 3} imply that $(\boldsymbol{\mathrm\alpha}, \boldsymbol{\mathrm\xi}, \omega)$ is in the interior of $S_{p}$, thus giving a contradiction. This concludes the proof of Proposition~\ref{p20-7-9}.
\end{proof}

\

\

Now we can complete the proof for  the degree counting formula.
\begin{proof}[{\bf Proof of Theorem \ref{th1-9-9}}]  It follows from Proposition~\ref{p20-7-9} that
\[
\deg(A,0,S_{p}) =\deg(B,0,S_{p}),
\]
where $B$ is the map in \eqref{def:B}. Hence using \eqref{BBarB-SameDeg}, we get
\begin{equation}\label{AbarB}
\deg(A,0,S_{p}) =\deg ( \bar B, 0, D_{p}),
\end{equation}
where the map $\bar B(\boldsymbol{\mathrm\alpha},\boldsymbol{\mathrm \xi})$ for
$(\boldsymbol{\mathrm\alpha},\boldsymbol{\mathrm \xi})\in D_{p}$ is the one in \eqref{defbarB}-\eqref{def:Dp}, namely
\[\bar B(\boldsymbol{\mathrm\alpha},\boldsymbol{\mathrm \xi})=\left(\int_\Omega \nabla \bar W_{\boldsymbol{\mathrm\alpha},p} \nabla  P\bar W_{p, j}-\int_\Omega (\bar W_{\boldsymbol{\mathrm\alpha},p})^{p}P\bar W_{p, j}, \int_\Omega \nabla \bar W_{\boldsymbol{\mathrm\alpha},p} \nabla  \Big(\eta_j\frac{
\partial \bar W_{\boldsymbol{\mathrm\alpha},p}}{\partial y_{h}}\Big)-\int_\Omega (\bar W_{\boldsymbol{\mathrm\alpha},p})^{p}\eta_j\frac{
\partial \bar W_{\boldsymbol{\mathrm\alpha},p}}{\partial y_{h}} \right).\]
\\
Noting that
\[
\frac{8\pi }{ p^{\frac p{p-1}}\bar\e_{p,1}^{\frac2{p-1}}}=\frac{\sqrt e}{p}\bigl(
1+o(1)\bigr),
\]
by Lemma~\ref{l1-8-9} and Lemma \ref{l3-8-9},  we can further deform $\bar B$
in $S_{p}$ to
\[
\widetilde B:=\Bigl( \bigl( (\alpha_1-\alpha_1^p)\frac{8\pi \sqrt e }{ p},\cdots, (\alpha_k-\alpha_k^p)\frac{8\pi \sqrt e }{ p}\bigr),\;   \frac{64\pi^2 e}{p^2}\nabla \Psi_k(\xi_1,
\cdots,\xi_k)\Bigr),
\]
so that
\begin{equation}\label{barBtildeB}\deg ( \bar B, 0, D_{p})=\deg(\widetilde B, 0, D_{p}).\end{equation}
It is obvious that
\begin{equation}\label{obviousdeg}
\deg(\widetilde B, 0, D_{p})=(-1)^{k+m(\boldsymbol{\mathrm x}_{\infty},\Psi_k)}.
\end{equation}
Hence \eqref{AbarB},
\eqref{barBtildeB} and \eqref{obviousdeg}  give the conclusion.
\end{proof}

\

\section{The proof of Theorem \ref{theorem:local-uniqueness}}\label{section:localUniqueness}$\,$\\
Using Theorem~\ref{th1.1} and  Theorem \ref{theorem:topological-degree-formula}, we can finally prove the
local  uniqueness result.
\begin{proof}[{\bf Proof of Theorem \ref{theorem:local-uniqueness}}]
From \cite[Theorem 1.1]{GILY2021}, we know that any $k$-spike solution concentrating at $\boldsymbol{\mathrm x}_{\infty}$ is
non-degenerate. So the number of $k$-spike solutions concentrating at $\boldsymbol{\mathrm x}_{\infty}$
is finite. Suppose this number is $q$.
\vskip 0.1cm

We also know from Theorem \ref{th1.1} that each $k$-spike solution concentrating at $\boldsymbol{\mathrm x}_{\infty}$
has Morse index $k+m(\boldsymbol{\mathrm x}_{\infty},\Psi_k)$. Thus its degree is
$(-1)^{k+m(\boldsymbol{\mathrm x}_{\infty},\Psi_k)}$. This gives that the total
degree of all the $k$-spike solutions concentrating at $\boldsymbol{\mathrm x}_{\infty}$ is
$q(-1)^{k+m(\boldsymbol{\mathrm x}_{\infty},\Psi_k)}$.

\vskip 0.1cm
On the other hand, Theorem~\ref{theorem:topological-degree-formula} shows that
the total
degree of all the $k$-spike solutions concentrating at $\boldsymbol{\mathrm x}_{\infty}$ is
$(-1)^{k+m(\boldsymbol{\mathrm x}_{\infty},\Psi_k)}$. Thus, $q=1$.
\end{proof}

\

\appendix
\section{The derivatives of the Kirchhoff-Routh function}
For the reader convenience we collect here the expression of the  derivatives of the Kirchhoff-Routh function $\Psi_{k}:\Omega^{k}(\subset\mathbb R^{2k})\rightarrow \mathbb R$ defined for $\boldsymbol{a}=(a_1,\cdots, a_k)$ with $a_{j}=({a_j}_{1}, {a_j}_{2})\in\Omega\ (\subset\mathbb R^{2}$), for $j=1,\cdots, k$, as
\[
\Psi_{k}(\boldsymbol{a}):= \sum^k_{j=1} \Psi_{k,j}(\boldsymbol{a}),\quad~\mbox{ with }~\Psi_{k,j}(\boldsymbol{a}):=  R\big(a_j\big)- \sum^{k}_{m=1,m\neq j} G\big(a_j,a_m\big).
\]
We introduce the notation
$y_{2j-2+i}:={a_j}_{i}$
for $j=1,\cdots, k$, $i=1,2$, hence
 \[\boldsymbol{y}=(y_{1},\cdots, y_{2k})=\boldsymbol a\in\mathbb R^{2k}.\]
Easy computations show that
\begin{equation*}
\frac{\partial \Psi_{k}(\boldsymbol{y})}{\partial y_{2j-2+i}}=
\frac{\partial R(a_{j})}{\partial x_{i}}
-2\sum_{m=1,m\neq j}^{k}\partial_{i}G(a_{j},a_{m})
\end{equation*}
and
\begin{equation}\label{SecondDer}
\frac{\partial^{2}\Psi_{k}(\boldsymbol{y})}{\partial y_{2j-2+i}\ \partial y_{2h-2+q}}=
\Big[\frac{\partial^{2} R(a_{j})}{\partial x_{i}\partial x_{q}}
-2\sum_{m=1,m\neq j}^{k}
\partial_{iq}^{2}G(a_{j},a_{m})\Big]\delta_{hj}-2
\partial_{i}D_{q} G(a_{j},a_{h})
(1-\delta_{hj})
\end{equation}
for $j, h=1,\cdots, k$ and $i,q=1,2$, where $\delta_{hj}$ denotes the Kronecker's symbol.

\section{A linear problem in $\mathbb R^2$}

Consider
\begin{equation}\label{1-30-8}
-\Delta u = f(x),\quad \text{in}\; \mathbb R^2.
\end{equation}
We assume that $f$ satisfies the following condition
\begin{equation*}
|f(x)|\le \frac C{1+  |x|^p},~~~\mbox{for some $p>3$ and $C>0$}.
\end{equation*}

\begin{Prop}\label{p1-30-8}
Let $u$ be a solution of problem \eqref{1-30-8}, satisfying that as $|x|\to +\infty$,
\[
u(x) = A \log |x| + B + O\Big(\frac1{|x|}\Big),
\]
for some constants $A>0$ and $B$, then
\[
u(x)=\frac1{2\pi} \int_{\mathbb R^2} \log\frac1{|y-x|} f(y)\,dy + B,~~
\mbox{if $\displaystyle\int_{\mathbb R^2}  f(y)\,dy>0$},
\]
 while
\[
u(x)=\frac1{2\pi} \int_{\mathbb R^2} \log\frac1{|y-x|} f(y)\,dy - B,~~
\mbox{if $\displaystyle\int_{\mathbb R^2}  f(y)\,dy<0$}.
\]
\end{Prop}

\begin{proof}
Since $u(x)-\frac1{2\pi} \displaystyle\int_{\mathbb R^2} \log\frac1{|y-x|} f(y)\,dy$
 is harmonic, whose growth at infinity is bounded by $\log |x|$, $u$ can be represented as
\[
u(x)=\frac1{2\pi} \int_{\mathbb R^2} \log\frac1{|y-x|} f(y)\,dy + C,
\]
for some constant $C$.

\vskip 0.1cm

On the other hand, for $|x|$ large,
\[
\begin{split}
&\frac1{2\pi} \int_{\mathbb R^2} \log \frac1{|y-x|} f(y)\,dy\\
=&\frac1{2\pi} \int_{B_{\delta |x|}(0)}\log \frac1{|y-x|} f(y)\,dy+\frac1{2\pi} \int_{\mathbb R^2\setminus B_{\delta |x|}(0)} \log\frac1{|y-x|} f(y)\,dy\\
=&\frac1{2\pi} \int_{B_{\delta |x|}(0)} \Bigl(\log \frac1{|x|}+O\bigl(\frac{|y|}
{|x|}\bigr)\Bigr) f(y)\,dy+O\Bigl( \frac1{|x|}\int_{\mathbb R^2\setminus B_{\delta |x|}(0)} |\log\frac1{|y-x|}| \frac1{1+|y|^{p-1}}\,dy
\Bigr)\\
=&\frac1{2\pi|x|} \int_{\mathbb R^2} f(y)\,dy+O\Bigl(\frac1{|x|}\Bigr).
\end{split}
\]
Thus, we see that $C=B$, if $\displaystyle\int_{\mathbb R^2}  f(y)\,dy>0$, while $C=-B$
if $\displaystyle\int_{\mathbb R^2}  f(y)\,dy<0$.
\end{proof}

\noindent\textbf{Acknowledgments} ~Peng Luo and Shusen Yan were supported by National Key R\&D Program (No. 2023YFA1010002). Isabella Ianni was partially supported by the MUR-PRIN-20227HX33Z ``Pattern formation in nonlinear
phenomena'' and the INDAM-GNAMPA project ``Propriet$\grave{a}$ qualitative delle soluzioni di equazioni ellittiche''.
Peng Luo was supported by NSFC grants (No. 12422106). Shusen Yan was supported by NSFC grants (No. 12171184).

\renewcommand\refname{References}
\renewenvironment{thebibliography}[1]{%
\section*{\refname}
\list{{\arabic{enumi}}}{\def\makelabel##1{\hss{##1}}\topsep=0mm
\parsep=0mm
\partopsep=0mm\itemsep=0mm
\labelsep=1ex\itemindent=0mm
\settowidth\labelwidth{\small[#1]}%
\leftmargin\labelwidth \advance\leftmargin\labelsep
\advance\leftmargin -\itemindent
\usecounter{enumi}}\small
\def\newblock{\ }
\sloppy\clubpenalty4000\widowpenalty4000
\sfcode`\.=1000\relax}{\endlist}
\bibliographystyle{model1b-num-names}

\end{document}